\documentclass[12pt]{amsart}
\usepackage{latexsym,amsmath,amssymb,epsfig,amsthm}
\usepackage{graphicx}
\usepackage{tikz}
\usepackage[enableskew,vcentermath]{youngtab}
\usepackage{hyperref}
\hypersetup{colorlinks=false}
\input epsf
\newtheorem{theorem}{Theorem}[section]
\newtheorem{proposition}[theorem]{Proposition}
\newtheorem{corollary}[theorem]{Corollary}
\newtheorem{conjecture}[theorem]{Conjecture}
\newtheorem{definition}[theorem]{Definition}
\newtheorem{example}[theorem]{Example}

\newtheorem{remark}[theorem]{Remark}
\newtheorem{problem}[theorem]{Problem}
\newtheorem{question}[theorem]{Question}
\newcommand{\asc}{{\rm asc}}
\newcommand{\Asc}{{\rm Asc}}
\newcommand{\ades}{{\rm ades}}
\newcommand{\aDes}{{\rm aDes}}
\newcommand{\cat}{{\rm Cat}}
\newcommand{\ch}{{\rm ch}}
\newcommand{\cox}{{\rm Cox}}
\newcommand{\cros}{{\rm cros}}
\newcommand{\des}{{\rm des}}
\newcommand{\Des}{{\rm Des}}
\newcommand{\DEX}{{\rm DEX}}
\newcommand{\esd}{{\rm esd}}

\newcommand{\exc}{{\rm exc}}
\newcommand{\Exc}{{\rm Exc}}
\newcommand{\fexc}{{\rm fexc}}
\newcommand{\fix}{{\rm fix}}
\newcommand{\inv}{{\rm inv}}
\newcommand{\link}{{\rm link}}
\newcommand{\maj}{{\rm maj}}
\newcommand{\nc}{{\rm NC}}
\newcommand{\nest}{{\rm nest}}

\newcommand{\sd}{{\rm sd}}

\newcommand{\stat}{{\rm stat}}
\newcommand{\Stab}{{\rm Stab}}
\newcommand{\Sum}{{\rm sum}}
\newcommand{\SYB}{{\rm SYB}}
\newcommand{\SYT}{{\rm SYT}}
\newcommand{\wexc}{{\rm wexc}}
\newcommand{\Wexc}{{\rm Wexc}}
\newcommand{\aA}{{\mathcal A}}
\newcommand{\bB}{{\mathcal B}}
\newcommand{\cC}{{\mathcal C}}
\newcommand{\dD}{{\mathcal D}}

\newcommand{\iI}{{\mathcal I}}
\newcommand{\lL}{{\mathcal L}}
\newcommand{\pP}{{\mathcal P}}
\newcommand{\qQ}{{\mathcal Q}}
\newcommand{\CC}{{\mathbb C}}
\newcommand{\RR}{{\mathbb R}}
\newcommand{\sS}{{\mathcal S}}
\newcommand{\tT}{{\mathcal T}}
\newcommand{\wW}{{\mathcal W}}
\newcommand{\NN}{{\mathbb N}}
\newcommand{\ZZ}{{\mathbb Z}}
\newcommand{\bx}{{\mathbf x}}
\newcommand{\by}{{\mathbf y}}
\newcommand{\bz}{{\mathbf z}}
\newcommand{\fS}{{\mathfrak S}}

\renewcommand{\to}{\rightarrow}

\newcommand{\sm}{{\smallsetminus}}
\begin{document}
\title[Gamma-positivity in combinatorics and geometry]
{Gamma-positivity in combinatorics and geometry}

\author{Christos~A.~Athanasiadis}

\address{Department of Mathematics \\
National and Kapodistrian University of Athens\\
Panepistimioupolis\\
15784 Athens, Greece}
\email{caath@math.uoa.gr}

\date{}

\begin{abstract}
Gamma-positivity is an elementary property that 
polynomials with symmetric coefficients may have, 
which directly implies their unimodality. The 
idea behind it stems from work of Foata, 
Sch\"utzenberger and Strehl on the Eulerian 
polynomials; it was revived independently by 
Br\"and\'en and Gal in the course of their study 
of poset Eulerian polynomials and face 
enumeration of flag simplicial spheres, respectively, 
and has found numerous applications since then. 
This paper surveys some of the main results and 
open problems on gamma-positivity, appearing in 
various combinatorial or geometric contexts, as 
well as some of the diverse methods that have 
been used to prove it.  
\end{abstract}

\maketitle

\tableofcontents

\section{Introduction}
\label{sec:intro}

The unimodality of polynomials is a major theme 
which has occupied mathematicians for the past 
few decades. At least three surveys
\cite{Bra15, Bre94b, Sta89} about unimodal, 
log-concave and real-rooted polynomials, showing 
the enormous variety of methods which are 
available to prove these properties, have been 
written. This survey article focuses on a related
elementary property, that of $\gamma$-positivity,
which directly implies symmetry and unimodality 
and has provided a new exciting approach to this 
topic. Gamma-positivity appears surprisingly often 
in combinatorial and geometric contexts; this 
article aims to discuss some of the main examples, 
results, methods and open problems around it. 

This introductory section provides basic 
definitions and related comments, as well as an 
outline of the remainder of the article. We recall 
that a polynomial $f(x) = \sum_i a_i x^i \in \RR[x]$
is called

\begin{itemize}
\item[$\bullet$] \emph{symmetric}, with center of
  symmetry $n/2$, if $a_i = a_{n-i}$ for all $i 
  \in \ZZ$ (where $a_i = 0$ for negative values 
  of $i$),
\item[$\bullet$] \emph{unimodal}, if $0 \le a_0 
  \le a_1 \le \cdots \le a_k \ge a_{k+1} \ge 
  \cdots$ for some $k \in \NN$,
\item[$\bullet$] \emph{$\gamma$-positive}, if 
  \begin{equation} \label{eq:def-gamma}
    f(x) \ = \ \sum_{i=0}^{\lfloor n/2 \rfloor} 
    \gamma_i x^i (1+x)^{n-2i}
  \end{equation}
  for some $n \in \NN$ and nonnegative reals 
  $\gamma_0, \gamma_1,\dots,\gamma_{\lfloor n/2 
  \rfloor}$, and
\item[$\bullet$] \emph{real-rooted}, if every 
  root of $f(x)$ is real, or else $f(x) = 0$.
  \end{itemize}

The notion of $\gamma$-positivity appeared 
first in the work of D.~Foata and 
M.~Sch\"utzenberger \cite{FSc70} and
subsequently of D.~Foata and 
V.~Strehl~\cite{FS74, FS76} on the classical 
Eulerian polynomials, discussed in detail in 
Sections~\ref{subsubsec:Euler} 
and~\ref{subsec:hopping}. After having implicitly 
reappeared in the theory of enriched poset 
partitions of J.~Stembridge~\cite{Ste97} (see 
also Section~\ref{subsec:enriched}),
it was brought again to light independently by 
P.~Br\"and\'en \cite{Bra04, Bra08} and 
\'S.~Gal \cite{Ga05} in the course of their study 
of poset Eulerian polynomials and face 
enumeration of flag triangulations of spheres, 
respectively (see  
Sections~\ref{subsubsec:posetEuler} 
and~\ref{subsec:gal}). These works made it clear 
that $\gamma$-positivity is a concept of independent 
interest which provides a powerful approach to 
the problem of unimodality for symmetric 
polynomials. Symmetry and real-rootedness of a 
polynomial $f(x) \in \RR_{\ge 0}[x]$ implies its 
$\gamma$-positivity \cite[Lemma~4.1]{Bra04} 
\cite[Remark~3.1.1]{Ga05}. On the other hand, 
every $\gamma$-positive polynomial (whether 
real-rooted or not) is symmetric and unimodal, 
as a sum of symmetric and unimodal polynomials 
with a common center of symmetry (a fact 
already implicit in \cite[p.~136]{Ga98}). Thus,
$\gamma$-positivity can be applied to general 
situations and may lead to a more elementary 
proof of the 
unimodality of $f(x)$, even when the latter is 
real-rooted. Moreover, an explicit expression 
of the form (\ref{eq:def-gamma}) gives additional
information about $f(x)$ (for instance, it 
implies a formula for $f(-1)$ and predicts its 
sign) and the problem to interpret algebraically 
or combinatorially the coefficients $\gamma_i$ 
is often of independent interest.

The present article is an expanded version of the 
author's lectures at the \emph{77th S\'eminaire 
Lotharingien de Combinatoire} (September 2016,
Strobl, Austria). Section~\ref{sec:comb} discusses 
the variety of examples of $\gamma$-positive 
polynomials in combinatorics, beginning with the 
prototypical example of Eulerian polynomials,
together with some general results. Many of these 
examples come up again in the geometric contexts 
of Section~\ref{sec:geom}. That section centers 
around Gal's conjecture, claiming that 
$h$-polynomials of flag simplicial homology 
spheres are $\gamma$-positive. This conjecture 
and its relatives provide a general framework 
under which a lot of the seemingly unrelated 
$\gamma$-positivity phenomena of 
Section~\ref{sec:comb} can be considered, partially 
explains their abundance and allows for tools from 
geometric combinatorics to be applied to their study.
Section~\ref{sec:methods} discusses the plethora 
of methods used in the literature to prove 
$\gamma$-positivity. Section~\ref{sec:gen} is 
devoted to generalizations and variations of the 
concept of $\gamma$-positivity, giving an emphasis
to equivariant and symmetric function generalizations. 
The choice of topics is strongly affected by the 
author's knowledge and personal taste; in particular, 
possible probabilistic aspects of $\gamma$-positivity 
are not treated. Some previously unpublished 
statements and open problems are included. Other 
expositions on $\gamma$-positivity can be found 
in \cite[Section~3]{Bra15} \cite[Chapter~4]{Pet15}.

\medskip
\noindent
\textbf{Notation}.
For integers $a \le b$ we set $[a, b] := \{a, 
a+1,\dots,b\}$ and use the abbreviation $[n] := [1, 
n]$ for $n \in \NN$. We will denote by $|\Omega|$ the 
cardinality, and by $2^\Omega$ the set of all subsets, 
of a finite set $\Omega$. For $\Omega \subseteq \ZZ$, 
we will also denote by $\Stab(\Omega)$ the set of 
all subsets of $\Omega$ which do not contain two 
consecutive integers.

\section{Gamma-positivity in combinatorics}
\label{sec:comb}

This section describes instances of 
$\gamma$-positivity in combinatorics. Much of 
the motivation comes from the study of Eulerian 
polynomials which, along with generalizations
and variations, are discussed in detail. For 
any undefined notation or terminology, we refer 
to Stanley's textbooks \cite{StaEC2, StaEC1}.

\subsection{Variations of Eulerian polynomials}
\label{subsec:Euler}

\subsubsection{Eulerian polynomials}
\label{subsubsec:Euler}

One of the most important polynomials in combinatorics
is the \emph{Eulerian polynomial}, defined by any of 
the equivalent formulas 
\begin{equation}
\label{eq:defAn}
A_n(x) \ := \ \sum_{w \in \fS_n} x^{\asc(w)} \ = \ 
              \sum_{w \in \fS_n} x^{\des(w)} \ = \ 
              \sum_{w \in \fS_n} x^{\exc(w)} \ = \ 
              \sum_{w \in \fS_n} x^{\wexc(w) - 1} 
\end{equation}
for every positive integer $n$. Here $\asc(w)$, 
$\des(w)$, $\exc(w)$ and $\wexc(w)$ denotes the 
cardinality of the set $\Asc(w)$ of \emph{ascents}, 
$\Des(w)$ of \emph{descents}, $\Exc(w)$ of 
\emph{excedances} and $\Wexc(w)$ of \emph{weak 
excedances}, respectively, of the permutation $w 
\in \fS_n$, defined by 

\begin{itemize}
\item[$\bullet$] $\Asc(w) := \{ i \in [n-1]: w(i) 
                                     < w(i+1) \}$, 
\item[$\bullet$] $\Des(w) := \{ i \in [n-1]: w(i) 
                                     > w(i+1) \}$, 
\item[$\bullet$] $\Exc(w) := \{ i \in [n-1]: w(i) 
                                    > i \}$,
\item[$\bullet$] $\Wexc(w) := \{ i \in [n]: w(i) 
                                       \ge i \}$.
  \end{itemize}
For the first few values of $n$, we have
\[  A_n(x) \ = \ \begin{cases}
    1, & \text{if \ $n=1$} \\
    1 + x, & \text{if \ $n=2$} \\
    1 + 4x + x^2, & \text{if \ $n=3$} \\
    1 + 11x + 11x^2 + x^3, & \text{if \ $n=4$} \\
    1 + 26x + 66x^2 + 26x^3 + x^4, & 
                           \text{if \ $n=5$} \\
    1 + 57x + 302x^2 + 302x^3 + 57x^4 + x^5, & 
                           \text{if \ $n=6$} \\
    1 + 120x + 1191x^2 + 2416x^3 + 1191x^4 
         + 120x^5 + x^6, & \text{if \ $n=7$}.  
                \end{cases} \]

The Eulerian polynomial $A_n(x)$, which is clearly 
symmetric, provides the prototypical example of a 
$\gamma$-positive polynomial in combinatorics. The 
corresponding $\gamma$-coefficients for the first 
few values of $n$ are determined from the 
expressions 
\[  A_n(x) \ = \ \begin{cases}
    1 + x, & \text{if \ $n=2$} \\
    (1 + x)^2 + 2x, & \text{if \ $n=3$} \\
    (1 + x)^3 + 8x(1 + x), & \text{if \ $n=4$} \\
    (1 + x)^4 + 22x(1 + x)^2 + 16x^2, & 
                           \text{if \ $n=5$} \\
    (1 + x)^5 + 52x(1 + x)^3 + 136x^2(1 + x), & 
                           \text{if \ $n=6$} \\
    (1 + x)^6 + 114x(1 + x)^4 + 720x^2(1 + x)^2 + 
                    272x^3, & \text{if \ $n=7$}.  
                \end{cases} \]

The $\gamma$-positivity of $A_n(x)$ (which, of 
course, follows from the well known fact that 
$A_n(x)$ is real-rooted for all $n$) was first 
shown combinatorially by Foata and 
Sch\"utzenberger~\cite[Theorem~5.6]{FSc70}. An 
explicit combinatorial interpretation of the  
corresponding $\gamma$-coefficients follows from
the results of~\cite{FS76}. Recall that a
\emph{double excedance} of $w \in \fS_n$ is any 
index $1 \le i \le n$ such that $w(i) > i > 
w^{-1}(i)$ and that $w \in \fS_n$ is called an 
\emph{up-down} permutation, if $\Asc(w) = \{1, 
3, 5,\dots\} \cap [n-1]$. The first four, as well 
as the last two, interpretations of $\gamma_{n,i}$ 
in the following fundamental result are easily 
shown to be equivalent 
to one another; the fifth one follows from the 
fourth and the bijection (one of the fundamental 
transformations of Foata and 
Sch\"utzenberger~\cite{FSc70}) of 
\cite[Proposition~1.3.1]{StaEC1}. An additional 
interpretation, in terms of increasing binary 
trees, is discussed in \cite[p.~136]{Ga98}.
\begin{theorem} \label{thm:FSSAn}
{\rm (cf. Foata--Strehl~\cite{FS76})}
For all $n \ge 1$,
\begin{equation}
\label{eq:FSSAn}
A_n(x) \ = \ \sum_{i=0}^{\lfloor (n-1)/2 \rfloor} 
\gamma_{n,i} x^i (1+x)^{n-1-2i},
\end{equation} 
where $\gamma_{n,i}$ is equal to each of the 
following:

\begin{itemize}
\item[$\bullet$] the number of $w \in \fS_n$ for 
which $\Asc(w) \in \Stab([n-2])$ has $i$ elements,
\item[$\bullet$] the number of $w \in \fS_n$ for 
which $\Asc(w) \in \Stab([2,n-1])$ has $i$ elements,
\item[$\bullet$] the number of $w \in \fS_n$ for 
which $\Des(w) \in \Stab([n-2])$ has $i$ elements,
\item[$\bullet$] the number of $w \in \fS_n$ for 
which $\Des(w) \in \Stab([2,n-1])$ has $i$ elements,
\item[$\bullet$] the number of $w \in \fS_n$ with 
$i$ excedances and no double excedance, for which the 
smallest of the maximum elements of the cycles of $w$ 
is a fixed point,
\item[$\bullet$] the number of $w \in \fS_n$ with $i$ 
excedances and no double excedance, for which the 
largest of the minimum elements of the cycles of $w$ 
is a fixed point.
  \end{itemize}
In particular, $A_n(x)$ is $\gamma$-positive and
\begin{equation}
\label{eq:x=-1An}
  A_n(-1) \ = \ \begin{cases}
  \, 0, & \text{if $n$ is even}, \\
  \, (-1)^{(n-1)/2} \, \gamma_{n,(n-1)/2}, & 
  \text{if $n$ is odd}  \end{cases}  
\end{equation}
for all $n$, where $\gamma_{n,(n-1)/2}$ is the 
number of up-down permutations in $\fS_n$. 
\end{theorem}

We will now describe some refinements and variations 
of this theorem (more related results appear in the 
sequel). The following two theorems refine the 
third interpretation of the $\gamma$-positivity of 
$A_n(x)$, given in Theorem~\ref{thm:FSSAn}. 
For $w \in \fS_n$, we denote by $\inv(w)$ the 
number of inversions (pairs $(i, j) \in [n] \times 
[n]$ such that $i < j$ and $w(i) > w(j)$) of $w$, 
and let
\begin{itemize}
\item[$\bullet$] (2-13)$w$ be the number of pairs 
$(i, j) \in [n-1] \times [n-1]$ such that $i < j$ 
and $w(j) < w(i) < w(j+1)$,
\item[$\bullet$] (31-2)$w$ be the number of pairs 
$(i, j) \in [n] \times [n]$ such that $i+1 < j$ 
and $w(i+1) < w(j) < w(i)$,
\item[$\bullet$] $\cros(w)$ be the number of pairs 
$(i, j) \in [n] \times [n]$ such that $i < j \le w(i) 
< w(j)$, or $w(i) < w(j) < i < j$,
\item[$\bullet$] $\nest(w)$ be the number of pairs 
$(i, j) \in [n] \times [n]$ such that $i < j \le w(j) 
< w(i)$, or $w(j) < w(i) < i < j$.
\end{itemize}
Equations~(\ref{eq:SZ1aAn})
and~(\ref{eq:SZ1bAn}) appear as the specialization
$u = v = w = 1$ of~\cite[Theorem~2]{SZ12} and  
as~\cite[Corollary~6]{SZ12}; the expansion 
(\ref{eq:SZ1bAn}) for the left-hand side of 
(\ref{eq:SZ1aAn}) was originally shown by Br\"and\'en
\cite[Section~5]{Bra08}. Equation~(\ref{eq:SZ2An})
is the main statement of~\cite[Theorem~1]{SZ16}; the 
positivity of the coefficient of $x^i(1+x)^{n-1-2i}$, 
as a polynomial in $q$, was conjectured 
in a preprint version of~\cite{BP14}. A related 
result appears in~\cite[Remark~3.4]{BP14}.
\begin{theorem} \label{thm:SZAn}
{\rm (Br\"and\'en~\cite{Bra08}, 
Shin--Zeng~\cite{SZ12, SZ16})}
For all $n \ge 1$,

\begin{align}
\label{eq:SZ1aAn}
\sum_{w \in \fS_n} p^{(2-13)w} q^{(31-2)w} x^{\des(w)} 
& \ = \ \sum_{w \in \fS_n} p^{\nest(w)} q^{\cros(w)}
x^{\wexc(w) - 1} \\
& \ = \ \sum_{i=0}^{\lfloor (n-1)/2 \rfloor} 
        a_{n,i} (p,q) \, x^i (1+x)^{n-1-2i} 
\label{eq:SZ1bAn}
\end{align}
and 
\begin{equation}
\label{eq:SZ2An}
\sum_{w \in \fS_n} q^{\inv(w)-\exc(w)} x^{\exc(w)} 
\ = \ \sum_{i=0}^{\lfloor (n-1)/2 \rfloor} a_{n,i} 
                    (q^2, q) \, x^i (1+x)^{n-1-2i},
\end{equation}
where 
\begin{equation}
\label{eq:SZa}
a_{n,i} (p, q) \ = \ \sum_w p^{(2-13)w} q^{(31-2)w} 
\end{equation}
and the sum runs through all permutations $w \in 
\fS_n$ for which $\Des(w) \in \Stab([n-2])$ has 
$i$ elements.
\end{theorem}

The following theorem of J.~Shareshian and M.~Wachs, 
the proof of which is sketched in 
Section~\ref{subsec:sfmethods}, follows from the 
methods used to prove the specialization $p=1$ 
in~\cite[Theorem~4.4]{SW17}. This special case 
was stated in~\cite[Remark~5.5]{SW10} without 
giving an explicit interpretation for the coefficients 
$\gamma_i(1,q)$; it is also implicit in~\cite{LSW12} 
(see Equations~(1.4) and~(6.1) there) and is reproven
by different methods in~\cite{LZ15}. For $w \in \fS_n$,
we set
\[  \des^*(w) \ := \ \begin{cases}
    \des(w), & \text{if \ $w(1)=1$} \\
    \des(w) - 1, & \text{if \ $w(1) > 1$}  
                \end{cases} \]
and denote by $\maj(w)$ the major index (sum of the
elements of $\Des(w)$) of $w$.
\begin{theorem} \label{thm:SWAn}
{\rm (cf. \cite[Section~4]{SW17})}
For all $n \ge 1$
\begin{equation}
\label{eq:SWAn}
\sum_{w \in \fS_n} p^{\des^*(w)} q^{\maj(w)-\exc(w)} 
x^{\exc(w)} \ = \ 
\sum_{i=0}^{\lfloor (n-1)/2 \rfloor} \gamma_{n,i} 
        (p,q) \, x^i (1+x)^{n-1-2i},
\end{equation}
where 
\begin{equation}
\label{eq:SWgamman}
\gamma_{n,i} (p,q) \ = \ \sum_w p^{\des(w^{-1})} 
                         q^{\maj(w^{-1})}
\end{equation}
and the sum runs through all permutations $w \in 
\fS_n$ for which $\Des(w) \in \Stab([n-2])$ has 
$i$ elements.
\end{theorem}

For further generalizations of the 
$\gamma$-positivity of $A_n(x)$ see, for instance,
\cite[Section~4]{Bra08} and \cite[Section~4]{NPT11}.

We close this section by discussing an interesting 
variant of $A_n(x)$, defined as 
\begin{equation}
\label{eq:defbinAn}
\widetilde{A}_n(x) \ := \ 1 \, + \, x \, 
\sum_{k=1}^n \binom{n}{k} A_k(x)
\end{equation}
and called a \emph{binomial Eulerian polynomial}. 
For the first few values of $n$, we have
\[  \widetilde{A}_n(x) \ = \ \begin{cases}
    1 + x, & \text{if \ $n=1$} \\
    1 + 3x + x^2, & \text{if \ $n=2$} \\
    1 + 7x + 7x^2 + x^3, & \text{if \ $n=3$} \\
    1 + 15x + 33x^2 + 15x^3 + x^4, & 
                           \text{if \ $n=4$} \\
    1 + 31x + 131x^2 + 131x^3 + 31x^4 + x^5, & 
                           \text{if \ $n=5$} \\
    1 + 63x + 473x^2 + 883x^3 + 473x^4 
         + 63x^5 + x^6, & \text{if \ $n=6$}.  
                \end{cases} \]

This polynomial first appeared in an 
enumerative-geometric context in
\cite[Section~10.4]{PRW08}, where it was shown to 
equal the $h$-polynomial of the $n$-dimensional 
stellohedron (see Section~\ref{sec:geom} for an 
explanation of these terms); its symmetry 
follows from this interpretation and was 
rediscovered in~\cite{CGK10}. 

The $\gamma$-positivity of $\widetilde{A}_n(x)$ 
follows from a more general theorem of A.~Postnikov,
V.~Reiner and L.~Williams~\cite[Theorem~11.6]{PRW08} 
(see also Theorem~\ref{thm:PRW} in the sequel); for 
the first few values of $n$, we have 
\[  \widetilde{A}_n(x) \ = \ \begin{cases}
    1 + x, & \text{if \ $n=1$} \\
    (1 + x)^2 + x, & \text{if \ $n=2$} \\
    (1 + x)^3 + 4x(1 + x), & \text{if \ $n=3$} \\
    (1 + x)^4 + 11x(1 + x)^2 + 5x^2, & 
                           \text{if \ $n=4$} \\
    (1 + x)^5 + 26x(1 + x)^3 + 43x^2(1 + x), & 
                           \text{if \ $n=5$} \\
    (1 + x)^6 + 57x(1 + x)^4 + 230x^2(1 + x)^2 + 
                    61x^3, & \text{if \ $n=6$}.  
                \end{cases} \]
We prefer to state a version due to 
Shareshian and Wachs, which is similar to 
Theorem~\ref{thm:FSSAn} and affords a $q$-analog, 
similar to that of Theorem~\ref{thm:SWAn} for 
$A_n(x)$.
\begin{theorem} \label{thm:SWbinAn}
{\rm (Shareshian--Wachs~\cite[Theorem~4.5]{SW17})}
For all $n \ge 1$
\begin{equation}
\label{eq:SWqbinAn}
1 \, + \, x \, \sum_{k=1}^n \binom{n}{k}_q \,
\sum_{w \in \fS_k} q^{\maj(w)-\exc(w)} x^{\exc(w)} 
\ = \ \sum_{i=0}^{\lfloor n/2 \rfloor} 
\widetilde{\gamma}_{n,i} (q) \, x^i (1+x)^{n-2i},
\end{equation}
where $\binom{n}{k}_q$ is a $q$-binomial coefficient,
\begin{equation}
\label{eq:SWqbingamman}
\widetilde{\gamma}_{n,i} (q) \ = \ 
\sum_w q^{\maj(w^{-1})} \ = \ \sum_w q^{\inv(w)}
\end{equation}
and the sums run through all permutations $w \in 
\fS_n$ for which $\Des(w) \in \Stab([n-1])$ has 
$i$ elements. In particular,
\begin{equation}
\label{eq:SWbinAn}
\widetilde{A}_n(x) \ = \ \sum_{i=0}^{\lfloor n/2 
\rfloor} \widetilde{\gamma}_{n,i} x^i (1+x)^{n-2i},
\end{equation}
where $\widetilde{\gamma}_{n,i}$ is equal to each 
of the following:

\begin{itemize}
\item[$\bullet$] the number of $w \in \fS_n$ for 
which $\Asc(w) \in \Stab([n-1])$ has $i$ elements,
\item[$\bullet$] the number of $w \in \fS_n$ for 
which $\Des(w) \in \Stab([n-1])$ has $i$ elements,
\item[$\bullet$] the number of $w \in \fS_n$ with 
$i$ excedances and no double excedance.
  \end{itemize}
Moreover,
\begin{equation}
\label{eq:x=-1binAn}
  \widetilde{A}_n(-1) \ = \ \begin{cases}
  \, 0, & \text{if $n$ is odd}, \\
  \, (-1)^{n/2} \, \widetilde{\gamma}_{n,n/2}, & 
  \text{if $n$ is even}  \end{cases}  
\end{equation}
where $\widetilde{\gamma}_{n,n/2}$ is the number 
of up-down permutations in $\fS_n$. 
\end{theorem}
For an alternative approach to the 
$\gamma$-positivity of $\widetilde{A}_n(x)$, see 
Remark~\ref{rem:dnBinAn}. 
\begin{problem}
Find a $p$-analog of Theorem~\ref{thm:SWbinAn},
similar to Theorem~\ref{thm:SWAn}. 
\end{problem}

\subsubsection{Poset Eulerian polynomials}
\label{subsubsec:posetEuler}

Given a partially ordered set (poset, for short) $\pP$ 
with $n$ elements, any bijective map $\omega: \pP \to 
[n]$ is called a \emph{labeling}. Let us write 
permutations $w \in \fS_n$ in one-line notation 
$(w(1), w(2),\dots,w(n))$.
\begin{definition} 
\label{def:posetEuler}
{\rm (Stanley~\cite{Sta72} 
\cite[Section~3.15.2]{StaEC1})}
Let $\omega: \pP \to [n]$ be a labeling of a poset
$\pP$. The $(\pP, \omega)$-Eulerian polynomial is 
defined as 
\[ A_{\pP, \omega} (x) \ = \ 
   \sum_{w \in \lL(\pP, \omega)} x^{\des(w)}, \]
where $\lL(\pP, \omega)$ is the set which consists of 
all permutations $(a_1, a_2,\dots,a_n) \in \fS_n$ 
such that $\omega^{-1} (a_i) <_\pP 
\omega^{-1} (a_j) \Rightarrow i < j$.  
\end{definition}

The polynomial $A_{\pP, \omega} (x)$ plays a major 
role in Stanley's theory of $(\pP,\omega)$-partitions 
\cite{Sta72} \cite[Section~3.15]{StaEC1}; it reduces
to the Eulerian polynomial $A_n(x)$ when $\pP$ is an
antichain on $n$ elements. For the labeled poset of
Figure~\ref{fg:poset} we have $\lL(\pP, \omega) = 
\{ (1,4,2,3)$, $(1,4,3,2)$, $(4,1,2,3)$, $(4,1,3,2)$, 
$(1,3,4,2) \}$ and $A_{\pP, \omega} (x) = 3x + 2x^2$. 
Examples where $A_{\pP, \omega} (x)$ is not real-rooted
were given by P.~Br\"and\'en~\cite{Bra04b} and 
J.~Stembridge~\cite{Ste07} (see also 
\cite[Section~6]{Bra15}), thus disproving
long-standing conjectures of J.~Neggers~\cite{Ne78} 
\cite[Conjecture~1]{Sta89} and 
R.~Stanley~\cite[Conjecture~1]{Bre89} 
\cite[Conjecture~3.9]{Bre94b}.

\begin{figure}
\begin{center}
\begin{tikzpicture}[scale=0.9]
\label{fg:poset}

   \draw(0,2) node(1){3};
   \draw(0,0) node(2){1};
   \draw(2,2) node(3){2};
   \draw(2,0) node(4){4};
   \draw(1) -- (2) -- (3) -- (4);

\end{tikzpicture}
\caption{A labeled poset with four elements}
\end{center}
\end{figure}
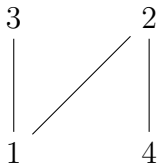

The polynomial $A_{\pP, \omega} (x)$ does not depend 
on $\omega$, when the latter is assumed to be order 
preserving (such labelings are called \emph{natural}),
and is thus denoted simply by $A_\pP (x)$. For example,
the labeling obtained from the one of 
Figure~1 by swapping 2 and 4 is natural 
and shows that for this poset $\aA_\pP (x) = 1 + 3x + 
x^2$. Moreover, as a consequence of the reciprocity
theorem~\cite[Theorem~3.15.10]{StaEC1} for 
$(\pP,\omega)$-partitions, $\aA_\pP(x)$ is symmetric 
if $\pP$ is graded; its unimodality in this case was 
first shown by V.~Reiner and V.~Welker~\cite{RW05}, 
whose proof relied on a deep result from algebraic 
geometry. Br\"and\'en \cite{Bra04} 
\cite[Section~6]{Bra08} gave two beautiful 
combinatorial proofs of the $\gamma$-positivity of 
$\aA_\pP(x)$ for the more general class of 
sign-graded posets.
\begin{theorem}
\label{thm:BrAP(x)}
{\rm (Br\"and\'en~\cite{Bra04})}
The polynomial $A_\pP (x)$ is $\gamma$-positive 
for every finite graded poset $\pP$.
\end{theorem}

A generalization to finite crystallographic root 
systems was given by Stembridge~\cite{Ste08}.

For a different $\gamma$-positivity result for 
posets, due to Stembridge~\cite{Ste97}, which 
generalizes the $\gamma$-positivity of $A_n(x)$, 
see Section~\ref{subsec:enriched}.

\subsubsection{Coxeter Eulerian polynomials}
\label{subsubsec:CoxeterEuler}

Let $(W, S)$ be a Coxeter system, with Coxeter 
length function $\ell_S: W \to \NN$ 
(see~\cite[Chapter~1]{BB05} for definitions). 
Assuming that $W$ is finite, the \emph{$W$-Eulerian 
polynomial} is defined as 
\begin{equation}
\label{eq:defW(x)}
W(x) \ := \ \sum_{w \in W} x^{\des(w)}, 
\end{equation}
where $\des(w)$ is the number of right descents 
(elements $s \in S$ such that $\ell_S (ws) < \ell_S
(w)$) of $w \in W$. The polynomial $W(x)$ was first 
studied systematically by F.~Brenti~\cite{Bre94a}. 
It has similar properties as $A_n(x)$, to which 
it reduces when $W$ is the symmetric group $\fS_n$; 
in particular, $W(x)$ is easily verified to be 
symmetric for all finite Coxeter groups $W$. The 
following statement combines Theorem~\ref{thm:FSSAn} 
with results of C.~Chow~\cite{Ch08} and 
J.~Stembridge~\cite{Ste08} (the proof that the 
$W$-Eulerian polynomials are real-rooted was 
completed more recently; see~\cite{SV15} and 
references therein).
\begin{theorem}
\label{thm:W(x)}
{\rm (\cite{Ch08} \cite[Theorem~1.2]{Ste08})} 
The $W$-Eulerian polynomial is $\gamma$-positive 
for every finite Coxeter group $W$.
\end{theorem}
\begin{problem}
Find a proof which does not use the classification of
finite Coxeter groups. 
\end{problem}

A common generalization of Theorems~\ref{thm:BrAP(x)}
and~\ref{thm:W(x)} is provided by 
\cite[Corollary~7.10]{Ste08}, mentioned earlier.

Just as is the case for the symmetric groups and the 
classical Eulerian polynomial, the corresponding
$\gamma$-coefficients admit interesting combinatorial 
interpretations for the other infinite families of 
finite Coxeter groups as well. We now describe such 
interpretations for the hyperoctahedral groups $\bB_n$.
Recall that $\bB_n$ consists of all permutations $w$ 
of the set $\Omega_n := \{1, -1, 2, -2,\dots,n, -n\}$
satisfying $w(-a) = -w(a)$ for each $a \in \Omega_n$. 
These can be viewed as signed permutations of length 
$n$ (see Section~\ref{subsubsec:colored} for the more 
general notion of $r$-colored permutation). The total
order 
\begin{equation}
\label{eq:def<r}
-1 <_r -2 <_r -3 <_r \cdots <_r 0 <_r 1 <_r 2 <_r 3 
<_r \cdots 
\end{equation}
of $\ZZ$ is convenient to use when $\bB_n$ is thought 
of as a colored permutation (rather than as a Coxeter) 
group. The $\bB_n$-Eulerian polynomial is given by
\begin{equation}
\label{eq:defBn}
B_n(x) \ = \ \sum_{w \in \bB_n} x^{\des^B(w)} \ = \ 
              \sum_{w \in \bB_n} x^{\des_B(w)},
\end{equation}
where
\begin{itemize}
\item[$\bullet$] $\des^B(w)$ is the number of 
indices $i \in \{0, 1,\dots,n-1\}$ such that $w(i) 
> w(i+1)$, 
\item[$\bullet$] $\des_B(w)$ is the number of 
indices $i \in \{0, 1,\dots,n-1\}$ such that $w(i) 
>_r w(i+1)$
  \end{itemize}
for $w \in \bB_n$, with $w(0) := 0$. For the first 
few values of $n$, we have
\[  B_n(x) \ = \ \begin{cases}
    1 + x, & \text{if \ $n=1$} \\
    1 + 6x + x^2, & \text{if \ $n=2$} \\
    1 + 23x + 23x^2 + x^3, & \text{if \ $n=3$} \\
    1 + 76x + 230x^2 + 76x^3 + x^4, & 
                           \text{if \ $n=4$} \\
    1 + 237x + 1682x^2 + 1682x^3 + 237x^4 + x^5, & 
                           \text{if \ $n=5$} \\ 
    1 + 722x + 10543x^2 + 23548x^3 + 10543x^4 + 
        722x^5 + x^6, &  \text{if \ $n=6$}.  
                \end{cases} \]

A \emph{descending run} 
of a permutation $w \in \fS_n$ is any maximal string 
$\{a, a+1,\dots,b\}$ of integers such that $w(a) > 
w(a+1) > \cdots > w(b)$. A \emph{left peak} of $w$ is 
any index $i \in [n-1]$ such that $w(i-1) < w(i) > 
w(i+1)$, where $w(0) := 0$ (note that $1$ can be a 
left peak, but $n$ cannot). The following result 
combines \cite[Theorem~4.7]{Ch08} with 
\cite[Proposition~4.15]{Pet07} and provides a 
$\bB_n$-analog to Theorem~\ref{thm:FSSAn}. The 
permutation obtained from $w \in \bB_n$ by 
forgetting all signs is denoted by $|w|$.
\begin{theorem} \label{thm:CPBn}
{\rm (\cite{Ch08, Pet07})}
For all $n \ge 1$,
\begin{equation}
\label{eq:CPBn}
B_n(x) \ = \ \sum_{i=0}^{\lfloor n/2 \rfloor} 
\gamma^B_{n,i} \, x^i (1+x)^{n-2i},
\end{equation} 
where $\gamma^B_{n,i}$ is equal to each of the 
following:

\begin{itemize}
\item[$\bullet$] the number of permutations $w \in 
\fS_n$ with $i$ left peaks, multiplied by $4^i$,
\item[$\bullet$] the number of signed permutations 
$w \in \bB_n$ with $\des^B(w) = i$, such that $|w| 
\in \fS_n$ has $i$ descending runs of size at least 
two.
\end{itemize}
In particular,
\begin{equation}
\label{eq:x=-1Bn}
  B_n(-1) \ = \ \begin{cases}
  \, 0, & \text{if $n$ is odd}, \\
  \, (-1)^{n/2} \, \gamma^B_{n,n/2}, & 
  \text{if $n$ is even}  \end{cases}  
\end{equation}
for all $n$, where $\gamma^B_{n,n/2}$ is the 
number of up-down permutations in $\fS_n$, multiplied 
by $4^{n/2}$. 
\end{theorem}

For the first few values of $n$, the numbers 
$\gamma^B_{n,i}$ are determined from the expressions 
\[  B_n(x) \ = \ \begin{cases}
    1 + x, & \text{if \ $n=1$} \\
    (1 + x)^2 + 4x, & \text{if \ $n=2$} \\
    (1 + x)^3 + 20x(1 + x), & \text{if \ $n=3$} \\
    (1 + x)^4 + 72x(1 + x)^2 + 80x^2, & 
                           \text{if \ $n=4$} \\
    (1 + x)^5 + 232x(1 + x)^3 + 976x^2(1 + x), & 
                           \text{if \ $n=5$} \\
    (1 + x)^6 + 716x(1 + x)^4 + 7664x^2(1 + x)^2 
              + 3904x^3,  & \text{if \ $n=6$}.  
                \end{cases} \]

An interesting extension of Theorem~\ref{thm:W(x)}
to affine Weyl groups was found by K.~Dilks, 
T.K.~Petersen and J.~Stembridge~\cite{DPS09}. Let 
$(W, S)$ be a Coxeter system with Coxeter 
length function $\ell_S: W \to \NN$, as before, and 
assume $W$ is finite, crystallographic and 
irreducible. Then, $W$ has a longest element $s_0$ 
and the set of \emph{affine right descents} of 
$w \in W$ is defined as
\[  \aDes(w) \ := \ \begin{cases}
    \Des(w), & \text{if \ $\ell_S(ws_0) < \ell_S(w)$} 
    \\
    \Des(w) \cup \{s_0\}, & \text{if \ $\ell_S(ws_0) 
    > \ell_S(w)$.}   \end{cases} \]
The \emph{affine Eulerian polynomial} associated to 
$W$ is defined as 
\begin{equation}
\label{eq:defaffW(x)}
W_a(x) \ := \ \sum_{w \in W} x^{\ades(w)}, 
\end{equation}
where $\ades(w)$ is the number of affine right descents
of $w \in W$.
\begin{theorem}
\label{thm:DPSaffW}
{\rm (Dilks--Petersen--Stembridge~\cite[Theorem~4.2]{DPS09})}
The affine Eulerian polynomial $W_a(x)$ 
is $\gamma$-positive for every finite irreducible 
crystallographic Coxeter group $W$.
\end{theorem}

We close this section with an intriguing related 
problem, posed by I.~Gessel in 2005 (see 
\cite[Conjecture~10.2]{Bra08} 
\cite[Conjecture~1]{Pet13b} 
\cite[Problem~4.12]{Pet15}). The \emph{two-sided 
Eulerian polynomial} associated to a finite Coxeter 
group $W$ is defined as
\begin{equation}
\label{eq:def2sided}
  W(x, y) \ = \ \sum_{w \in W} x^{\des(w)} 
                  y^{\des(w^{-1})}.
\end{equation} 
The specialization $W(x,x)$ appeared also in work 
of A.~Hultman \cite[Example~5.9]{Hu07}. The 
following result was conjectured by Gessel 
(unpublished) for the symmetric groups and, more 
generally, by Petersen~\cite[Conjecture~1]{Pet16} 
for finite Coxeter groups.
\begin{theorem} \label{thm:2sided}
{\rm (Lin~\cite{Li16})}
Let $W$ be a symmetric or hyperoctahedral group. 
Then, there exist nonnegative integers $\gamma_{i, 
j} = \gamma_{i,j}(W)$ such that
\begin{equation}
\label{eq:thm2sided}
W(x, y) \ = \ \sum_{2i+j \le n} \gamma_{i,j} 
              (xy)^i (x + y)^j (1 + xy)^{n-2i-j},
\end{equation}
where $n$ is the rank of $W$. 
\end{theorem}

It remains an interesting open problem to find a 
combinatorial interpretation for the numbers 
$\gamma_{i,j}(W)$, even in the symmetric 
group case. 

\subsubsection{Derangements}
\label{subsubsec:derange}

Counting derangements (permutations without fixed
points) in $\fS_n$ by the number of excedances leads
to a well-behaved analog of the Eulerian polynomial
$A_n(x)$, defined by 
\begin{equation}
\label{eq:defdn}
d_n(x) \ := \ \sum_{w \in \dD_n} x^{\exc(w)},
\end{equation}
where $\dD_n$ denotes the set of all derangements
in $\fS_n$. For the first few values of $n$, we have
\[ d_n(x) \ = \ \begin{cases}
    0, & \text{if \ $n=1$} \\
    x, & \text{if \ $n=2$} \\
    x + x^2, & \text{if \ $n=3$} \\
    x + 7x^2 + x^3, & \text{if \ $n=4$} \\
    x + 21x^2 + 21x^3 + x^4, & 
                           \text{if \ $n=5$} \\
    x + 51x^2 + 161x^3 + 51x^4 + x^5, & 
                           \text{if \ $n=6$} \\
    x + 113x^2 + 813x^3 + 813x^4 + 113x^5 + x^6, 
                           & \text{if \ $n=7$}.  
                \end{cases} \]

The polynomial $d_n(x)$ (often called the 
\emph{$n$th derangement polynomial}) was first 
considered in a purely combinatorial context
in~\cite[p.~530]{Sta89} by Stanley who, however,
seems to have been motivated by a geometric 
interpretation~\cite[Proposition~2.4]{Sta92} 
of $d_n(x)$; see Section~\ref{subsec:exa} for more 
explanation. While the symmetry of $d_n(x)$ is 
nearly obvious, its unimodality was derived by
Brenti~\cite[Corollary~1]{Bre90} from a more 
general result in the theory of symmetric 
functions, although it also follows from deep 
results of Stanley~\cite{Sta92} on local 
$h$-polynomials, discussed in 
Section~\ref{subsec:local}. A more elementary
combinatorial proof was later given by 
Stembridge~\cite[Section~2]{Ste92}. More 
recently, using methods discussed in 
Section~\ref{subsec:gmethods}, M.~Juhnke-Kubitzke,
S.~Murai and R.~Sieg~\cite[Corollary~4.2]{KMS17} 
found the recurrence
\begin{equation}
\label{eq:KMSdn}
d_n(x) \ = \ \sum_{k=0}^{n-2} \binom{n}{k} d_k(x)
(x + x^2 + \cdots + x^{n-1-k}),
\end{equation}
which directly implies the unimodality of $d_n(x)$
by induction on $n$.

The question of $\gamma$-positivity of $d_n(x)$ 
arises naturally. Just as is the case with $A_n(x)$,
the polynomial $d_n(x)$ turns out to be 
real-rooted for all $n$~\cite{Zha95}, so the 
interesting part of the question is to find a proof
of $\gamma$-positivity which provides a combinatorial
interpretation for the $\gamma$-coefficients. For the 
first few values of $n$, we have 
\[  d_n(x) \ = \ \begin{cases}
    x, & \text{if \ $n=2$} \\
    x(1 + x), & \text{if \ $n=3$} \\
    x(1 + x)^2 + 5x^2, & \text{if \ $n=4$} \\
    x(1 + x)^3 + 18x^2(1 + x), & 
                           \text{if \ $n=5$} \\
    x(1 + x)^4 + 47x^2(1 + x)^2 + 61x^3, & 
                           \text{if \ $n=6$} \\
    x(1 + x)^5 + 108x^2(1 + x)^3 + 479x^3(1 + x), 
                         & \text{if \ $n=7$}.  
                \end{cases} \]

For a permutation $w \in \fS_n$, a \emph{double 
descent} is any index $2 \le i \le n-1$ such that 
$w(i-1) > w(i) > w(i+1)$; a \emph{left to right 
maximum} is any index $1 \le j \le n$ such that 
$w(i) < w(j)$ for all $1 \le i < j$.
\begin{theorem} \label{thm:ASdn}
{\rm (cf. \cite[Theorem~1.4]{AS12})}
We have $d_n(x) = \sum_{i=0}^{\lfloor n/2 \rfloor} 
\xi_{n,i} x^i (1+x)^{n-2i}$, where $\xi_{n,i}$ is 
equal to each of the following:

\begin{itemize}
\item[$\bullet$] the number of derangements $w \in 
\dD_n$ with $i$ excedances and no double excedance,
\item[$\bullet$] the number of $w \in \fS_n$ for 
which $\Asc(w) \in \Stab([2,n-2])$ has $i-1$ elements,
\item[$\bullet$] the number of $w \in \fS_n$ for 
which $\Des(w) \in \Stab([2,n-2])$ has $i-1$ elements,
\item[$\bullet$] the number of permutations $w \in 
\fS_n$ with $i$ descents and no double descent, 
such that every left to right maximum of $w$ is a 
descent.
  \end{itemize}
In particular, $d_n(x)$ is $\gamma$-positive and
\begin{equation}
\label{eq:x=-1dn}
  d_n(-1) \ = \ \begin{cases}
  \, 0, & \text{if $n$ is odd}, \\
  \, (-1)^{n/2} \, \xi_{n,n/2}, & 
  \text{if $n$ is even}  \end{cases}  
\end{equation}
for all $n$, where $\xi_{n,n/2}$ is the number of
up-down permutations in $\fS_n$. 
\end{theorem}

The first three interpretations appeared 
(implicitly or explicitly) in various contexts in 
independent works by several authors, roughly at
the same time, who employed different methods; 
see~\cite[Equations~(1.3) and (3.2)]{AS12} 
\cite[Section~4]{LSW12} \cite[Section~5]{SW10}
\cite[Section~6]{SW17} \cite{SZ12} \cite{SuWa14}. 
These references provide various interesting 
refinements, some of which are
described in the sequel. A symmetric function 
generalization is discussed in 
Section~\ref{subsec:sym}. 

Three refinements of the first interpretation 
given in Theorem~\ref{thm:ASdn} were found by H.~Shin 
and J.~Zeng~\cite{SZ12, SZ16}. We denote by $c(w)$ the 
number of cycles of $w \in \fS_n$ and note that the 
meanings of $\inv(w)$ and $\nest(w)$ have been 
explained earlier, before the statement of 
Theorem~\ref{thm:SZAn}. The result in the following 
theorem about $\nest(w)$ is the specialization $q=1$ 
of~\cite[Corollary~9]{SZ12}, the one about $c(w)$ 
(which also follows from the proof of 
Theorem~\ref{thm:ASdn} given 
in~\cite[Section~4]{AS12}) is a restatement 
of~\cite[Theorem~11]{SZ12} and the one about $\inv(w)$ 
appears as~\cite[Theorem~2]{SZ16}. 
\begin{theorem} \label{thm:SZdn}
{\rm (Shin--Zeng \cite{SZ12, SZ16})}
For all positive integers $n$ and for each of 
the statistics $\stat(w) \in \{ c(w), \inv(w), 
\nest(w) \}$  
\begin{equation}
\label{eq:SZdn}
\sum_{w \in \dD_n} q^{\stat(w)} x^{\exc(w)} \ = \ 
\sum_{i=0}^{\lfloor n/2 \rfloor} b_{n,i} (q) \, 
                    x^i (1+x)^{n-2i},
\end{equation}
where 
\begin{equation}
\label{eq:SZxin}
b_{n,i} (q) \ = \ \sum_{w \in \dD_n(i)} 
q^{\stat(w)}
\end{equation}
and $\dD_n(i)$ consists of all elements of $\dD_n$ 
with exactly $i$ excedances and no double excedance.
\end{theorem}

The following theorem refines the third 
interpretation of the $\gamma$-positivity of $d_n(x)$,
given in Theorem~\ref{thm:ASdn}. This result was stated 
in~\cite[Remark~5.5]{SW10} without giving an explicit
interpretation for the coefficients $\xi_i(p,q)$. The 
combinatorial interpretation which appears here 
follows from the methods used to prove the 
specialization $p=1$ in~\cite[Theorem~6.1]{SW17}; 
the proof will be sketched in 
Section~\ref{subsec:sfmethods}. The special case 
$p=1$ is also implicit in~\cite{LSW12} (see 
Equation~(1.3) and Corollary~3.7 there) and is 
proven by different methods in~\cite{LZ15}.
\begin{theorem} \label{thm:SWdn}
{\rm (Shareshian--Wachs~\cite[Remark~5.5]{SW10}
\cite[Section~6]{SW17})}
For all $n \ge 1$
\begin{equation}
\label{eq:SWdn}
\sum_{w \in \dD_n} p^{\des(w)} q^{\maj(w)-\exc(w)} 
x^{\exc(w)} \ = \ 
\sum_{i=0}^{\lfloor n/2 \rfloor} \xi_{n,i} (p,q) \, 
                    x^i (1+x)^{n-2i},
\end{equation}
where 
\begin{equation}
\label{eq:SWxin}
\xi_{n,i} (p,q) \ = \ p \cdot \sum_w p^{\des(w^{-1})} 
                         q^{\maj(w^{-1})}
\end{equation}
and the sum runs through all permutations $w \in 
\fS_n$ for which $\Des(w) \in \Stab([2,n-2])$ has 
$i-1$ elements.
\end{theorem}

The following result, which is derived using the same
methods in Section~\ref{subsec:sfmethods}, generalizes
the specialization $p=1$ of Theorem~\ref{thm:SWdn}.
We denote by $\fix(w)$ the number of fixed points of 
$w \in \fS_n$.
\begin{theorem} \label{thm:ASWAn}
{\rm (Shareshian--Wachs~\cite[Corollary~4.6]{SW10}
\cite[Theorem~6.1]{SW17})}
For all $n \ge 1$ and $0 \le k \le n$
\begin{equation}
\label{eq:ASWAn}
\sum_{w \in \fS_n: \, \fix(w) = k} q^{\maj(w)-\exc(w)} 
x^{\exc(w)} \ = \ \binom{n}{k}_q
\sum_{i=0}^{\lfloor (n-k)/2 \rfloor} \xi_{n-k,i} (q) \, 
                    x^i (1+x)^{n-k-2i},
\end{equation}
where $\binom{n}{k}_q$ is a $q$-binomial coefficient,
$\xi_{n,i} (q) := \xi_{n,i} (1,q)$ and $\xi_{n,i} 
(p, q)$ is as in the statement of 
Theorem~\ref{thm:SWdn}.
\end{theorem}
\begin{problem}
Find a $p$-analog of Theorem~\ref{thm:ASWAn} which 
reduces to Theorem~\ref{thm:SWdn} for $k=0$. 
\end{problem}

For a generalization of Theorem~\ref{thm:ASdn} to 
$r$-colored permutations and some related results,
see Section~\ref{subsubsec:colored}. We now briefly
describe an application to the $\gamma$-positivity
of binomial Eulerian polynomials.
\begin{remark} \label{rem:dnBinAn} \rm
By the symmetry $\widetilde{A}_n(x) = x^n 
\widetilde{A}_n(1/x)$ of the binomial Eulerian 
polynomial $\widetilde{A}_n(x)$, its defining 
equation~(\ref{eq:defbinAn}) can be rewritten 
as
\[ \widetilde{A}_n(x) \ = \ \sum_{m=0}^n 
   \binom{n}{m} x^{n-m} A_m(x) . \]
Replacing $A_m(x)$ by the expression 
$\sum_{k=0}^m \binom{m}{k} d_k(x)$ and changing 
the order of summation results in the formula
\begin{equation}
\label{eq:binAndn}
\widetilde{A}_n(x) \ = \ \sum_{k=0}^n \binom{n}{k} 
  d_k(x) \, (1+x)^{n-k}.
\end{equation}
Using the $\gamma$-expansion of 
Theorem~\ref{thm:ASdn} for the derangement 
polynomials, we conclude that (\ref{eq:SWbinAn}) 
holds with  
\begin{equation}
\label{eq:binAndn-gamma}
\widetilde{\gamma}_{n,i} \ = \ \sum_{k=2i}^n 
  \binom{n}{k} \xi_{k,i}.
\end{equation}
Because of the third interpretation for the 
coefficients $\xi_{n,i}$ given in 
Theorem~\ref{thm:ASdn}, this formula is 
equivalent to the combinatorial interpretation 
for $\widetilde{\gamma}_{n,i}$ predicted 
in~\cite{PRW08}. This approach can be generalized 
in the context of $r$-colored permutations; details 
will appear elsewhere.
\end{remark}

\subsubsection{Involutions}
\label{subsubsec:invo}

Let $\iI_n := \{ w \in \fS_n: w^{-1} = w \}$ be the 
set of involutions in $\fS_n$ and 
\begin{equation}
\label{eq:defIn}
I_n(x) \ = \ \sum_{w \in \iI_n} x^{\des(w)} 
\end{equation}
be the polynomial defining the Eulerian distribution
on $\iI_n$. For the first few values of $n$, we have
\[ I_n(x) \ = \ \begin{cases}
    1, & \text{if \ $n=1$} \\
    1 + x, & \text{if \ $n=2$} \\
    1 + 2x + x^2, & \text{if \ $n=3$} \\
    1 + 4x + 4x^2 + x^3, & \text{if \ $n=4$} \\
    1 + 6x + 12x^2 + 6x^3 + x^4, & 
                           \text{if \ $n=5$} \\
    1 + 9x + 28x^2 + 28x^3 + 9x^4 + x^5, & 
                           \text{if \ $n=6$} \\
    1 + 12x + 57x^2 + 92x^3 + 57x^4 + 12x^5 + x^6, 
                         & \text{if \ $n=7$}.  
                \end{cases} \]
The polynomial $I_n(x)$ seems to have first appeared
in~\cite{Str80}, where V.~Strehl proved its symmetry 
(conjectured by D.~Dumont). Strehl's argument uses 
basic properties of the Robinson--Schensted
correspondence to derive the alternative formula
\begin{equation}
\label{eq:StreIn}
I_n(x) \ = \ \sum_{Q \in \SYT(n)} x^{\des(Q)}, 
\end{equation}
where $\SYT(n)$ stands for the set of all standard 
Young tableaux with $n$ squares and $\des(Q)$ is the 
numbers of entries (called descents) $i \in [n-1]$
such that $i+1$ appears in $Q$ in a lower row than 
$i$, for $Q \in \SYT(n)$. This formula makes the 
symmetry of $I_n(x)$ apparent, since transposing a
standard Young tableau interchanges descents 
with ascents (nondescents).

Two basic results about $I_n(x)$ are as follows.
The proof of the second part, given in~\cite{GZ06}, 
uses the first part to find recursions for the 
coefficients of $I_n(x)$ and proceeds by induction
on $n$. 
\begin{theorem} \label{thm:In(x)}
\begin{itemize}
\itemsep=5pt
\item[\textrm{(a)}] 
{\rm (D\'esarm\'enien--Foata~\cite{DF85})} We have
\begin{equation}
\label{eq:genIn}
\sum_{n \ge 0} \, \frac{I_n(x)}{(1-x)^{n+1}} \, z^n 
   \ = \ \sum_{m \ge 0} \, \frac{x^m}{(1-z)^{m+1} 
         (1-z^2)^{m(m+1)/2}}, 
\end{equation}
where $I_0(x) := 1$.

\item[\textrm{(b)}] {\rm (Guo--Zeng~\cite{GZ06})}
The polynomial $I_n(x)$ is (symmetric and) unimodal 
for every positive integer $n$.
\end{itemize}
\end{theorem}

The polynomials $I_n(x)$ are not real-rooted (in 
fact, not even log-concave~\cite{BBS09}) for $n$ 
large enough, but the following data suggests they 
may be $\gamma$-positive:
\[ I_n(x) \ = \ \begin{cases}
    1, & \text{if \ $n=1$} \\
    1 + x, & \text{if \ $n=2$} \\
    (1 + x)^2, & \text{if \ $n=3$} \\
    (1 + x)^3 + x(1 + x), & \text{if \ $n=4$} \\
    (1 + x)^4 + 2x(1 + x)^2 + 2x^2, & 
                           \text{if \ $n=5$} \\
    (1 + x)^5 + 4x(1 + x)^3 + 6x^2(1 + x), & 
                           \text{if \ $n=6$} \\
    (1 + x)^6 + 6x(1 + x)^4 + 18x^2(1 + x)^2, 
                         & \text{if \ $n=7$}.  
                \end{cases} \]

Somewhat surprisingly, the following intriguing 
conjecture is still open (an analogous statement
for fixed-point free involutions is conjectured
\cite[Conjecture~4.3]{GZ06} for large $n$).
\begin{conjecture} \label{conj:In(x)}
{\rm (Guo--Zeng~\cite[Conjecture~4.1]{GZ06})}
The polynomial $I_n(x)$ is $\gamma$-positive
for every positive integer $n$.
\end{conjecture}

The previous conjecture suggests that it may be 
interesting to study the distribution of the 
descent set $\Des(w)$ for $w \in \iI_n$. For 
instance, for
$S \subseteq [n-1]$ let $\beta_n(S)$ denote the 
number of permutations $w \in \fS_n$ such that
$\Des(w) = S$. It is well known 
(see~\cite[Section~1.6.3]{StaEC1}) that $\beta_n(S)
= \beta_n([n-1] \sm S)$ for every $S \subseteq 
[n-1]$ and that for each $n \ge 1$, when $S$ 
ranges over all subsets of $[n-1]$, the quantity  
$\beta_n (S)$ attains its maximum for $S = \{1, 
3, 5,\dots\} \cap [n-1]$. 

Similarly, for
$S \subseteq [n-1]$ let $\delta_n(S)$ denote the 
number of involutions $w \in \iI_n$ such that
$\Des(w) = S$. Strehl's proof of the symmetry 
of $I_n(x)$ shows that $\delta_n(S) = \delta_n
([n-1] \sm S)$ for every $S \subseteq [n-1]$.
\begin{question}
Is it true that for each $n \ge 1$, when $S$ ranges 
over all subsets of $[n-1]$, the quantity  
$\delta_n (S)$ attains its maximum for $S = 
\{1, 3, 5,\dots\} \cap [n-1]$?
\end{question}

It is natural to consider the $\bB_n$-analog
\begin{equation}
\label{eq:defInB}
I^B_n(x) \ = \ \sum_{w \in \iI^B_n} x^{\des_B(w)} 
\end{equation}
of $I_n(x)$, where $\iI^B_n := \{ w \in \bB_n: 
w^{-1} = w \}$ is the set of involutions in the 
hyperoctahedral group $\bB_n$. Presumably (but not
obviously), the right-hand side of the defining 
equation~(\ref{eq:defInB}) is unaffected when 
$\des_B$ is replaced with the Coxeter group descent
statistic $\des^B$ for $\bB_n$. For the first few 
values of $n$, we have
\[ I^B_n(x) \ = \ \begin{cases}
    1 + x, & \text{if \ $n=1$} \\
    1 + 4x + x^2, & \text{if \ $n=2$} \\
    1 + 9x + 9^2 + x^3, & \text{if \ $n=3$} \\
    1 + 17x + 40x^2 + 17x^3 + x^4, & 
                          \text{if \ $n=4$} \\
    1 + 28x + 127x^2 + 127x^3 + 28x^4 + x^5, & 
                          \text{if \ $n=5$} \\
    1 + 43x + 331x^2 + 634x^3 + 331x^4 + 43x^5 + x^6, 
                            & \text{if \ $n=6$}.  
                \end{cases} \]

The symmetry of $I^B_n(x)$ can be demonstrated by an 
analog of Strehl's argument, using basic 
properties (see~\cite[Proposition~5.1]{AAER17})
of the Robinson--Schensted correspondence of type
$B$ and replacing $\SYT(n)$ with the set of all
standard Young bitableaux with a total of $n$
squares; see~\cite[Section~3.2]{MouMS} for the 
details. 

The following analog 
\begin{equation}
\label{eq:genInB}
\sum_{n \ge 0} \, \frac{I^B_n(x)}{(1-x)^{n+1}} \, z^n 
   \ = \ \sum_{m \ge 0} \, \frac{x^m}{(1-z)^{2m+1} 
         (1-z^2)^{m^2}}, 
\end{equation}
of Equation~(\ref{eq:genIn}) for $I^B_n(x)$, where 
$I^B_0(x) := 1$, was found by 
V.~Moustakas~\cite{MouMS}, who also showed that 
$I^B_n(x)$ is unimodal for every positive integer 
$n$. As expected, the following data suggests that 
$I^B_n(x)$ may be $\gamma$-positive for every $n$:
\[ I^B_n(x) \ = \ \begin{cases}
    1 + x, & \text{if \ $n=1$} \\
    (1 + x)^2 + 2x, & \text{if \ $n=2$} \\
    (1 + x)^3 + 6x(1 + x), & \text{if \ $n=3$} \\
    (1 + x)^4 + 13x(1 + x)^2 + 8x^2, & 
                             \text{if \ $n=4$} \\
    (1 + x)^5 + 23x(1 + x)^3 + 48x^2(1 + x), & 
                             \text{if \ $n=5$} \\
    (1 + x)^6 + 37x(1 + x)^4 + 168x^2(1 + x)^2 + 
                      56x^3, & \text{if \ $n=6$}.  
                \end{cases} \]

The formulas in the following proposition, which 
seem not to have appeared in the literature 
explicitly before, express $I_n(x)$ and 
$I^B_n(x)$ in terms of Eulerian polynomials of 
types $A$ and $B$, respectively. We sketch their 
proofs, which we find interesting. We recall (see,
for instance, \cite[Section~2.3]{AAER17}) that 
the cycle form of elements of $\bB_n$ involves 
positive (paired) cycles and negative (balanced) 
cycles.
\begin{proposition} \label{prop:athIn}
For $n \ge 1$,
\begin{align}
I_n(x) & \ = \ \frac{1}{n!} \sum_{w \in \fS_n} 
             (1-x)^{n-c(w^2)} A_{c(w^2)}(x), 
                           \label{eq:athIn} \\
\nonumber & \\ 
I^B_n(x) & \ = \ \frac{1}{2^n n!} 
       \sum_{w \in \bB_n} (1-x)^{n-c_{+}(w^2)} 
       B_{c_{+}(w^2)}(x), \label{eq:athInB}
\end{align}
where $c(u)$ stands for the number of cycles of 
$u \in \fS_n$ and $c_{+}(v)$ stands for the 
number of positive cycles of $v \in \bB_n$.
\end{proposition}
\begin{proof}
Recall from \cite[Section~7.19]{StaEC2} that the 
fundamental quasisymmetric function associated to 
$S \subseteq [n-1]$ is defined as
\begin{equation} \label{eq:defFS(x)}
  F_{n, S} (\bx) \ =
  \sum_{\substack{1 \le i_1 \le i_2 \le \cdots 
  \le i_n \\ j \in S \,\Rightarrow\, 
  i_j < i_{j+1}}}
  x_{i_1} x_{i_2} \cdots x_{i_n},
\end{equation}
where $\bx = (x_1, x_2,\dots)$ is a sequence of 
commuting independent indeterminates, and that 
\begin{equation} \label{eq:speFS(x)}
  \sum_{m \ge 1} F_{n, S} (1^m) x^{m-1} \ = \ 
  \frac{x^{|S|}}{(1-x)^{n+1}}.
\end{equation}
Applying this equality for $S = \Des(w)$, summing
over all involutions $w \in \iI_n$, changing the 
order of summation on the left-hand side and using 
the correspondence between involutions and standard
Young tableaux gives  
\begin{align*}
\frac{I_n(x)}{(1-x)^{n+1}} & \ = \ \sum_{m \ge 1} 
\sum_{w \in \iI_n} F_{n, \Des(w)} (1^m) x^{m-1} \ = 
\sum_{m \ge 1} \sum_{Q \in \SYT(n)} F_{n, \Des(Q)} 
(1^m) x^{m-1} \\
& \ = \ \sum_{m \ge 1} \sum_{\lambda \vdash n} 
\sum_{Q \in \SYT(\lambda)} F_{n, \Des(Q)} (1^m) 
x^{m-1}. 
\end{align*}
Using the well known expansion 
\cite[Theorem~7.19.7]{StaEC2}
\begin{equation} \label{eq:sFexpan}
  s_\lambda(\bx) \ = \ \sum_{Q \in \SYT (\lambda)} 
  F_{n, \Des(Q)} (\bx)
\end{equation}
of the Schur function $s_\lambda(\bx)$ associated 
to $\lambda \vdash n$, our formula for $I_n(x)$ 
may be rewritten as
\[ \frac{I_n(x)}{(1-x)^{n+1}} \ = \ \sum_{m \ge 1} 
\sum_{\lambda \vdash n} s_\lambda (1^m) x^{m-1}. \]
Expanding now $s_\lambda(\bx)$ in the power-sum 
basis (see~\cite[Theorem~7.18.5]{StaEC2}) and 
changing the order of summation gives  
\[ \frac{I_n(x)}{(1-x)^{n+1}} \ = \ \frac{1}{n!}
\sum_{m \ge 1} \sum_{\lambda \vdash n} \sum_{u \in 
\fS_n} \chi^\lambda(u) m^{c(u)} x^{m-1} \ = \ 
\frac{1}{n!} \sum_{u \in \fS_n} \sum_{m \ge 1} 
m^{c(u)} x^{m-1} \sum_{\lambda \vdash n} 
\chi^\lambda(u), \]
where $\chi^\lambda$ is the irreducible character
of $\fS_n$ corresponding to $\lambda \vdash n$. 
The desired expression (\ref{eq:athIn}) for $I_n(x)$ 
follows by applying Worpitzky's identity 
$\sum_{m \ge 1} m^k x^{m-1} = A_k(x)/(1-x)^{k+1}$ 
and the fact (see~\cite[p.~58]{Isa76}) that 
$\sum_{\lambda \vdash n} \chi^\lambda(u)$ is equal  
to the number of $w \in \fS_n$ satisfying $w^2 = u$, 
for every $u \in \fS_n$.

The proof of the expression (\ref{eq:athInB}) for
$I^B_n(x)$ is similar, provided one uses Poirier's 
signed quasisymmetric functions~\cite{Poi98} to 
replace the functions $F_{n,S}(\bx)$. In the sequel,
we assume some familiarity 
with~\cite[Section~2]{AAER17}, especially with the
notion of signed descent set for signed permutations
and standard Young bitableaux.
 
Let $\bx = (x_1, x_2,\dots)$ and $\by = (y_1, 
y_2,\dots)$ be sequences of commuting independent 
indeterminates. Given $w \in \bB_n$, let 
$\varepsilon = (\varepsilon_1, 
\varepsilon_2,\dots,\varepsilon_n) \in \{-, +\}^n$ 
be the vector with $i$th coordinate equal to 
the sign of $w(i)$ and let $\Des(w)$ be the set 
consisting of those indices $i \in [n-1]$ for which 
either $\varepsilon_i = +$ and $\varepsilon_{i+1} = 
-$, or $\varepsilon_i = \varepsilon_{i+1}$ and 
$|w(i)| > |w(i+1)|$. The signed quasisymmetric 
function associated (in a more general setting) to 
$w$ by Poirier~\cite{Poi98} may be defined as 
\begin{equation} \label{eq:defF(xy)}
  F_w (\bx, \by) \ =
  \sum_{\substack{1 \le i_1 \le i_2 \le \cdots 
  \le i_n \\ j \in \Des(w) \, 
  \Rightarrow \, i_j < i_{j+1}}}
  z_{i_1} z_{i_2} \cdots z_{i_n},
\end{equation}
where $z_{i_j} = x_{i_j}$ if $\varepsilon_j = +$, 
and $z_{i_j} = y_{i_j}$ if $\varepsilon_j = -$;
see~\cite[Sections~2.2 and~2.4]{AAER17}. There is
a similar definition of $F_Q (\bx, \by)$ for every
standard Young bitableau $Q$; 
see~\cite[Section~2]{AAER17} for a uniform 
treatment of these functions. One can check that 
\begin{equation} \label{eq:speF(xy)}
  \sum_{m \ge 1} F_w (1^m,01^{m-1}) x^{m-1} \ = \ 
  \frac{x^{\des_B(w)}}{(1-x)^{n+1}},
\end{equation}
where $F_w (1^m,01^{m-1})$ stands for the 
specialization $x_1 = \cdots = x_m = y_2 = \cdots 
= y_m = 1$, $y_1 = 0$ and $x_i = y_i = 0$ for 
$i > m$ of the function $F_w (\bx,\by)$. Let 
$\SYB(\lambda,\mu)$ denote the set of standard 
Young bitableaux of shape $(\lambda,\mu)$.
Following the proof for $I_n(x)$ described earlier
and using the one-to-one correspondence between 
involutions in $\bB_n$ and standard Young 
bitableaux with a total of $n$ squares, as well 
as the expansion 
\begin{equation} \label{eq:sxyFexpan}
  s_\lambda(\bx) s_\mu(\by) \ = \ \sum_{Q \in \SYB 
  (\lambda,\mu)} F_Q (\bx, \by)
\end{equation}
of~\cite[Proposition~4.2]{AAER17} instead 
of~(\ref{eq:sFexpan}), results in the equation
\[ \frac{I^B_n(x)}{(1-x)^{n+1}} \ = \ \sum_{m \ge 1} 
\sum_{(\lambda, \mu) \vdash n} s_\lambda (1^m) 
s_\mu(1^{m-1}) x^{m-1}. \]
Finally, denote by $\chi^{\lambda, \mu}$ the 
irreducible $\bB_n$-character associated to the 
bipartition $(\lambda, \mu)$ of $n$. One uses 
the characteristic 
map~\cite[Equation~(2.5)]{AAER17} for $\bB_n$ to 
expand $s_\lambda (\bx) s_\mu(\by) = \ch
(\chi^{\lambda,\mu})$ in the power sum basis, the 
type $B$ Worpitzky identity $\sum_{m \ge 1} 
(2m-1)^k x^{m-1} = B_k(x) / (1-x)^{k+1}$ and the 
fact (see again~\cite[p.~58]{Isa76}) that 
$\sum_{(\lambda, \mu) \vdash n} \chi^{\lambda, 
\mu} (v)$ is equal to the number of square roots 
of $v$ in $\bB_n$, for every $v \in \bB_n$, to 
reach the desired conclusion.
\end{proof}

\begin{remark} \rm
Another proof of the symmetry of $I_n(x)$ and 
$I^B_n(x)$ can be infered from 
Proposition~\ref{prop:athIn} as follows. Replacing $x$
by $1/x$ in the right-hand sides of 
Equations~(\ref{eq:athIn}) and~(\ref{eq:athInB}) and
using the symmetry of the Eulerian polynomials $A_k(x)$
and $B_k(x)$, as well as the fact that $n - c(u^2)$ and 
$n - c_{+}(v^2)$ are even for all $u \in \fS_n$ and
$v \in \bB_n$, shows that $x^{n-1} I_n (1/x) = I_n(x)$ 
and $x^n I^B_n (1/x) = I^B(x)$, as desired. 
\qed
\end{remark}

The polynomial $I_n(x)$ affords a natural Coxeter group 
generalization. Given a finite Coxeter group $W$ 
(together with a set of simple generators), let $\iI_W 
:= \{ w \in W: w^{-1} = w\}$ be the set of involutions 
in $W$ and define
\begin{equation}
\label{eq:defIW}
I_W(x) \ := \ \sum_{w \in \iI_W} x^{\des(w)}
\end{equation}
where, as in Equation~(\ref{eq:defW(x)}), $\des(w)$ 
is the number of right descents of $w$ (the notions 
of left and right descents coincide for involutions). 
The polynomial $I_W(x)$ is equal to
$I_n(x)$, when $W$ is the symmetric group $\fS_n$,
and presumably to $I^B_n(x)$, when $W$ is the 
hyperoctahedral group $\bB_n$. 

The symmetry of $I_W(x)$ follows from the work of 
Hultman~\cite[Section~5]{Hu07}, who showed (in 
a more general context) that $I_W(x)$ is equal 
to the $h$-polynomial of a Boolean cell complex 
which is homeomorphic to a sphere. Thus, it seems 
natural to ask the following question.
\begin{question}
For which finite Coxeter groups $W$ is $I_W(x)$ 
unimodal, or even $\gamma$-positive?
\end{question}

Another possible generalization of $I_n(x)$ is  
provided by the polynomial 
\begin{equation}
\label{eq:defInk}
I_{n,k}(x) \ := \ \sum_{w \in \fS_n: w^k = e} 
                     x^{\des(w)},
\end{equation}
defined for positive integers $k$, where $e \in 
\fS_n$ stands for the identity permutation. This
polynomial is no longer symmetric for $k \ge 3$,
but it may still be worthwhile to investigate
its unimodality. Following the proof of the first 
part of Proposition~\ref{prop:athIn} and using 
the main result of~\cite{Roi14}, in the equivalent 
form stated in~\cite[Theorem~7.7]{AAER17}, to 
evaluate the quasisymmetric generating function 
of the descent set over all $k$-roots of $e \in 
\fS_n$, results in the generalization
\begin{equation}
\label{eq:athInk}
I_{n,k}(x) \ = \ \frac{1}{n!} \sum_{w \in \fS_n} 
             (1-x)^{n-c(w^k)} A_{c(w^k)}(x) 
\end{equation}
of Equation~(\ref{eq:athIn}).

\subsubsection{Multiset permutations}
\label{subsubsec:multi}

Let $\lambda = (\lambda_1, \lambda_2,\dots,\lambda_k)
\vdash n$ be an integer partition of $n$. We denote 
by $M_\lambda$ the multiset consisting of 
$\lambda_i$ copies of $i$ for each $i \in [k]$ and by 
$\fS(M_\lambda)$ the set of all permutations of 
$M_\lambda$, written in one-line notation. The sets 
of \emph{ascents, descents} and \emph{excedances} of 
$w = (a_1, a_2,\dots,a_n) \in \fS(M_\lambda)$ are 
defined as 

\begin{itemize}
\item[$\bullet$] $\Asc(w) := \{ i \in [n-1]: a_i 
                                 \le a_{i+1} \}$, 																
\item[$\bullet$] $\Des(w) := \{ i \in [n-1]: a_i 
                                     > a_{i+1} \}$, 
\item[$\bullet$] $\Exc(w) := \{ i \in [n-1]: a_i 
                                    > j_i \}$,
\end{itemize}
where $(j_1, j_2,\dots,j_n)$ is the unique 
permutation of $M_\lambda$ with no descents; the 
cardinalities of these sets are denoted by $\asc(w)$,
$\des(w)$ and $\exc(w)$, respectively. We say that 
$w$ is a \emph{Smirnov permutation} of type $\lambda$ 
if it has no two equal successive entries, and a 
\emph{derangement} of type $\lambda$ if $a_i \ne j_i$ 
for all indices $i \in [n]$. We denote by $\sS_\lambda$ 
and $\dD_\lambda$ the set of all Smirnov permutations 
and derangements of type $\lambda$, respectively. 

The following result, which is a consequence of 
\cite[Theorem~5.1]{LSW12} and its proof, generalizes 
Theorems~\ref{thm:FSSAn} and~\ref{thm:ASdn} in the 
context of permutations of multisets.   
\begin{theorem} \label{thm:LSWmulti}
{\rm 
(Linusson--Shareshian--Wachs~\cite[Section~5]{LSW12})}
For all partitions $\lambda \vdash n$,
\begin{equation}
\label{eq:LSWmulti1}
\sum_{w \in \sS_\lambda} x^{\des(w)} \ = \ 
\sum_{i=0}^{\lfloor (n-1)/2 \rfloor} 
\gamma_{\lambda,i} x^i (1+x)^{n-1-2i}
\end{equation}
and 
\begin{equation}
\label{eq:LSWmulti2}
\sum_{w \in \dD_\lambda} x^{\exc(w)} \ = \ 
\sum_{i=0}^{\lfloor n/2 \rfloor} 
\xi_{\lambda,i} x^i (1+x)^{n-2i},
\end{equation}
where

\begin{itemize}
\item[$\bullet$] $\gamma_{\lambda,i}$ is the number 
of permutations $w \in \fS(M_\lambda)$ for which 
$\Asc(w) \in \Stab([n-2])$ has $i$ elements, and
\item[$\bullet$] $\xi_{\lambda,i}$ is the number 
of permutations $w \in \fS(M_\lambda)$ for which 
$\Asc(w) \in \Stab([2,n-2])$ has $i-1$ elements.
  \end{itemize}
\end{theorem}

\subsubsection{Colored permutations}
\label{subsubsec:colored}

Theorem~\ref{thm:ASdn} has been generalized to the 
wreath product group $\ZZ_r \wr \fS_n$. Recall that 
the elements of $\ZZ_r \wr \fS_n$ can be viewed as 
$r$-colored permutations of the form $\sigma \times 
\bz$, where $\sigma = (\sigma(1), 
\sigma(2),\dots,\sigma(n)) \in \fS_n$ and $\bz = (z_1, 
z_2,\dots,z_n) \in \{0, 1,\dots,r-1\}^n$ (the number 
$z_i$ is thought of as the color assigned to 
$\sigma(i)$). To define the notions of descent and 
excedance for colored permutations, we follow 
\cite[Section~2]{Stei92} \cite[Section~3.1]{Stei94}. 
A \emph{descent} of $\sigma \times \bz \in \ZZ_r 
\wr \fS_n$ is any index $i \in [n]$ such that either 
$z_i > z_{i+1}$, or $z_i = z_{i+1}$ and $\sigma(i)
> \sigma(i+1)$, where $\sigma(n+1) := n+1$ and 
$z_{n+1} := 0$ (in particular, $n$ is a descent if 
and only if $\sigma(n)$ has nonzero color). An 
\emph{excedance} of $\sigma \times \bz$ is any 
index $i \in [n]$ such that either $\sigma(i) > i$, 
or $\sigma(i) = i$ and $z_i > 0$. Let $\des(w)$ and 
$\exc(w)$ be the number of descents and excedances,
respectively, of $w \in \ZZ_r \wr \fS_n$. The Eulerian 
polynomial 
\begin{equation}
\label{eq:defAnr}
A_{n,r}(x) \ := \ \sum_{w \in \ZZ_r \wr \fS_n} 
x^{\des(w)} \ = \ \sum_{w \in \ZZ_r \wr \fS_n} 
x^{\exc(w)} 
\end{equation}
for $\ZZ_r \wr \fS_n$ (the second equality being
the content of \cite[Theorem~3.15]{Stei92} 
\cite[Theorem~15]{Stei94}) was defined and studied 
by E.~Steingr\'imsson~\cite{Stei92, Stei94}, who 
showed it to be real-rooted for all $n, r$; it 
reduces to the Eulerian polynomials $A_n(x)$ and 
$B_n(x)$ when $r=1$ and $r=2$, respectively. The 
derangement polynomial for $\ZZ_r \wr \fS_n$ is 
defined as
\begin{equation}
\label{eq:defdnr}
d_{n,r}(x) \ := \ \sum_{w \in \dD_{n,r}} x^{\exc(w)}, 
\end{equation}
where $\dD_{n,r}$ denotes the set of all derangements
(colored permutations without fixed points of zero 
color) in $\ZZ_r \wr \fS_n$.
This polynomial was introduced and studied by C.~Chow 
and T.~Mansour~\cite{CM10}, who showed 
that it is real-rooted for all positive integers 
$n, r$; it reduces to $d_n(x)$ for $r=1$. For
$r=2$ it was first studied by W.~Chen, R.~Tang and 
A.~Zhao~\cite{CTZ09} and, independently, by 
C.~Chow~\cite{Ch09}. For the first few values of $n$, 
we have
\[ d_{n,2} (x) \ = \ \begin{cases}
    x, & \text{if \ $n=1$} \\
    4x + x^2, & \text{if \ $n=2$} \\
    8x + 20x^2 + x^3, & \text{if \ $n=3$} \\
    16x + 144x^2 + 72x^3 + x^4, & \text{if \ $n=4$} \\
    32x + 752x^2 + 1312x^3 + 232x^4 + x^5, & 
                           \text{if \ $n=5$} \\
    64x + 3456x^2 + 14576x^3 + 9136x^4 + 716x^5 + x^6, 
                                  & \text{if \ $n=6$}.  
                \end{cases} \]

Even though $d_{n,r}(x)$ and $A_{n,r}(x)$ are no 
longer symmetric for $r \ge 2$ and $r \ge 3$, 
respectively, the theory of $\gamma$-positivity is 
still relevant to their study. Indeed, the 
following theorem shows that $d_{n,r}(x)$ 
is equal to the sum of two $\gamma$-positive, hence
symmetric and unimodal, polynomials whose centers of
symmetry differ by 1/2 and thus implies its 
unimodality. We denote by $\Des(w)$ the set of 
descents of $w \in \ZZ_r \wr \fS_n$ and set $\Asc(w) 
:= [n] \sm \Des(w)$. For $r=1$, these notions differ 
from the ones already defined for permutations $w \in 
\fS_n$ only in the fact that $\Asc(w)$ is now forced 
to contain $n$. We call $w \in \ZZ_r \wr \fS_n$ 
\emph{down-up}, if $\Des(w) = \{1, 3, 5,\dots\} 
\cap [n]$. 
\begin{theorem} \label{thm:athdnr}
{\rm (Athanasiadis~\cite[Theorem~1.3 and 
Corollary~6.1]{Ath14})} For all positive integers 
$n, r$
\begin{equation} \label{eq:athdnr}
  d_{n,r} (x) \ = \ \sum_{i=0}^{\lfloor n/2 \rfloor} 
  \xi^+_{n,r,i} \, x^i (1+x)^{n-2i} \ \, + 
  \sum_{i=0}^{\lfloor (n+1)/2 \rfloor} \xi^-_{n,r,i}
       \, x^i (1+x)^{n+1-2i},
\end{equation}
where $\xi^+_{n, r, i}$ is the number of permutations 
$w \in \ZZ_r \wr \fS_n$ for which $\Asc(w) \in 
\Stab([2,n])$ has $i$ elements and contains $n$, and 
$\xi^-_{n, r, i}$ is the number of permutations $w \in 
\ZZ_r \wr \fS_n$ for which $\Asc(w) \in \Stab([2,n-1])$ 
has $i-1$ elements.

In particular, $d_{n,r} (x)$ is unimodal with a peak 
at $\lfloor (n+1)/2 \rfloor$. Moreover, if $r \ge 2$, 
then $(-1)^{\lfloor \frac{n+1}{2} \rfloor} \, 
d^r_n (-1)$ is equal to the number of down-up colored 
permutations in $\ZZ_r \wr \mathfrak{S}_n$.
\end{theorem}

The following question seems natural to ask.
\begin{question}
Is there a positive expansion for $A_{n,r}(x)$ of the 
type provided by Equation~(\ref{eq:athdnr}) 
for $d_{n,r}(x)$? Is there a $(p,q)$-analog which
reduces to Theorem~\ref{thm:SWAn} for $r=1$?
\end{question}
\begin{problem}
Find a $(p,q)$-analog of Theorem~\ref{thm:athdnr}
which reduces to Theorem~\ref{thm:SWdn} for $r=1$. 
\end{problem}

The proof of Theorem~\ref{thm:athdnr} uses a general
result of Linusson, Shareshian and 
Wachs~\cite[Corollary~3.8]{LSW12} on the M\"obius 
function of the Rees product for posets, which we 
discuss in Section~\ref{subsec:homo}.
\begin{problem}
Find a direct combinatorial proof of 
Theorem~\ref{thm:athdnr}. 
\end{problem}

In the case $r=2$, further discussed in the
sequel, a different combinatorial interpretation for
the numbers $\xi^+_{n,2,i}$ and $\xi^{-}_{n,2,i}$,
in terms of permutations with no double excedance, 
was found by Shin and Zeng~\cite[Section~2]{SZ16}.
This suggests the problem to find a combinatorial 
interpretation for the numbers $\xi^+_{n,r,i}$ and 
$\xi^{-}_{n,r,i}$ which generalizes the first 
interpretation of the numbers $\xi_{n,i}$, stated 
in Theorem~\ref{thm:ASdn}, and prove directly its 
equivalence to the interpretation of 
Theorem~\ref{thm:athdnr}. 

Let us denote by $d^+_{n,r}(x)$ and $d^{-}_{n,r}(x)$ 
the first and second summand, respectively, in the 
right-hand side of~(\ref{eq:athdnr}). A geometric 
interpretation of the former will be discussed in 
Section~\ref{subsec:exa}. The following statement 
appeared as \cite[Conjecture~3.7.10]{Sav13} in the 
special case $r=2$; see also 
\cite[Question~4.11]{Ath16a}.
\begin{conjecture} \label{conj:dnr}
The polynomials $d^+_{n,r}(x)$ and $d^{-}_{n,r}(x)$ 
are real-rooted for all positive integers $n,r$. 
\end{conjecture}

We close our discussion with generalizations of 
$A_n(x)$ and $d_n(x)$ to $r$-colored permutations 
which are different from, but related to (especially 
for $r=2$), $A_{n,r}(x)$ and $d_{n,r}(x)$. 
The \emph{flag excedance} of a colored permutation 
$w = \sigma \times \bz \in \ZZ_r \wr \fS_n$ was 
defined by E.~Bagno and D.~Garber~\cite{BG06} as
\begin{equation} \label{eq:deffexc}
  \fexc(w) \ = \ r \cdot \exc (\sigma) \ + \, 
                                  \sum_{i=1}^n z_i,
\end{equation}
where the sum of the colors $z_i$ takes place in 
$\ZZ$. The generating polynomial of $\fexc$ over 
all $r$-colored permutations factors nicely 
\cite{BB07, FH11} (see also 
\cite[Proposition~2.2]{Ath14}) as 
\begin{equation}
\label{eq:defAfnr}
A^\fexc_{n,r}(x) \ := \ \sum_{w \in \ZZ_r \wr \fS_n} 
x^{\fexc(w)} \ = \ (1 + x + x^2 + \cdots + x^{r-1})^n 
\, A_n(x). 
\end{equation}
The polynomial
\begin{equation}
\label{eq:deffnr}
f_{n,r}(x) \ := \ \sum_{w \in \dD_{n,r}} x^{\fexc(w)} 
\end{equation}
was studied by P.~Mongelli~\cite[Section~3]{Mo13} in 
the case $r=2$ and by Z.~Lin~\cite[Section~2.4.1]{Lin14} 
and by H.~Shin and J.~Zeng~\cite{SZ16} for any 
$r \ge 1$; it reduces to $d_n(x)$ for $r=1$. For 
$r=2$ and for the first few values of $n$, we have
\[ f_{n,2} (x) \ = \ \begin{cases}
    x, & \text{if \ $n=1$} \\
    x + 3x^2 + x^3, & \text{if \ $n=2$} \\
    x + 7x^2 + 13x^3 + 7x^4 + x^5, & 
                                 \text{if \ $n=3$} \\
    x + 15x^2 + 57x^3 + 87x^4 + 57x^5 + 15x^6 + x^7, &  
                                  \text{if \ $n=4$} \\
    x + 31x^2 + 201x^3 + 551x^4 + 761x^5 + 551x^6 + 
                           201x^7 + 31x^8 + x^9, & 
                           \text{if \ $n=5$}.  
                \end{cases} \]
The symmetry of $f_{n,r}(x)$ is nearly obvious (see 
\cite[Proposition~2.5]{Ath14}); its unimodality was 
shown in \cite[Theorem~2.4.11]{Lin14} 
\cite[Corollary~4]{SZ16} (where the first reference 
offers a generalization and $q$-analog as well).
The connection between the generating polynomials
for $\exc$ and $\fexc$, respectively, over $\ZZ_r 
\wr \fS_n$ and $\dD_{n,r}$ are the formulas 
\begin{align}
A_{n,r}(x) & \ = \ \widetilde{{\rm E}}_r 
       (A^\fexc_{n,r}(x)),   \label{eq:athAnrfexc} \\
d_{n,r}(x) & \ = \ \widetilde{{\rm E}}_r (f_{n,r}(x)),     
                           \label{eq:athdnrfexc}
\end{align}
where $\widetilde{{\rm E}}_r: \RR[x] \to \RR[x]$ is 
the linear operator defined by setting 
$\widetilde{{\rm E}}_r (x^m) = x^{\lceil m/r \rceil}$
for $m \in \NN$; see \cite[Proposition~2.3]{Ath14} and 
its proof. Equation~(\ref{eq:athdnrfexc}) implies (see
\cite[Section~5]{Ath14}) that
\begin{align}
d^+_{n,r}(x) & \ = \ \sum_{w \in (\dD_{n,r})^b} 
       x^{\fexc(w)/r} ,   \label{eq:athdnr+fexc} \\
& \nonumber \\
d^{-}_{n,r}(x) & \ = \ \sum_{w \in \dD_{n,r} \sm   
 (\dD_{n,r})^b} x^{\lceil \frac{\fexc(w)}{r} \rceil},     
                              \label{eq:athdnr-fexc}
\end{align}
where $(\dD_{n,r})^b$ stands for the set of derangements 
$w \in \dD_{n,r}$ such that $\fexc(w)$ is divisible by 
$r$ (colored permutations with this property are 
sometimes called \emph{balanced}). 

The polynomial $f_{n,r}(x)$ is not $\gamma$-positive
for $r \ge 3$, the case $r=2$ being more interesting.
As discussed in \cite[Section~5]{Ath14}, it follows 
from Equations~(\ref{eq:athdnr+fexc}) 
and~(\ref{eq:athdnr-fexc}) that $d_{n,2}(x)$ and $f_{n,2}
(x)$ determine one another in a simple way. The 
$\gamma$-positivity of $f_{n,2}(x)$ (which is not 
real-rooted for all $n$) follows from the formula
in \cite[Proposition~3.4]{Mo13}
\begin{equation}
\label{eq:Mofn2}
f_{n,2}(x) \ = \ \sum_{k=0}^n \binom{n}{k} x^k 
                 (1+x)^{n-k} \, d_{n-k}(x) 
\end{equation}
and the $\gamma$-positivity of $d_n(x)$ 
(Mongelli~\cite[Conjecture~8.1]{Mo13} further 
conjectured that $f_{n,2}(x)$ is log-concave for all 
$n$). The corresponding $\gamma$-coefficients for the 
first few values of $n$ are given by
\[ f_{n,2} (x) \ = \ \begin{cases}
    x, & \text{if \ $n=1$} \\
    x(1 + x)^2 + x^2, & \text{if \ $n=2$} \\
    x(1 + x)^4 + 3x^2(1 + x)^2 + x^3, & 
                                 \text{if \ $n=3$} \\
    x(1 + x)^6 + 9x^2(1 + x)^4 + 6x^3(1 + x)^2 + x^4, &  
                                  \text{if \ $n=4$} \\
    x(1 + x)^8 + 23x^2(1 + x)^6 + 35x^3(1 + x)^4 + 
    10x^4(1 + x)^2 + x^5, &     \text{if \ $n=5$}.  
                \end{cases} \]
The following elegant combinatorial interpretation 
for these coefficients was found by Shin and 
Zeng~\cite{SZ16}.
\begin{theorem} \label{thm:SZfn2}
{\rm (Shin--Zeng \cite[Cor.~5]{SZ16})}
We have $f_{n,2}(x) = \sum_{i=0}^{\lfloor n/2 \rfloor} 
\hat{\xi}_{n,i} x^i (1+x)^{2n-2i}$, where 
$\hat{\xi}_{n,i}$ is equal to the number of 
permutations $w \in \fS_n$ with $i$ weak excedances 
and no double excedance. 
In particular, $f_{n,2}(-1) = (-1)^n$ for $n \ge 1$. 
\end{theorem}

\subsection{Coxeter--Narayana polynomials}
\label{subsec:nara}

Let $W$ be a finite Coxeter group with set of simple
generators $S$ and let $\ell_T: W \to \NN$ be the 
length (known as \emph{absolute length}) function 
with respect to the generating set $T = \{ wsw^{-1}: 
s \in S, w \in W\}$ of all reflections in $W$. Choose
a Coxeter element $c \in W$ and set  
\[ \nc_W = \nc_W(c) \ : = \ \{ w \in W: \, \ell_T(w)
             + \ell_T(w^{-1}c) = \ell_T(c) \}. \]
This set, endowed with a natural partial order which 
is graded by absolute length, was introduced by 
D.~Bessis~\cite{Be03} and independently by T.~Brady 
and C.~Watt~\cite{BW02} as the \emph{poset of 
noncrossing partitions} associated to $W$; see 
\cite[Chapter~2]{Arm09} for more information (the 
isomorphism type of this poset does not depend on 
the choice of $c$). The generating polynomial
\begin{equation}
\label{eq:defC(W,x)}
\cat(W, x) \ := \ \sum_{w \in \nc_W} x^{\ell_T(w)} 
\end{equation}
of the absolute length function over $\nc_W$ admits numerous 
algebraic, combinatorial and geometric interpretations
and plays an important role in the subject of Catalan 
combinatorics of Coxeter groups~\cite[Chapter~5]{Arm09}.
For the symmetric group $\fS_n$, its coefficients, known
as \emph{Narayana numbers}, refine the $n$th Catalan 
number; explicitly, we have: 
\begin{equation}
\label{eq:ABD-C(W,x)}
    \cat(W, x) \ = \ \begin{cases}
    {\displaystyle \sum_{i=0}^n \, \frac{1}{i+1} \, 
    \binom{n}{i}   
    \binom{n+1}{i} x^i}, & \text{if \ $W = \fS_{n+1}$} 
    \\ & \\
    {\displaystyle \sum_{i=0}^n \binom{n}{i}^2 x^i},
    & \text{if \ $W = \bB_n$} \\ & \\
    {\displaystyle \sum_{i=0}^n \binom{n}{i} \left( 
    \binom{n-1}{i} + \binom{n-2}{i-2} \right) x^i}, & 
    \text{if \ $W = \bB^e_n$}
    \end{cases} 
\end{equation}
where (to avoid confusion with our notation for the 
set of derangements in $\fS_n$) $\bB^e_n$ 
stands for the group of even-signed permutations of 
$\Omega_n$ (which is a Coxeter group of type $D_n$).

The symmetry of $\cat(W, x)$ is a simple consequence 
of the definition. The following result can be verified 
with case-by-case computations; it is due essentially 
to R.~Simion and D.~Ullman~\cite[Corollary~3.1]{SU91} 
\cite[Proposition~11.14]{PRW08} for the group $\fS_n$, 
to A.~Postnikov, V.~Reiner and L.~Williams 
\cite[Proposition~11.15]{PRW08} for the group $\bB_n$ 
and to M.~Gorsky~\cite{Go10} for the group $\bB^e_n$
(see also Remark~\ref{rem:gammaCC++}).
\begin{theorem}
\label{thm:C(W,x)}
The polynomial $\cat(W, x)$ is $\gamma$-positive for 
every finite Coxeter group $W$. Moreover, writing 
\begin{equation}
\label{eq:gammaC(W,x)}
  \cat(W, x) \ = \ \sum_{i=0}^{\lfloor n/2 \rfloor}   
  \gamma_i(W) \, x^i (1+x)^{n-2i},
\end{equation}
where $n$ is the rank of $W$, we have the explicit 
formulas

\begin{equation}
\label{eq:ABD-gammaC(W,x)}
    \gamma_i (W) \ = \ \begin{cases}
    {\displaystyle \frac{1}{i+1} \binom{n}{i,i,n-2i}}, 
               & \text{if \ $W = \fS_{n+1}$} \\ & \\
    {\displaystyle \binom{n}{i,i,n-2i}},
                    & \text{if \ $W = \bB_n$} \\ & \\
    {\displaystyle \frac{n-i-1}{n-1} 
    \binom{n}{i,i,n-2i}}, & \text{if \ $W = \bB^e_n$}
    \end{cases} 
\end{equation}
for $0 \le i \le \lfloor n/2 \rfloor$.
\end{theorem}

For an extension of the $\gamma$-positivity of $\cat
(W, x)$ to well-generated complex reflection groups, 
see \cite[Theorem~1.3]{Mu17}. Elegant $q$-analogs
of the $\gamma$-positivity of $\cat(\fS_n, x)$ are
given in~\cite[Theorem~1.2]{BP14} 
and~\cite[Theorem~4]{LF17}.

As is usual in Catalan combinatorics, the polynomial 
$\cat(W, x)$ has a positive analog, defined as  
\begin{equation}
\label{eq:defC+(W,x)}
\cat^+(W, x) \ := \ \sum_{J \subseteq S} 
                   \, (-x)^{|S \sm J|} \, \cat(W_J, x), 
\end{equation}
where $W_J$ is the standard parabolic subgroup of $W$
associated to $J \subseteq S$; we have the explicit
formulas:
\begin{equation}
\label{eq:ABD-C+(W,x)}
    \cat^+(W, x) \ = \ \begin{cases}
    {\displaystyle \sum_{i=0}^n \, \frac{1}{i+1} \, 
    \binom{n-1}{i} \binom{n}{i} x^i}, 
                    & \text{if \ $W = \fS_{n+1}$} \\
    & \\
    {\displaystyle \sum_{i=0}^n \binom{n-1}{i} 
    \binom{n}{i} x^i}, & \text{if \ $W = \bB_n$} \\ 
    & \\
    {\displaystyle \sum_{i=0}^n \left( \binom{n-2}{i} 
    \binom{n}{i} + \binom{n-2}{i-2} \binom{n-1}{i} 
    \right) x^i}, & \text{if \ $W = \bB^e_n$}.
    \end{cases} 
\end{equation}

This polynomial is not symmetric, except for the case
of the symmetric groups. However, the polynomial 
\begin{equation}
\label{eq:defC++(W,x)}
\cat^{++}(W, x) \ := \ \sum_{J \subseteq S} 
              \, (-1)^{|S \sm J|} \, \cat^+(W_J, x) 
\end{equation}
turns out to be symmetric and to have nonnegative 
coefficients for all $W$. This polynomial was first 
considered in an enumerative-geometric context 
in~\cite{AS12}, explained in Section~\ref{subsec:exa}, 
and more recently in a different context in~\cite{BR18}, 
where our notation comes from. For the symmetric and 
hyperoctahedral groups, the work~\cite{AS12} provides 
combinatorial interpretations for the coefficients of 
$\cat^{++}(W, x)$ and the corresponding 
$\gamma$-coefficients $\xi_i(W)$, in terms of 
noncrossing partitions with restrictions.
\begin{theorem} \label{thm:C++(W,x)}
{\rm (Athanasiadis--Savvidou~\cite{AS12})}
The polynomial $\cat^{++}(W, x)$ is $\gamma$-positive 
for every finite Coxeter group $W$. Moreover, writing 
\begin{equation}
\label{eq:gammaC++(W,x)}
  \cat^{++}(W, x) \ = \ \sum_{i=0}^{\lfloor n/2 \rfloor}   
  \xi_i(W) \, x^i (1+x)^{n-2i},
\end{equation}
where $n$ is the rank of $W$, we have $\xi_0 (W) = 0$
and the explicit formulas

\begin{equation}
\label{eq:ABD-gammaC++(W,x)}
    \xi_i (W) \ = \ \begin{cases}
    {\displaystyle \frac{1}{n-i+1} \binom{n}{i} 
    \binom{n-i-1}{i-1} }, 
                & \text{if \ $W = \fS_{n+1}$} \\ & \\
    {\displaystyle \binom{n}{i} \binom{n-i-1}{i-1} },
                    & \text{if \ $W = \bB_n$} \\ & \\
    {\displaystyle \frac{n-2}{i} \binom{2i-2}{i-1} 
    \binom{n-2}{2i-2}}, & \text{if \ $W = \bB^e_n$}
    \end{cases} 
\end{equation}
for $1 \le i \le \lfloor n/2 \rfloor$.
\end{theorem}

\begin{remark} \label{rem:gammaCC++}
\rm Equations~(\ref{eq:defC+(W,x)}) 
and~(\ref{eq:defC++(W,x)}) can be inverted to give 

\begin{align}
\cat(W, x) & \ = \ \sum_{J \subseteq S} 
                 \, x^{|S \sm J|} \, \cat^+(W_J, x), \\
\cat^+(W, x) & \ = \ \sum_{J \subseteq S} \, 
                     \cat^{++}(W_J, x), 
\end{align}
respectively. Hence,

\begin{align*}
\cat(W, x) & \ = \ \sum_{J \subseteq S} \, x^{|S \sm J|} 
\, \cat^+(W_J, x) \ = \ \sum_{J \subseteq S} \, 
x^{|S \sm J|} \, \sum_{I \subseteq J} \, 
                                   \cat^{++}(W_I, x) \\
& \ = \ \sum_{I \subseteq S} \, (x + 1)^{n - |I|} \,
\cat^{++}(W_I, x),
\end{align*}
where $n$ is the rank of $W$. Setting
\[ \gamma(W, x) \ = \ \sum_{i=0}^{\lfloor n/2 \rfloor} 
   \gamma_i(W) x^i, \ \ \ \ \ \ \ \ 
\xi(W, x) \ = \ \sum_{i=0}^{\lfloor n/2 \rfloor}  
\xi_i(W) x^i, \]
the previous formula for $\cat(W, x)$ translates into
the equation
\begin{equation}
\label{eq:gamma-xi(W,x)}
\gamma(W, x) \ = \ \sum_{J \subseteq S} \xi(W_J, x).
\end{equation}
Alternatively, given the geometric interpretations 
of $\cat(W, x)$ and $\cat^{++}(W, x)$ discussed in 
Section~\ref{subsec:exa}, 
Equation~(\ref{eq:gamma-xi(W,x)}) is a special case  
of \cite[Corollary~5.5]{Ath12} 
\cite[Equation~(6)]{Ath16a}; see also 
Theorem~\ref{thm:ath-gammaxi} in the sequel. 

In particular, the $\gamma$-positivity statement of 
Theorem~\ref{thm:C++(W,x)} is stronger than that of
Theorem~\ref{thm:C(W,x)}.
\qed
\end{remark}

\medskip
The numbers $\gamma_1(W)$ and $\xi_1(W)$ are equal 
to the number of nonsimple reflections in $W$ and 
the number of reflections in $W$ which do not belong 
to any proper standard parabolic subgroup, 
respectively; see the first remark in 
\cite[Section~5]{AS12}.

\begin{problem}
Find a proof of the $\gamma$-positivity of $\cat
(W, x)$ and $\cat^{++}(W, x)$, as well as algebraic
or combinatorial interpretations for the numbers 
$\gamma_i(W)$ and $\xi_i(W)$, which do not depend 
on the classification of finite Coxeter groups. 
\end{problem}

\subsection{Polynomials arising from enriched poset
partitions}
\label{subsec:enriched}

Let $(\pP,\omega)$ be a labeled poset with $n$ 
elements, as in Section~\ref{subsubsec:posetEuler},
and recall that $\Omega_m := \{1, -1, 2, -2,\dots,m, 
-m\}$. For $m \in \NN$, denote by 
$\Omega'_{\pP,\omega}(m)$ the number of maps 
$f: \pP \to \Omega_m$ which are such that for 
all $x, y \in \pP$ with $x <_\pP y$:
\begin{itemize}
\item[$\bullet$] $|f(x)| \le |f(y)|$, 
\item[$\bullet$] $|f(x)| = |f(y)| \Rightarrow
                 f(x) \le f(y)$,  
\item[$\bullet$] $f(x) = f(y) > 0 \Rightarrow
                 \omega(x) < \omega(y)$, 
\item[$\bullet$] $f(x) = f(y) < 0 \Rightarrow
                  \omega(x) > \omega(y)$,  
\end{itemize}
where $\Omega'_{\pP,\omega}(0) := 0$. These 
maps are called \emph{enriched 
$(\pP,\omega)$-partitions}; their theory was 
developed by Stembridge~\cite{Ste97}, in analogy 
with Stanley's theory of $(\pP,\omega)$-partitions. 
The function $\Omega'_{\pP,\omega}(m)$, studied 
in \cite[Section~4]{Ste97}, turns out to be a 
polynomial in $m$ of degree at most $n$ and hence
\begin{equation} \label{eq:defA'}
  \sum_{m \ge 0} \Omega'_{\pP,\omega}(m) x^m \ = \   
  \frac{\aA'_{\pP,\omega}(x)}{(1-x)^{n+1}}
\end{equation}
for some polynomial $\aA'_{\pP,\omega}(x)$ of degree
at most $n$. As pointed out by Stembridge to the 
author, the more explicit formula 
\begin{equation} \label{eq:SteA'}
  \aA'_{\pP,\omega}(x) \ = \ x \sum_{\varepsilon: 
  \pP \to \{-1,1\}} \aA_{\pP, \varepsilon \omega} (x)
\end{equation}
follows from \cite[Theorem~3.6]{Ste97}, where 
$\aA_{\pP, \varepsilon \omega} (x)$ is the ordinary
$(\pP, \varepsilon \omega)$-Eulerian polynomial 
defined as in Section~\ref{subsubsec:posetEuler},
with labels taken from the totally ordered set
$\ZZ$. 

The following result is a restatement of 
\cite[Theorem~4.1]{Ste97}; it appeared several 
years before the work on $\gamma$-positivity of 
Br\"and\'en~\cite{Bra04, Bra08} and Gal~\cite{Ga05}.
A \emph{peak} of a permutation $w \in \fS_n$ is any 
index $2 \le i \le n-1$ such that $w(i-1) < w(i) > 
w(i+1)$. 
\begin{theorem} \label{thm:SteEnriched}
{\rm (Stembridge~\cite{Ste97})}
For every $n$-element labeled poset $(\pP,\omega)$
\begin{equation} \label{eq:SteEnriched}
  \aA'_{\pP,\omega} (x) \ = \ 
  \sum_{i=0}^{\lfloor (n-1)/2 \rfloor} p_{n,i} \, 
             2^{2i+1} \, x^{i+1} (1+x)^{n-1-2i},
\end{equation}
where $p_{n,i}$ is the number of permutations $w \in 
\lL(\pP, \omega)$ (see 
Definition~\ref{def:posetEuler}) with $i$ peaks.

In particular, $\aA'_{\pP,\omega} (x)$ is 
$\gamma$-positive. 
\end{theorem}

This theorem reduces to the $\gamma$-positivity of 
the Eulerian polynomial $A_n(x)$ when $\pP$ is the 
$n$-element antichain, since in this case 
$\aA'_{\pP,\omega} (x) = x \, 2^n A_n(x)$.

\subsection{Polynomials arising from poset homology}
\label{subsec:homo}

A $\gamma$-positivity theorem, due to S.~Linusson,
J.~Shareshian and M.~Wachs~\cite{LSW12}, which 
captures as special cases some of the basic examples 
we have seen so far, comes from the study of the 
homology of the Rees product of posets, a construction
due to A.~Bj\"orner and V.~Welker~\cite{BW05}.
Given finite graded posets $\pP$ and $\qQ$ with rank 
functions $\rho_\pP$ and $\rho_\qQ$, respectively, 
their \emph{Rees product} is defined
in~\cite{BW05} as $\pP \ast \qQ = \{ (p, q) \in \pP 
\times \qQ: \rho_\pP (p) \ge \rho_\qQ (q) \}$,
partially ordered by setting $(p_1, q_1) \preceq 
(p_2, q_2)$ if all of the following conditions are 
satisfied:
\begin{itemize}
\itemsep=0pt
\item[$\bullet$] $p_1 \preceq p_2$ holds in $\pP$,
\item[$\bullet$] $q_1 \preceq q_2$ holds in $\qQ$ and
\item[$\bullet$] $\rho_\pP (p_2) - \rho_\pP (p_1) \ge 
\rho_\qQ (q_2) - \rho_\qQ (q_1)$.
\end{itemize}

\noindent
The poset $\pP \ast \qQ$ is graded with rank function 
given by $\rho(p,q) = \rho_\pP (p)$ for $(p,q) \in \pP 
\ast \qQ$. 

To state the result of~\cite{LSW12}, we need 
to recall a few more definitions. Let $\pP$ be a finite 
graded poset, as before, and assume further that $\pP$ 
is bounded, with minimum element $\hat{0}$ and maximum 
element $\hat{1}$, and that it has rank $n+1$. For 
$S \subseteq [n]$, we denote by $a_\pP (S)$ the number of 
maximal chains of the rank-selected subposet 
\begin{equation} \label{eq:P_S}
\pP_S \ := \ \{ p \in \pP: \rho_\pP (p) \in S \} \, 
\cup \, \{\hat{0}, \hat{1} \}
\end{equation}
of $\pP$ and set
\begin{equation} \label{eq:b(S)1}
b_\pP (S) \ := \ \sum_{T \subseteq S} (-1)^{|S-T|}
\, a_\pP (T).
\end{equation}
The numbers $a_\pP (S)$ and $b_\pP (S)$ are important
enumerative invariants of $\pP$; see 
\cite[Section~3.13]{StaEC1}. The numbers $b_\pP(S)$ 
are nonnegative, if $\pP$ is Cohen--Macaulay over some 
field, and afford a simple combinatorial interpretation, 
if $\pP$ admits an $R$-labeling. The number $(-1)^{n-1}
b_\pP ([n])$ is the \emph{M\"obius number} of the poset
$\bar{\pP}$, obtained from $\pP$ by removing $\hat{0}$
and $\hat{1}$; it is denoted by $\mu(\bar{\pP})$. The 
posets obtained from $\pP$ by removing $\hat{0}$ or 
$\hat{1}$ are denoted by $\pP_{-}$ or $\pP^-$, 
respectively. The poset whose Hasse diagram is a 
complete $x$-ary tree of height $n$, rooted at the 
minimum element, is denoted by $\tT_{x,n}$. For the 
notion of EL-shellability, the reader is referred 
to \cite[Section~2]{LSW12} and references therein. 
The following theorem is a restatement 
of~\cite[Corollary~3.8]{LSW12}. 

\begin{theorem} \label{thm:LSWrees}
{\rm (Linusson--Shareshian--Wachs~\cite{LSW12})} 
Let $\pP$ be a finite bounded graded poset of rank 
$n+1$. If $\pP$ is EL-shellable, then 
  \begin{align} \label{eq:lsw1}  
    \mu \, ((\pP^- \ast \tT_{x,n})_{-}) 
    \ = & \
    \sum_{S \in \Stab ([n-1])} b_\pP ([n] \sm S) 
    \, x^{|S|} (1+x)^{n-2|S|} \ \ + \\
    & \nonumber \\
    & \ \sum_{S \in \Stab ([n-2])} 
    b_\pP ([n-1] \sm S) \, x^{|S|+1} (1+x)^{n-1-2|S|} 
    \nonumber
  \end{align}
and 
  \begin{align} \label{eq:lsw2}
    \mu \, (\bar{\pP} \ast \tT_{x,n-1}) 
    \ = & \
    \sum_{S \in \Stab ([2, n-2])} b_\pP ([n-1] \sm S) 
    \, x^{|S|+1} (1+x)^{n-2-2|S|} \ \ + \\
    & \nonumber \\
    & \ \sum_{S \in \Stab ([2, n-1])} b_\pP ([n] \sm S) 
    \, x^{|S|} (1+x)^{n-1-2|S|} 
      \nonumber
  \end{align}
for every positive integer $x$.
\end{theorem}

This theorem turns out to be valid without the 
assumption of EL-shellability; see~\cite{Ath17}. 
When $\pP^{-}$ has a maximum element, the first 
and the second summand of the right-hand sides of 
Equations~(\ref{eq:lsw1}) and~(\ref{eq:lsw2}), 
respectively, vanishes and hence the left-hand 
sides are $\gamma$-positive polynomials in $x$, 
provided $\pP$ is Cohen--Macaulay over some field.
\begin{example} \rm
Suppose that $\pP^{-}$ is the Boolean lattice of 
subsets of the set $[n]$, ordered by inclusion.
Then, the left-hand sides of 
Equations~(\ref{eq:lsw1}) and~(\ref{eq:lsw2}) are
equal to $xA_n(x)$ and $d_n(x)$, respectively 
(this is the content of~\cite[Equation~(3.1)]{SW09} 
in the former case, and is implicit in~\cite{SW09} 
in the latter). The number $b_\pP (S)$ is known 
to count permutations $w \in \fS_n$ such that 
$\Des(w) = S$ \cite[Corollary~3.13.2]{StaEC1}. Thus, 
Theorem~\ref{thm:LSWrees} reduces to the 
$\gamma$-positivity of $A_n(x)$ and $d_n(x)$ 
(specifically, to the third interpretations given 
for $\gamma_{n,i}$ and $\xi_{n,i}$ in 
Theorems~\ref{thm:FSSAn} and~\ref{thm:ASdn}, 
respectively) in this case. 
\qed
\end{example}

The $p=1$ specializations of Theorems~\ref{thm:SWAn} 
and~\ref{thm:SWdn} can be viewed as the special 
case of Theorem~\ref{thm:LSWrees} in which $\pP^{-}$ 
is the lattice of subspaces of an $n$-dimensional 
vector space over a field with $q$ elements. This 
follows from Equations~(1.3) and~(1.4) 
in~\cite{LSW12} and from the known interpretation
\cite[Theorem~3.13.3]{StaEC1} of $b_\pP (S)$ for 
this poset. Theorem~\ref{thm:athdnr} is derived in
\cite[Section~4]{Ath14} from the special case of 
Theorem~\ref{thm:LSWrees} in which $\pP^{-}$ is the
set of $r$-colored subsets of $[n]$, ordered by 
inclusion.
 
Similarly, Theorem~\ref{thm:LSWmulti} can be viewed 
as the special case of Theorem~\ref{thm:LSWrees} in 
which $\pP^{-}$ is the product of chains whose 
lengths are the parts of $\lambda$; see 
\cite[Section~5]{LSW12}. The proof of 
Theorem~\ref{thm:LSWmulti} given there, however, 
involves a different approach and thus, it 
would be interesting to prove directly that the 
left-hand sides of Equations~(\ref{eq:lsw1}) 
and~(\ref{eq:lsw2}) are equal to those of 
Equations~(\ref{eq:LSWmulti1}) 
and~(\ref{eq:LSWmulti2}), respectively, in this 
case. 
\begin{remark} \rm
Theorem~\ref{thm:LSWmulti} can be applied, more 
generally, when $\pP$ is the distributive lattice 
of order ideals of an $n$-element poset $\qQ$, 
since then $b_\pP (S)$ counts linear extensions 
of $\qQ$ with descent set equal to $S$
\cite[Theorem~3.13.1]{StaEC1}. Are there 
combinatorial interpretations for the coefficients 
of the left-hand sides of Equations~(\ref{eq:lsw1}) 
and~(\ref{eq:lsw2}) which reduce to those of 
Theorem~\ref{thm:LSWmulti} when $\qQ$ is a disjoint
union of chains of cardinalities equal to the parts 
of $\lambda$? 
\end{remark}

\subsection{Polynomials arising from symmetric 
functions} \label{subsec:sym}

Let $\bx = (x_1, x_2,\dots)$ be a sequence of 
commuting independent indeterminates (to avoid
confusion with this notation, in this section 
we consider polynomials in the variable $t$, 
rather than $x$). For integer partitions 
$\lambda$, consider the polynomials $P_\lambda
(t)$ and $R_\lambda(t)$ which are defined by 
the equations

\begin{equation} \label{eq:Bre1}
\frac{(1-t) H(\bx; z)} {H(\bx; tz) - tH(\bx; z)} = 
\frac{{\displaystyle \sum_{k \ge 0} h_k(\bx) z^k}}
   {{\displaystyle 1 - \sum_{k \ge 2} \, (t + t^2 + \cdots 
    + t^{k-1}) \, h_k(\bx) z^k}} = \sum_\lambda \, 
    P_\lambda(t) s_\lambda (\bx) \, z^{|\lambda|} 
\end{equation}
and

\begin{equation} \label{eq:Bre2}
\frac{1-t} {H(\bx; tz) - tH(\bx; z)} = 
\frac{1} {{\displaystyle 1 - \sum_{k \ge 2} \, (t + t^2 
    + \cdots + t^{k-1}) \, h_k(\bx) z^k}} = \sum_\lambda 
    \, R_\lambda(t) s_\lambda (\bx) \, z^{|\lambda|}, 
\end{equation}

\noindent
where $H(\bx; z) = \sum_{n \ge 0} h_n(\bx) z^n$ is the 
generating function for the complete homogeneous 
symmetric functions in $\bx$ and $\lambda$ ranges over
all partitions (including the one without any parts).
The left-hand sides of Equations~(\ref{eq:Bre1}) 
and~(\ref{eq:Bre2}) were first considered by Stanley 
\cite[Propositions~12 and~13]{Sta89} and, since then,
they have appeared in various algebraic-geometric and 
combinatorial contexts; see the relevant discussions 
in \cite[Section~7]{SW10} \cite[Section~4]{LSW12} 
\cite{SW16b}, as well as Section~\ref{sec:gen} in the 
sequel. Stanley~\cite{Sta89} observed that for $\lambda 
\vdash n$, the polynomials $P_\lambda(t)$ and 
$R_\lambda(t)$ have nonnegative and symmetric 
coefficients, with centers of symmetry $(n-1)/2$ and 
$n/2$, respectively, which satisfy  
\begin{equation} \label{eq:refAn}
\sum_{\lambda \vdash n} f^\lambda P_\lambda(t) \ = \ 
   A_n(t)
\end{equation}
and
\begin{equation} \label{eq:refdn}
\sum_{\lambda \vdash n} f^\lambda R_\lambda(t) \ = \ 
   d_n(t),
\end{equation}
where $f^\lambda$ is the number of standard Young 
tableaux of shape $\lambda$. This shows that the 
coefficients of $z^n$ in the functions appearing in
Equations~(\ref{eq:Bre1}) and~(\ref{eq:Bre2}) 
correspond, via Frobenius characteristic, to graded
$\fS_n$-representations whose graded dimensions are
equal to $A_n(t)$ and $d_n(t)$, respectively. Thus,
the left-hand sides of Equations~(\ref{eq:Bre1}) 
and~(\ref{eq:Bre2}) can be considered as 
representation-theoretic analogs of $A_n(t)$ and
$d_n(t)$, respectively. Combinatorial interpretations
for the coefficients of $P_\lambda(t)$ and 
$R_\lambda(t)$ were provided by 
Stembridge~\cite[Section~4]{Ste92}; a different one 
for $P_\lambda(t)$ can be derived as a special case 
of \cite[Theorem~6.3]{SW16a}.

The unimodality of $P_\lambda(t)$ and $R_\lambda(t)$,
which refines that of $A_n(t)$ and $d_n(t)$, was 
proved by Stanley~\cite[Proposition~12]{Sta89} and
Brenti~\cite[Corollary~1]{Bre90}, respectively.
Given that $A_n(t)$ and $d_n(t)$ are 
$\gamma$-positive, it is natural to ask whether 
$P_\lambda(t)$ and $R_\lambda(t)$ have the same 
property. An affirmative answer follows from 
explicit formulas which refine 
Theorems~\ref{thm:FSSAn} and~\ref{thm:ASdn}.
These formulas are essentially equivalent to
the following unpublished result of Gessel, 
stated without proof in \cite[Section~4]{LSW12} 
\cite[Section~7]{SW10} \cite[Section~3]{SW17}.
We write $E(\bx; z) = \sum_{n \ge 0} e_n(\bx) 
z^n$ for the generating function for the 
elementary symmetric functions in $\bx$. For a
map $w: [n] \to \ZZ_{>0}$ we set $\bx_w := 
x_{w(1)} x_{w(2)} \cdots x_{w(n)}$ and, as 
usual, define by $\Asc(w) := [n-1] \sm \Des(w)$ 
and $\Des(w) := \{ i \in [n-1]: w(i) > w(i+1) 
\}$ the set of \emph{ascents} and \emph{descents} 
of $w$ and denote by $\asc(w)$ and $\des(w)$ the 
cardinalities of these sets, respectively.
\begin{theorem} \label{thm:Ge}
{\rm (Gessel, unpublished)} 
We have

\begin{equation} \label{eq:Ge1}
\frac{(1-t) H(\bx; z)} {H(\bx; tz) - tH(\bx; z)} 
\ = \ 1 \, + \, \sum_{n \ge 1} z^n 
\sum_{\scriptsize \begin{array}{c} w: [n] \to \ZZ_{>0}
    \\ \Des(w) \in \Stab ([n-2]) \end{array}}
t^{\des(w)} (1+t)^{n-1 - 2\des(w)} \, \bx_w,
\end{equation}
\begin{equation} \label{eq:Ge2}
\frac{1-t} {H(\bx; tz) - tH(\bx; z)} 
\ = \ 1 \, + \, \sum_{n \ge 2} z^n 
\sum_{\scriptsize \begin{array}{c} w: [n] \to \ZZ_{>0}
    \\ \Des(w) \in \Stab ([2,n-2]) \end{array}}
t^{\des(w)+1} (1+t)^{n-2 - 2\des(w)} \, \bx_w,
\end{equation}
\begin{equation} \label{eq:Ge3}
\frac{(1-t) E(\bx; z)} {E(\bx; tz) - tE(\bx; z)} 
\ = \ 1 \, + \, \sum_{n \ge 1} z^n 
\sum_{\scriptsize \begin{array}{c} w: [n] \to \ZZ_{>0}
    \\ \Asc(w) \in \Stab ([n-2]) \end{array}}
t^{\asc(w)} (1+t)^{n-1 - 2\asc(w)} \, \bx_w
\end{equation}
and

\begin{equation} \label{eq:Ge4}
\frac{1-t} {E(\bx; tz) - tE(\bx; z)} 
\ = \ 1 \, + \, \sum_{n \ge 2} z^n 
\sum_{\scriptsize \begin{array}{c} w: [n] \to \ZZ_{>0}
    \\ \Asc(w) \in \Stab ([2,n-2]) \end{array}}
t^{\asc(w)+1} (1+t)^{n-2 - 2\asc(w)} \, \bx_w.
\end{equation}
\end{theorem}

\smallskip
Equations~(\ref{eq:Ge1}) and~(\ref{eq:Ge2}) can  
be shown to be equivalent to~(\ref{eq:Ge3}) 
and~(\ref{eq:Ge4}), respectively, via an application 
of the standard involution $\omega$ on symmetric 
functions. Perhaps not surprisingly, 
Equations~(\ref{eq:Ge3}) and~(\ref{eq:Ge4}) can be 
deduced \cite[Corollary~4.1]{Ath17} as special cases 
of an equivariant version of Theorem~\ref{thm:LSWrees}; 
see \cite[Theorem~1.2]{Ath17} and 
Theorem~\ref{thm:Athrees} in the sequel. A direct 
combinatorial proof of~(\ref{eq:Ge3}) will be 
sketched in Section~\ref{subsec:hopping}. 
\begin{corollary} \label{cor:PRgamma}
For $\lambda \vdash n$,
\begin{equation}
\label{eq:Pgamma}
P_\lambda (t) \ = \ \sum_{i=0}^{\lfloor (n-1)/2 \rfloor} 
\gamma^\lambda_i \, t^i (1+t)^{n-1-2i}
\end{equation} 
and
\begin{equation}
\label{eq:Rgamma}
R_\lambda (t) \ = \ \sum_{i=0}^{\lfloor n/2 \rfloor} 
\xi^\lambda_i \, t^i (1+t)^{n-2i},
\end{equation} 
where

\begin{itemize}
\item[$\bullet$] $\gamma^\lambda_i$ is the number 
of tableaux $Q \in \SYT(\lambda)$ for which 
$\Des(Q) \in \Stab([n-2])$ has $i$ elements, and
\item[$\bullet$] $\xi^\lambda_i$ is the number 
of tableaux $Q \in \SYT(\lambda)$ for which 
$\Des(Q) \in \Stab([2,n-2])$ has $i-1$ elements.
  \end{itemize}
\end{corollary}
In particular, $P_\lambda (t)$ and $R_\lambda (t)$
are $\gamma$-positive for every partition $\lambda$.
\begin{proof}
For $B \subseteq [n-1]$, we consider the skew hook 
shape whose row lengths, read from bottom to top, 
form the composition of $n$ corresponding to $B$ and
denote by $s_B(\bx)$ the skew Schur function
associated to this shape. As in the proof of 
\cite[Corollary~3.2]{SW17}, we note that, because of 
the obvious correspondence between monomials $\bx_w$
and semistandard Young tableaux of skew hook shape, 
Equation~(\ref{eq:Ge1}) may be rewritten as
\begin{equation} \label{eq:Ge1b}
\frac{(1-t) H(\bx; z)} {H(\bx; tz) - tH(\bx; z)} 
\ = \ 1 \, + \, \sum_{n \ge 1} z^n 
\sum_{i=0}^{\lfloor (n-1)/2 \rfloor} \gamma_{n,i} 
(\bx) \, t^i (1+t)^{n-1-2i}
\end{equation}
where
\begin{equation} \label{eq:Ge1bb}
\gamma_{n,i} (\bx) \ = \sum_{\scriptsize 
\begin{array}{c} B \in \Stab ([n-2]) \\ |B| = i 
\end{array}} s_B(\bx).
\end{equation}

Expanding $s_B(\bx)$ by the well known rule
\begin{equation} \label{eq:skewhookexpand}
s_B(\bx) \ = \ \sum_{\lambda \vdash n} 
\, |\{ Q \in \SYT(\lambda): \Des(Q) = B\}| 
\cdot s_\lambda(\bx) 
\end{equation}
in the basis of Schur functions and comparing 
the coefficients of $s_\lambda(\bx)$ on the two 
sides of~(\ref{eq:Ge1b}) gives the desired 
formula for $P_\lambda(t)$. A similar argument 
works for $R_\lambda(t)$.
\end{proof}

The author is not aware of any results in the 
literature concerning the roots of $P_\lambda(t)$ 
and $R_\lambda(t)$.

\subsection{Polynomials arising from trees}
\label{subsec:other}

This section discusses two more instances 
of $\gamma$-positivity in combinatorics. For 
further examples, see \cite{Bra08, Li17} and 
Sections~\ref{subsec:exa} 
and~\ref{subsec:gmethods}. 

A tree $T$ on the vertex set $[n]$ is called 
\emph{rooted} if it has a distinguished vertex, 
called \emph{root}. An edge $\{a, b\}$, with 
$a < b$, of such a tree $T$ is called a 
\emph{descent} if the unique path in $T$ which
connects $a$ with the root passes through $b$.
Let $t_n(x)$ be the polynomial in which the 
coefficient of $x^k$ equals the number of 
rooted trees on the vertex set $[n]$ with 
exactly $k$ descents. The following elegant 
$x$-analog 
\begin{equation} \label{eq:ERtrees}
t_n(x) \ = \ \prod_{i=1}^{n-1} \left( (n-i) + ix
\right)
\end{equation}
of Cayley's formula, which is a specialization of 
results of \"O.~E\u{g}ecio\u{g}lu and J.~Remmel
\cite{ER86}, implies the $\gamma$-positivity 
of $t_n(x)$. Several explicit combinatorial 
interpretations for the corresponding 
$\gamma$-coefficients were found by R.~Gonz\'alez 
D'Le\'on~\cite{Gon16}.

A tree $T$ on the vertex set $[n]$ is called 
\emph{alternating}, or \emph{intransitive} 
\cite{Pos97}, if every vertex of $T$ is either
less than all its neighbors, in which case it 
is called a \emph{left vertex}, or greater 
than all its neighbors, in which case it is 
called a \emph{right vertex}. Let $g_n(x)$ be
the polynomial in which the coefficient of 
$x^k$ equals the number of alternating trees 
on the vertex set $[n]$ with exactly $k$ left 
vertices. These polynomials were considered
in \cite{Ge96, Pos97}; for the first few 
values of $n$, we have
\[  g_n (x) \ = \ \begin{cases}
    1, & \text{if \ $n=1$} \\
    x, & \text{if \ $n=2$} \\
    x + x^2, & \text{if \ $n=3$} \\
    x + 5x^2 + x^3, &  \text{if \ $n=4$} \\
    x + 17x^2 + 17x^3 + x^4, & \text{if \ $n=5$} \\ 
    x + 49x^2 + 146x^3 + 49x^4 + x^5, & 
                               \text{if \ $n=6$} \\
    x + 129x^2 + 922x^3 + 922x^4 + 129x^5 + x^6, & 
                               \text{if \ $n=7$}.  
                \end{cases} \]

\begin{theorem} \label{thm:GGT17}
{\rm 
(Gessel--Griffin--Tewari~\cite[Theorem~5.9]{GGT17})} 
The polynomial $g_n(x)$ is $\gamma$-positive for
every positive integer $n$. 
\end{theorem}

The proof, which uses symmetric functions, yields 
a combinatorial formula for the 
$\gamma$-coefficients involving certain plane binary 
trees. An explicit combinatorial interpretation has 
also been found by V.~Tewari~\cite{Te17}. 

\section{Gamma-positivity in geometry}
\label{sec:geom}

This section discusses the main geometric contexts 
in which $\gamma$-positivity occurs. Many of the 
examples treated in Section~\ref{sec:comb} reappear 
here as interesting special cases of more general 
phenomena. Familiarity with basic notions on 
simplicial complexes and their geometric 
realizations is assumed; see \cite{Bj95} 
\cite[Chapter~2]{BP15} \cite{Koz08} 
\cite[Lecture~1]{Wa07} for detailed expositions. 
All simplicial complexes considered here are 
assumed to be finite. To keep our discussion as 
elementary as possible, we work with 
triangulations of spheres and balls, rather than 
with the more general classes of homology spheres 
and homology balls. 

We first review some basic definitions and 
background related to the face enumeration of 
simplicial complexes. For $i \ge 0$ we denote 
by $f_{i-1}(\Delta)$ the number of 
$(i-1)$-dimensional faces of a simplicial complex 
$\Delta$, where (unless $\Delta$ is empty) 
$f_{-1} (\Delta) := 1$.
\begin{definition} \label{def:h-poly} 
The $h$-polynomial of a simplicial complex 
$\Delta$ of dimension $n-1$ is defined as
\begin{equation} \label{eq:defh-poly} 
h(\Delta, x) \ := \ \sum_{i=0}^n f_{i-1}(\Delta) 
x^i (1-x)^{n-i} \ := \ \sum_{i=0}^n h_i(\Delta) x^i. 
\end{equation}
The sequence $h(\Delta) = (h_0(\Delta), 
h_1(\Delta),\dots,h_n(\Delta))$ is the $h$-vector 
of $\Delta$.
\end{definition}

The polynomial $h(\Delta, x)$ can be thought of as 
an $x$-analog of the number $f_{n-1} (\Delta)$ of 
facets (meaning, faces of maximum dimension $n-1$) 
of $\Delta$, since $h(\Delta, 1) = f_{n-1}(\Delta)$. 
We refer the 
reader to \cite{StaCCA} for the significance of the 
$h$-polynomial, as well as for important algebraic 
and combinatorial interpretations which are valid 
for special classes of simplicial complexes. 
\begin{example} \rm
The simplicial complex $\Delta$ shown in 
Figure~2 triangulates a 
two-dimensional simplex with eight vertices, 
fifteen edges and eight two-dimensional faces. 
Thus, $f_{-1} (\Delta) = 1$, $f_0(\Delta) = 8$, 
$f_1(\Delta) = 15$ and $f_2(\Delta) = 8$ and 
we may compute that $h (\Delta, x) = (1-x)^3 + 
8x(1-x)^2 + 15x^2 (1-x) + 8x^3 = 1 + 5x + 
2x^2$. \qed
\end{example} 

\begin{figure}
\begin{center}
\begin{tikzpicture}[scale=0.85]
\label{fg:two-dim}

   \draw(0,0) node(1){$\bullet$};
   \draw(2,0) node(2){$\bullet$};
   \draw(1,1) node(3){$\bullet$};
   \draw(2,2) node(4){$\bullet$};
   \draw(2,4) node(5){$\bullet$};
   \draw(3,0) node(6){$\bullet$};
   \draw(3,2) node(7){$\bullet$};
   \draw(4,0) node(8){$\bullet$};

   \draw(0,0) -- (4,0) -- (2,4) -- (0,0);
   \draw(0,0) -- (2,2) -- (3,2) -- (2,0);
   \draw(2,0) -- (1,1) -- (2,4);
   \draw(2,0) -- (2,4);
   \draw(3,0) -- (3,2);

\end{tikzpicture}
\caption{A two-dimensional simplicial complex}
\end{center}
\end{figure}
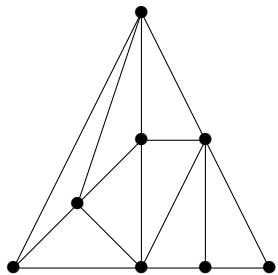

The relevance of the $h$-polynomial to the 
symmetry and unimodality of real polynomials stems 
from the following theorem, which combines important 
results by V.~Klee, G.~Reisner and R.~Stanley 
\cite{Kl64, Rei76, Sta75, Sta80} in geometric 
combinatorics; see~\cite{StaCCA} for more 
information. The first two parts hold 
more generally for the classes of Cohen--Macaulay 
and Eulerian, respectively, simplicial complexes.

\begin{theorem} \label{thm:KRS}
{\rm (Klee, Reisner, Stanley)}
Let $\Delta$ be $(n-1)$-dimensional.

\begin{itemize}
\item[{\rm (a)}]
  The polynomial $h(\Delta, x)$ has nonnegative 
  coefficients, i.e., we have $h_i(\Delta) \ge 0$ for
  $0 \le i \le n$, if $\Delta$ triangulates a ball or 
  a sphere.
\item[{\rm (b)}]
  The polynomial $h(\Delta, x)$ is symmetric, i.e., we 
  have $h_i(\Delta) = h_{n-i}(\Delta)$ for $0 \le i \le 
  n$, if $\Delta$ triangulates a sphere.
\item[{\rm (c)}] 
  The polynomial $h(\Delta, x)$ is unimodal, i.e., we 
  have 
  \begin{equation} \label{eq:GLBC}
    1 = h_0(\Delta) \le h_1(\Delta) \le \cdots \le 
                  h_{\lfloor n/2 \rfloor}(\Delta),
  \end{equation}
  if $\Delta$ is the boundary complex of a simplicial 
  polytope. 
\end{itemize}
\end{theorem}

\smallskip
Part (c), proved by R.~Stanley~\cite{Sta80}, is known 
as the \emph{Generalized Lower Bound Theorem} for 
simplicial polytopes, since the inequalities $h_i
(\Delta) \ge h_{i-1}(\Delta)$ impose a lower bound 
for each face number $f_{i-1}(\Delta)$ in terms of 
the numbers $f_{j-1}(\Delta)$ for $1 \le j < i$. 

\subsection{Flag triangulations of spheres}
\label{subsec:gal}

Let $\Delta$ be an (abstract) simplicial complex 
on the vertex set $V$. We say that $\Delta$ is 
\emph{flag} if we have $F \in \Delta$ for every $F 
\subseteq V$ for which all two-element subsets of 
$F$ belong to $\Delta$ and refer to \cite{Ath11} 
\cite[Section~III.4]{StaCCA} for a glimpse of the 
combinatorial properties of this fascinating class 
of simplicial complexes. The complex shown in 
Figure~2, for example, is flag. 

The following major open problem in geometric 
combinatorics was posed by \'S.~Gal~\cite{Ga05}
who disproved a more optimistic conjecture,
claiming that $h(\Delta, x)$ is real-rooted for 
every flag triangulation of a sphere.
\begin{conjecture} \label{conj:Gal}
{\rm (Gal~\cite[Conjecture~2.1.7]{Ga05})} 
The polynomial $h(\Delta, x)$ is $\gamma$-positive 
for every flag simplicial complex $\Delta$ which 
triangulates a sphere. 
\end{conjecture}
This conjecture extends an earlier conjecture of
R.~Charney and M.~Davis~\cite{CD95} on the sign 
of $h(\Delta, -1)$; it is also open for the more 
restrictive class of flag boundary complexes of 
simplicial polytopes.
\begin{example} \rm
Consider a triangulation $\Delta$ of the  
one-dimensional sphere with $m$ vertices (cycle 
of length $m$). We have $f_{-1} (\Delta) = 1$ 
and $f_0(\Delta) = f_1(\Delta) = m$ and hence 
$h (\Delta, x) = (1-x)^2 + mx(1-x) + mx^2 = 1 + 
(m-2)x + x^2$ is $\gamma$-positive if and only 
if $m \ge 4$. Note that this is exactly the 
necessary and sufficient condition for $\Delta$
to be flag. \qed
\end{example} 

Conjecture~\ref{conj:Gal} can be viewed as a 
Generalized Lower Bound Conjecture for flag 
triangulations of spheres $\Delta$. To be more 
precise, we may write
\begin{equation} \label{eq:def-gammavector} 
h(\Delta, x) \ = \ \sum_{i=0}^{\lfloor n/2 \rfloor} 
\gamma_i(\Delta) x^i (1+x)^{n-2i},
\end{equation}
where $n-1 = \dim(\Delta)$ and $\gamma_0 (\Delta) = 
1$. The vector $\gamma(\Delta) := (\gamma_0(\Delta),
\gamma_1(\Delta),\dots,\gamma_{\lfloor n/2 \rfloor}
(\Delta))$ is known as the \emph{$\gamma$-vector} of 
$\Delta$ and Gal's conjecture states that $\gamma_i
(\Delta) \ge 0$ for all $i$. Just as in the case of 
Theorem~\ref{thm:KRS} (c), these inequalities impose 
a lower bound on each $h_i(\Delta)$ in terms of the 
numbers $h_j(\Delta)$ for $0 \le j < i$, and hence a
lower bound on each face number $f_{i-1}(\Delta)$ 
in terms of the face numbers $f_{j-1}(\Delta)$ for 
$1 \le j < i$. 

For instance, since $\gamma_1(\Delta) = h_1 
(\Delta) - n = f_0(\Delta) - 2n$, the inequality 
$\gamma_1(\Delta) \ge 0$ states that every flag 
triangulation of the $(n-1)$-dimensional sphere has 
at least $2n$ vertices, a fact which is easy to 
prove by induction on $n$. 
The inequality $\gamma_2(\Delta) \ge 0$ already 
provides a challenge. Since 
\[ \gamma_2 (\Delta) \ = \ h_2(\Delta) \, - \, 
   (n-2) \gamma_1 (\Delta) \, - \, {n \choose 2} 
   \ = \ f_1 (\Delta) \, - \, (2n-3) f_0(\Delta) 
   \, + \, 2n(n-2), \]
its validity turns out to be equivalent to the 
following conjectural flag analog of Barnette's 
Lower Bound Theorem~\cite{Ba73}. We recall that 
the suspension (simplicial join with a 
zero-dimensional sphere) of a flag triangulation
of a sphere is a flag triangulation of a 
sphere of one dimension higher.  
\begin{conjecture} \label{conj:gamma2}
Among all flag triangulations of the 
$(n-1)$-dimensional sphere with a given number 
$m$ of vertices, the $(n-2)$-fold 
suspension over the one-dimensional sphere 
with $m-2n+4$ vertices has the smallest 
possible number of edges. 
\end{conjecture}

For a generalization of this statement, see 
\cite[Conjecture~1.4]{Ne09}. A conjecture 
concerning the possible vectors which can 
arise as $\gamma$-vectors of flag 
triangulations of spheres is proposed 
in~\cite{NP11}. 

An edge subdivision of a simplicial complex 
$\Delta$ is a stellar subdivision of $\Delta$ 
on any of its edges; we refer to 
\cite[Section~6]{Ath12} \cite[Section~5.3]{CD95}
\cite[Section~2.4]{Ga05} for the precise 
definition and Figure~3 for
an example. Edge subdivisions preserve flagness
and homeomorphism type. We will denote by 
$\Sigma_n$ the simplicial join of $n$ copies of the 
zero-dimensional sphere (this simplicial complex 
is combinatorially isomorphic to the boundary 
complex of the $n$-dimensional cross-polytope; 
it satisfies $h(\Sigma_n, x) = (1+x)^n$ or, 
equivalently, $\gamma(\Sigma_n) = (1,0,\dots,0)$).
Apart from the results on barycentric subdivisions
and nested complexes, discussed separately in 
the sequel, Gal's conjecture is known to hold for 
flag triangulations of spheres:
\begin{itemize}
\itemsep=0pt
\item[$\bullet$] of dimension less than five
                 \cite[Theorem~11.2.1]{DO01}
                 \cite[Corollary~2.2.3]{Ga05},
\item[$\bullet$] with at most $2n+3$ vertices, i.e., 
                 with $\gamma_1(\Delta) \le 3$  
                 \cite[Theorem~1.3]{NP11}, 
\item[$\bullet$] which can be obtained from $\Sigma_n$ 
                 by successive edge subdivisions
                 \cite[Section~2.4]{Ga05},
\item[$\bullet$] for other special classes of flag 
                 triangulations of spheres, discussed 
                 in Section~\ref{subsec:exa}.
\end{itemize}

\begin{remark} \label{rem:edge-sub} \rm
The proof of the third statement above is only 
implicit in \cite[Section~2.4]{Ga05}. To make it 
more explicit, let us denote by $\Delta_e$ the edge 
subdivision of $\Delta$ on its edge $e \in \Delta$. 
Gal observes \cite[Proposition~2.4.3]{Ga05} that
\begin{equation} \label{eq:h-edge-sub} 
h(\Delta_e, x) \ = \ h(\Delta, x) \, + \, x \, 
h(\link_\Delta(e), x),
\end{equation}
where $\link_\Delta (F) = \{ G \sm F: G \in \Delta, 
\, F \subseteq G\}$ is the \emph{link} of $F \in 
\Delta$ in $\Delta$. In particular, if $\Delta$ and 
$\link_\Delta (e)$ satisfy Gal's conjecture, then so  
does $\Delta$ \cite[Corollary~2.4.7]{Ga05}. One can 
verify that every link of a nonempty face in 
$\Delta_e$ is combinatorially isomorphic to either 
the link of a nonempty face in $\Delta$, or to an 
edge subdivision or the suspension of the link of a 
nonempty face in $\Delta$. Thus, it follows by 
induction and Equation~(\ref{eq:h-edge-sub}) that if
$\Delta$ can be obtained from $\Sigma_n$ by successive 
edge subdivisions, then $\Delta$ and all links of its 
faces satisfy Gal's conjecture. 
\end{remark}

\begin{figure}
\begin{center}
\begin{tikzpicture}[scale=0.85]
\label{fg:edge-sub}

   \draw(-4,0) node(1){$\bullet$};
   \draw(-2,0) node(2){$\bullet$};
   \draw(-2,2) node(4){$\bullet$};
   \draw(-2,4) node(5){$\bullet$};
   \draw(-1,0) node(6){$\bullet$};
   \draw(-1,2) node(7){$\bullet$};
   \draw(0,0) node(8){$\bullet$};

   \draw(-4,0) -- (0,0) -- (-2,4) -- (-4,0);
   \draw(-4,0) -- (-2,2) -- (-1,2) -- (-2,0);
   \draw(-2,0) -- (-2,4);
   \draw(-1,0) -- (-1,2);

   \draw(4,0) node(9){$\bullet$};
   \draw(6,0) node(10){$\bullet$};
   \draw(5,1) node(11){$\bullet$};
   \draw(6,2) node(12){$\bullet$};
   \draw(6,4) node(13){$\bullet$};
   \draw(7,0) node(14){$\bullet$};
   \draw(7,2) node(15){$\bullet$};
   \draw(8,0) node(16){$\bullet$};

   \draw(4,0) -- (8,0) -- (6,4) -- (4,0);
   \draw(4,0) -- (6,2) -- (7,2) -- (6,0);
   \draw(6,0) -- (5,1) -- (6,4);
   \draw(6,0) -- (6,4);
   \draw(7,0) -- (7,2);
 
   \draw(2,1.4) node(17){$\longrightarrow$};

\end{tikzpicture}
\caption{An edge subdivision}
\end{center}
\end{figure}
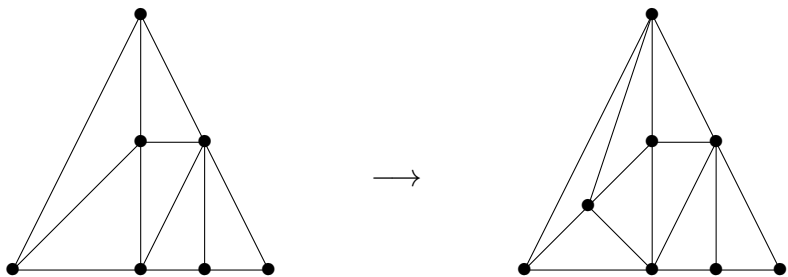

\subsubsection{Barycentric subdivisions}
\label{subsubsec:Karu}

An important class of flag simplicial complexes is 
that of order complexes~\cite[Section~III.4]{StaCCA}. 
Recall that the \emph{order complex} of a poset
$\pP$ is defined as the simplicial complex $\Delta
(\pP)$ of all chains (totally ordered subsets) of 
$\pP$. When $\pP$ is the poset of (nonempty) faces 
of a regular CW-complex $X$, the complex $\Delta
(\pP)$ is by definition the (first) barycentric 
subdivision of $X$. As explained, for instance, in
\cite[Section~2.3]{Ga05}, the following result is 
a direct consequence of important theorems of
K.~Karu and R.~Stanley on the nonnegativity of the 
cd-index of Gorenstein* posets~\cite{Ka06} (for 
introductions to the cd-index, see 
\cite[Section~III.4]{StaCCA} 
\cite[Section~3.17]{StaEC1}). 
\begin{theorem} \label{thm:Karu}
{\rm (Karu~\cite{Ka06}, Stanley~\cite{Sta94}; 
see~\cite[Corollary~2.3.5]{Ga05})} 
Conjecture~\ref{conj:Gal} holds for every order 
complex $\Delta$ which triangulates a sphere. In
particular, it holds for barycentric subdivisions
of regular CW-spheres.
\end{theorem}

As recorded in \cite[Section~7.3]{CD95} 
\cite[p.~103]{StaCCA}, the connection between the 
cd-index of an Eulerian poset $\pP$ and the last
entry of the $\gamma$-vector of $\Delta(\pP)$ was 
observed by E.~Babson. The nonnegativity of
the cd-index was proved for the augmented face 
posets of certain shellable CW-spheres by Stanley 
\cite[Theorem~2.2]{Sta94} and for arbitrary
Gorenstein* posets by Karu \cite[Theorem~1.3]{Ka06}
(Stanley's result suffices to conclude the validity
of Gal's conjecture for barycentric subdivisions of
boundary complexes of convex polytopes).
For barycentric subdivisions of Boolean complexes 
which triangulate a sphere, more elementary proofs 
of Theorem~\ref{thm:Karu} can be found 
in~\cite{BW08, NPT11}.

There is an interesting analog of 
Theorem~\ref{thm:Karu} for triangulations of balls, 
rather than spheres. As explained in
\cite[Section~4]{KMS17}, the following result is 
a direct consequence of~\cite[Theorem~2.5]{EK07}.
Here $\partial \Delta$ denotes the boundary of a 
triangulated ball $\Delta$. 
\begin{theorem} \label{thm:EKaru}
{\rm (Ehrenborg--Karu~\cite{EK07}; 
see~\cite[Theorem~4.6]{KMS17})} 
The polynomial $h(\Delta, x) - h(\partial \Delta, x)$
is $\gamma$-positive for every order complex $\Delta$ 
which triangulates a ball. In particular, this holds 
for barycentric subdivisions of regular CW-balls.
\end{theorem}

\subsubsection{Flag nested complexes}
\label{subsubsec:nesto}

We begin with a few definitions, following~\cite{Ze06}.
A \emph{building set} on the ground set $[n]$ is a
collection $\bB$ of nonempty subsets of $[n]$ such 
that:
\begin{itemize}
\itemsep=0pt
\item[$\bullet$] $\bB$ contains all singletons $\{i\}                    
                 \subseteq [n]$, and
\item[$\bullet$] if $I, J \in \bB$ and $I \cap J \ne 
                \varnothing$, then $I \cup J \in \bB$. 
\end{itemize}
We denote by $\bB_{\max}$ the set of maximal (with 
respect to inclusion) elements of $\bB$ and say that 
$\bB$ is \emph{connected} if $[n] \in \bB$. A set $N
\subseteq \bB \sm \bB_{\max}$ is called \emph{nested}
if for all $k \ge 2$ and all $I_1, I_2,\dots,I_k \in 
N$ such that none of the $I_i$ contains another, their
union $I_1 \cup I_2 \cup \cdots \cup I_k$ is not in 
$\bB$. The simplicial complex $\Delta_\bB$ on the 
vertex set $\bB \sm \bB_{\max}$, consisting of all 
nested sets, is the \emph{nested complex} associated 
to $\bB$. For example, if $\bB$ consists of all nonempty
subsets of $[n]$, then $\Delta_\bB$ is the barycentric 
subdivision of the boundary of the simplex $2^{[n]}$.

The nested complex $\Delta_\bB$ is isomorphic to the
boundary complex of a simplicial polytope of dimension
$n - |\bB_{\max}|$ \cite[Theorem~3.14]{FS05} 
\cite[Theorem~7.4]{Pos09} \cite[Theorem~6.1]{Ze06}; 
the corresponding polar-dual simple polytope is called 
a \emph{nestohedron}. Nestohedra form an important 
class of Postnikov's generalized 
permutohedra~\cite{Pos09} which includes 
permutohedra, graph-associahedra and other well 
studied simple polytopes in geometric combinatorics. 
As shown in \cite[Theorem~4]{FY04} and explained in 
\cite[Remark~6.6]{PRW08}, the complex $\Delta_\bB$ 
can be constructed as follows. Assume (without loss
of generality) that $\bB$ is connected and consider 
the boundary $\partial \sigma$ of the simplex 
$\sigma$ on the vertex set $[n]$. Choose any 
ordering of the nonsingleton elements of $\bB 
\sm \bB_{\max}$ which respects the reverse of the 
inclusion order. Starting with $\partial \sigma$, 
and following the chosen ordering, for each 
nonsingleton $I \in \bB \sm \bB_{\max}$ perform a 
stellar subdivision on the face $I$. The resulting 
simplicial complex is $\Delta_\bB$.
\begin{example} \label{ex:build} \rm
For $n=4$, consider the building set $\bB$ 
consisting of $\{1\}$, $\{2\}$, $\{3\}$, $\{4\}$,
$\{3,4\}$, $\{2,3,4\}$ and $\{1,2,3,4\}$. The
nested complex $\Delta_\bB$ triangulates the boundary 
of the three-dimensional simplex on the vertex set 
$\{\{1\},\{2\},\{3\},\{4\}\}$; it has eight facets,
six of which are shown in Figure~4 (the two unlabeled 
vertices being $\{3,4\}$ and $\{2,3,4\}$). The 
remaining facets are the nonvisible facets $\{\{1\},
\{2\},\{3\}\}$ and $\{\{1\},\{2\},\{4\}\}$ of the 
simplex. Note that $\Delta_\bB$ corresponds indeed
to the simplicial complex obtained by stellarly 
subdividing the boundary complex of the simplex on 
the vertex set $\{1,2,3,4\}$ first on the face 
$\{2,3,4\}$ and then on $\{3,4\}$.
\qed
\end{example}

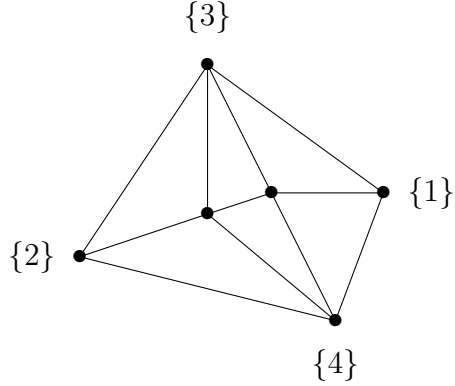
\begin{figure}
\begin{center}
\begin{tikzpicture}[scale=0.85]
\label{fg:nest}

   \draw(0,1) node(1){$\bullet$};
   \draw(2,4) node(2){$\bullet$};
   \draw(4,0) node(3){$\bullet$};
   \draw(3,2) node(4){$\bullet$};
   \draw(2,1.67) node(5){$\bullet$};
   \draw(4.75,2) node(6){$\bullet$};
   \draw(5.5,2) node(7){$\{1\}$};
   \draw(-0.75,1) node(8){$\{2\}$};
   \draw(2,4.75) node(9){$\{3\}$};
   \draw(4,-0.7) node(10){$\{4\}$};

   \draw(0,1) -- (4,0) -- (2,4) -- (0,1);
   \draw(0,1) -- (2,1.67) -- (3,2);
   \draw(2,4) -- (2,1.67) -- (4,0);
   \draw(4,0) -- (4.75,2) -- (2,4);
   \draw(3,2) -- (4.75,2);

\end{tikzpicture}
\caption{A nested complex}
\end{center}
\end{figure}

Gal's conjecture was verified for flag nested 
complexes by V.~Volodin~\cite{Vo10}, who showed 
that they can be obtained from the boundary 
complex of a cross-polytope by successive edge 
subdivisions. The main result of \cite{PRW08} 
provides an explicit combinatorial interpretation  
of the entries of the $\gamma$-vector of a large 
family of flag nested complexes. To state this 
result, we need to introduce a few definitions 
from \cite{PRW08}. Let $\bB$ be a building set 
on the ground set $[n]$. The \emph{restriction}
of $\bB$ to $I \subseteq [n]$ is defined as the
family of all subsets 
of $I$ which belong to $\bB$. The restrictions 
of $\bB$ to the elements of $\bB_{\max}$ are the 
\emph{connected components} of $\bB$. The building 
set $\bB$ is called \emph{chordal} if for every
$\{i_1, i_2,\dots,i_k\} \in \bB$ with $i_1 < i_2
< \cdots < i_k$ and all indices $1 \le r \le k$ 
we have $\{i_r, i_{r+1},\dots,i_k\} \in \bB$. 
Nested complexes of chordal building sets are 
always flag \cite[Proposition~9.7]{PRW08}. We 
denote by $\fS_\bB$ the set of all permutations
$w \in \fS_n$ such that $w(i)$ and $\max\{w(1),
w(2),\dots,w(i)\}$ lie in the same connected 
component of the restriction of $\bB$ to $\{w(1),
w(2),\dots,w(i)\}$, for every $i \in [n]$. The 
following result combines Corollary~9.6 with 
Theorem~11.6 in~\cite{PRW08} and, as shown in  
\cite[Section~10]{PRW08}, provides a 
common generalization for the $\gamma$-positivity
of the Eulerian and binomial Eulerian polynomials 
$A_n(x)$ and $\widetilde{A}_n(x)$ and the Narayana
polynomials $\cat(W, x)$ for the symmetric and
hyperoctahedral groups, discussed in 
Section~\ref{sec:comb}.
\begin{theorem} \label{thm:PRW}
{\rm 
(Postnikov--Reiner--Williams~\cite{PRW08})}
For every connected, chordal building set $\bB$
on the ground set $[n]$
\begin{equation} \label{eq:PRW} 
h(\Delta_\bB, x) \ = \ \sum_{i=0}^n
h_{\bB, i} \, x^i \ = \ \sum_{i=0}^{\lfloor n/2 
\rfloor} \gamma_{\bB, i} \, x^i (1+x)^{n-2i},
\end{equation}
where 

\begin{itemize}
\item[$\bullet$] $h_{\bB,i}$ is the number of 
permutations $w \in \fS_\bB$ with $i$ descents, 
                 and
\item[$\bullet$] $\gamma_{\bB,i}$ is the number 
of permutations $w \in \fS_\bB$ for which 
$\Des(w) \in \Stab([n-2])$ has $i$ elements.
  \end{itemize}
\end{theorem}

The nested complex $\Delta_\bB$ of 
Example~\ref{ex:build} has eight facets. The 
elements of $\fS_\bB$, written in one-line 
notation, are $1234$, $1243$, $1342$, $1432$, 
$2341$, $2431$, $3421$ and $4321$ and
$h(\Delta_\bB, x) = 1 + 3x + 3x^2 + x^3 = 
(1+x)^3$, in agreement with 
Theorem~\ref{thm:PRW}.

\subsection{Flag triangulations of simplices}
\label{subsec:local}

This section focuses on a variant of the 
$h$-polynomial of a triangulation of a sphere, 
the local $h$-polynomial, defined for 
triangulations of simplices (and for more general 
topological, rather than geometric, simplicial 
subdivisions of the simplex; see 
\cite[Part~I]{Sta92} \cite{Ath12, Ath16a}). 
The local $h$-polynomial was introduced by 
Stanley~\cite{Sta92} and plays a key role in his 
enumerative theory of triangulations of simplicial 
complexes, developed in order to study their effect 
on the $h$-polynomial of a simplicial complex.

Let $V$ be an $n$-element set and $\Gamma$ 
be a triangulation of the simplex $2^V$. The 
restriction of $\Gamma$ to the face $F \subseteq 
V$ of the simplex is a triangulation of $F$, 
denoted by $\Gamma_F$.
\begin{definition} \label{def:localh}
{\rm (Stanley~\cite[Definition~2.1]{Sta92})} 
The local $h$-polynomial of $\Gamma$ (with respect 
to $V$) is defined as
\begin{equation} \label{eq:deflocalh}
\ell_V (\Gamma, x) \ := \ \sum_{F \subseteq V} 
(-1)^{n - |F|} \, h (\Gamma_F, x) \ := \ \ell_0 + 
\ell_1 x + \cdots + \ell_n x^n.
\end{equation}
The sequence $\ell_V (\Gamma) = (\ell_0, 
\ell_1,\dots,\ell_n)$ is the local $h$-vector of 
$\Gamma$ (with respect to $V$).
\end{definition}

For example, the complex of 
Figure~2 is naturally a triangulation $\Gamma$ of 
a two-dimensional simplex $2^V$. Since 
$h(\Gamma, x) = 1 + 5x + 2x^2$ and the restrictions 
of $\Gamma$ on the three edges of the simplex have 
$h$-polynomials $1$, $1+x$ and $1+2x$, we have 
$\ell_V (\Gamma, x) = (1 + 5x + 2x^2) - 1 - (1+x) -
(1+2x) + 1 + 1 + 1 - 1 = 2x + 2x^2$.

The following theorem shows that the local 
$h$-polynomial has properties similar to those of
the $h$-polynomial of a triangulation of the sphere,
stated in Theorem~\ref{thm:KRS}. For the notion of a 
regular triangulation, we refer to 
\cite[Chapter~5]{DRS10} 
\cite[Definition~5.1]{Sta92}.
\begin{theorem} \label{thm:proplocal} 
{\rm (Stanley~\cite{Sta92})}
Let $\Gamma$ be a triangulation of an 
$(n-1)$-dimensional simplex $2^V$ and set $\ell_V 
(\Gamma, x) = \ell_0 + \ell_1 x + \cdots + \ell_n 
x^n$.

\begin{itemize}
\item[{\rm (a)}]
The polynomial $\ell_V (\Gamma, x)$ has 
nonnegative coefficients.
\item[{\rm (b)}]
The polynomial $\ell_V (\Gamma, x)$ is 
symmetric, i.e., we have $\ell_i = \ell_{n-i}$ for 
$0 \le i \le n$.
\item[{\rm (c)}]
The polynomial $\ell_V (\Gamma, x)$ is 
unimodal, i.e., we have $\ell_0 \le \ell_1 \le 
\cdots \le \ell_{\lfloor d/2 \rfloor}$, if $\Gamma$ 
is regular. 
\end{itemize}
\end{theorem}

The following analog of Conjecture~\ref{conj:Gal}
may come as no surprise to some readers; it is stated
more generally in \cite[Conjecture~5.4]{Ath12} for
a class of topological simplicial subdivisions which
models combinatorially geometric subdivisions.
\begin{conjecture} \label{conj:Ath}
{\rm (Athanasiadis~\cite{Ath12})} 
The polynomial $\ell_V(\Gamma, x)$ is 
$\gamma$-positive for every flag triangulation 
$\Gamma$ of the simplex $2^V$. 
\end{conjecture}

Since $\ell_V(\Gamma, x)$ is symmetric, with center
of symmetry $n/2$, we may write
\begin{equation} \label{eq:def-localgammavector} 
\ell_V(\Gamma, x) \ = \ \sum_{i=0}^{\lfloor n/2 \rfloor} 
\xi_i(\Gamma) x^i (1+x)^{n-2i}.
\end{equation}
The sequence $\xi(\Gamma) := (\xi_0(\Gamma),
\xi_1(\Gamma),\dots,\xi_{\lfloor n/2 \rfloor}
(\Gamma))$ is termed in~\cite{Ath12} as the 
\emph{local $\gamma$-vector} of $\Gamma$ (with 
respect to $V$) and the polynomial $\xi_V(\Gamma, 
x) := \sum_{i=0}^{\lfloor n/2 \rfloor} \xi_i
(\Gamma) x^i$ as the corresponding \emph{local 
$\gamma$-polynomial}. Thus, 
Conjecture~\ref{conj:Ath} claims that $\xi_i
(\Gamma) \ge 0$ for all $0 \le i \le 
\lfloor n/2 \rfloor$ and every flag triangulation 
$\Gamma$ of the simplex $2^V$. Apart from 
results on barycentric subdivisions, discussed in 
the sequel, Conjecture~\ref{conj:Ath} has been 
verified for flag triangulations of simplices:
\begin{itemize}
\itemsep=0pt
\item[$\bullet$] of dimension less than four
                 \cite[Proposition~5.7]{Ath12},
\item[$\bullet$] which can be obtained from the trivial 
                 triangulation by successive edge 
                 subdivisions \cite[Proposition~6.1]{Ath12},
\item[$\bullet$] for other special classes of flag 
                 triangulations of simplices, discussed 
                 in Section~\ref{subsec:exa}.
\end{itemize}

\medskip
The connection between $\gamma$-vectors and their
local counterparts is explained by the following 
result. We recall that $\Sigma_n$ stands for the 
boundary complex of the $n$-dimensional 
cross-polytope and refer to \cite{Ath12} 
\cite[Section~2]{Ath16a} for the precise 
definitions of the more general notions of 
simplicial subdivision which appear there. 
\begin{theorem} \label{thm:ath-gammaxi}
{\rm (Athanasiadis~\cite{Ath12} 
\cite[Theorem~3.7]{Ath16a})}
Every flag triangulation $\Delta$ of the 
$(n-1)$-dimensional sphere is a vertex-induced 
simplicial homology subdivision of $\Sigma_n$. 
Moreover,
\begin{equation} \label{eq:gamma-xi}
  \gamma (\Delta, x) \ = \ \sum_{F \in \Sigma_n} 
  \, \xi_F (\Delta_F, x).
  \end{equation}
\end{theorem}

The version of Conjecture~\ref{conj:Ath} stated 
as \cite[Conjecture~5.4]{Ath12} applies to the 
summands in the right-hand side of 
(\ref{eq:gamma-xi}). As a result, 
\cite[Conjecture~5.4]{Ath12} implies Gal's
conjecture, as well as its stronger 
version \cite[Conjecture~2.1.7]{Ga05} for flag 
homology spheres. Moreover, as explained at the
end of \cite[Section~1]{Ath16b}, the validity 
of Conjecture~\ref{conj:Ath} for regular flag 
triangulations of simplices implies Gal's 
conjecture for flag boundary complexes of convex 
polytopes.

The following statement is a direct consequence 
of Theorem~\ref{thm:ath-gammaxi} and Stanley's 
monotonicity theorem \cite[Theorem~4.10]{Sta92}
(in the form of \cite[Corollary~2.9]{Ath16a}) 
for $h$-vectors of subdivisions of simplicial 
complexes; it was proved more generally for flag 
doubly Cohen--Macaulay complexes, using different 
methods, in \cite[Theorem~1.3]{Ath11}.
\begin{corollary} \label{cor:h-bounds}
Every flag triangulation $\Delta$ of the 
$(n-1)$-dimensional sphere satisfies 
\begin{equation} \label{eq:h-bounds}
h_i(\Delta) \ \ge \ {n \choose i} 
\end{equation} 
for $0 \le i \le n$.
\end{corollary}

\subsubsection{Barycentric subdivisions}
\label{subsubsec:KMS}

Barycentric subdivisions of polyhedral 
subdivisions of the simplex form a natural class
of its flag triangulations. The question of the 
validity of Conjecture~\ref{conj:Ath} for 
them was raised in \cite[Question~6.2]{Ath12} 
\cite[Problem~3.8]{Ath16a} and answered in the 
affirmative by Juhnke-Kubitzke, Murai and 
Sieg \cite[Theorem~1.3]{KMS17}. The last 
statement in the following theorem is a 
consequence of Equation~(\ref{eq:KMS}), 
Theorem~\ref{thm:EKaru} and the 
$\gamma$-positivity of the derangement 
polynomials $d_n(x)$.
\begin{theorem} \label{thm:KMS}
{\rm (Kubitzke--Murai--Sieg~\cite{KMS17})} 
For every triangulation $\Gamma$ of an
$(n-1)$-dimensional simplex $2^V$,
\begin{equation} \label{eq:KMS}
\ell_V (\Gamma, x) \ = \ \sum_{F \subseteq V} \,
(h(\Gamma_F, x) - h(\partial (\Gamma_F), x)) \,
d_{n-|F|} (x).
\end{equation} 
In particular, Conjecture~\ref{conj:Ath} holds 
for barycentric subdivisions of regular 
CW-complexes which subdivide the simplex.
\end{theorem}

\subsubsection{Nested subdivisions}
\label{subsubsec:nest-sub}

Let $\bB$ be a connected building set on the 
ground set $[n]$. We denote by $\Gamma_\bB$
the cone of the nested complex $\Delta_\bB$. 
Alternatively, $\Gamma_\bB$ is the simplicial 
complex on the vertex set $\bB$, consisting 
of all nested subsets of $\bB$ (rather than
$\bB \sm \bB_{\max}$), defined by the condition 
described in Section~\ref{subsubsec:nesto}. 
Since $\Delta_\bB$ triangulates the boundary of 
an $(n-1)$-dimensional simplex, the complex 
$\Gamma_\bB$ is a triangulation of this simplex
(for example, it is the barycentric subdivision, 
if $\bB$ consists of all nonempty subsets of 
$[n]$). We call $\Gamma_\bB$ the \emph{nested 
subdivision} associated to $\bB$.

No doubt, just as is the case with $\Delta_\bB$, 
one can show that whenever $\Gamma_\bB$ is a flag 
complex, it can be obtained from the trivial 
subdivision of the $(n-1)$-dimensional simplex by 
successive edge subdivisions. Thus, the 
$\gamma$-positivity of the local $h$-polynomial 
of $\Gamma_\bB$ follows from 
\cite[Proposition~6.1]{Ath12}. In view of 
Theorem~\ref{thm:PRW}, it seems natural to pose 
the following problem.

\begin{problem} \rm
Find a combinatorial interpretation for the 
local $\gamma$-polynomial of $\Gamma_\bB$,
for any connected chordal building set $\bB$.
\end{problem}

\subsection{Examples}
\label{subsec:exa}

This section discusses further examples of 
flag triangulations of spheres or simplices,
which appear naturally in algebraic-geometric 
contexts, and the $\gamma$-positivity of their 
$h$-polynomials or local $h$-polynomials. 
These examples provide algebraic-geometric 
interpretations for several of the 
$\gamma$-positive polynomials discussed in
Section~\ref{sec:comb}.

\subsubsection{Barycentric and edgewise 
subdivisions}
\label{subsubsec:baryedge}

As noted in Section~\ref{subsubsec:Karu},
the barycentric subdivision of a regular 
CW-complex $X$, denoted here by $\sd(X)$, is 
defined as the simplicial complex of all 
chains in the poset of nonempty faces of $X$.

Let $V$ be an $n$-element set and consider  
the barycentric subdivision $\sd(2^V)$ of the
simplex $2^V$. The facets of $\sd(2^V)$ are 
in one-to-one correspondence with the 
permutations of $V$. Thus, the $h$-polynomial 
of $\sd(2^V)$ is an $x$-analog of the number 
$n!$; it is well known, in fact, that 
\begin{equation} \label{eq:hpolyAn}
h(\sd(2^V), x) \ = \ A_n(x). 
\end{equation}   
Since coning a simplicial complex does not 
affect its $h$-polynomial, $A_n(x)$ is also 
equal to the $h$-polynomial of the barycentric 
subdivision $\sd(\partial(2^V))$ of the boundary 
complex of $2^V$. Moreover, given that $\sd(2^V)$ 
restricts to the barycentric subdivision $\sd(2^F)$ 
for every $F \subseteq V$, Stanley 
\cite[Proposition~2.4]{Sta92} showed that  
\begin{equation} \label{eq:localdn}
\ell_V (\sd(2^V), x) \ = \ \sum_{k=0}^n (-1)^{n-k} 
{n \choose k} A_k(x) \ = \ \sum_{w \in \dD_n} 
x^{\exc(w)} \ = \ d_n(x).
\end{equation}
Consequently, the unimodality of $A_n(x)$ and 
$d_n(x)$ follows from parts (c) of 
Theorems~\ref{thm:KRS} and~\ref{thm:proplocal}
and their $\gamma$-positivity is an instance 
of Conjectures~\ref{conj:Gal} 
and~\ref{conj:Ath}, respectively. A part of the
barycentric subdivision of the boundary complex
of the three-dimensional simplex is shown in 
Figure~5. The restrictions on the facets of the 
simplex are barycentric subdivisions of 
two-dimensional simplices. 
 
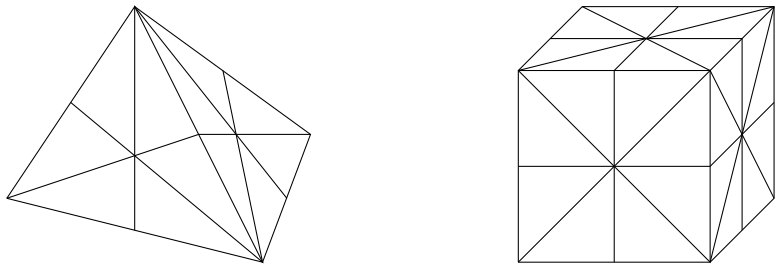
\begin{figure}
\begin{center}
\begin{tikzpicture}[scale=0.85]
\label{fg:bary}

   \draw(-4,1) -- (0,0) -- (-2,4) -- (-4,1);
   \draw(0,0) -- (0.75,2) -- (-2,4);
   \draw(-4,1) -- (-1,2) -- (0.75,2);
   \draw(-2,0.5) -- (-2,4) -- (0.38,1);
   \draw(-3,2.5) -- (0,0) -- (-0.62,3);

   \draw(4,0) -- (7,0) -- (7,3) -- (4,3) -- (4,0);
   \draw(4,3) -- (5,4) -- (8,4) -- (8,1) -- (7,0);
   \draw(7,3) -- (8,4);
   \draw(4,1.5) -- (7,1.5) -- (8,2.5);
   \draw(5.5,0) -- (5.5,3) -- (6.5,4);
   \draw(4.5,3.5) -- (7.5,3.5) -- (7.5,0.5);
   \draw(4,0) -- (7,3) -- (8,1);
   \draw(4,3) -- (7,0) -- (8,4);
   \draw(5,4) -- (7,3);
   \draw(4,3) -- (8,4);

\end{tikzpicture}
\caption{Two Coxeter complexes}
\end{center}
\end{figure}

The $r$-fold edgewise subdivision, denoted here
by $\esd_r(\Delta)$, is another elegant (but not 
as well known as barycentric subdivision) 
triangulation of a simplicial complex $\Delta$ 
with a long history in mathematics; see 
\cite[Section~1]{Ath16b} and references therein.
To describe it in a geometrically intuitive way, 
consider positive integers $n,r$ and the 
geometric simplex $\sigma_{n,r} = \{ (x_1, 
x_2,\dots,x_{n-1}) \in \RR^{n-1}: 0 \le x_{n-1} 
\le \cdots \le x_2 \le x_1 \le r \}$. The 
$r$-fold edgewise subdivision $\esd_r(2^V)$
of an $(n-1)$-dimensional simplex $2^V$ is 
realized as the triangulation of 
$\sigma_{n,r}$ whose facets are the 
$(n-1)$-dimensional simplices into which 
$\sigma_{n,r}$ is dissected by the affine 
hyperplanes in $\RR^{n-1}$ with equations  
\begin{itemize}
\itemsep=0pt
\item[$\bullet$] $x_i = k$, where $1 \le i \le n-1$,
\item[$\bullet$] $x_i - x_j = k$, where $1 \le i < j 
                 \le n-1$
\end{itemize}
for $k \in \{1, 2,\dots,r-1\}$. This triangulation 
has exactly $r^{n-1}$ facets; it is shown in Figure~6 
for $n=3$ and $r=4$. For an arbitrary simplicial 
complex $\Delta$, the $r$-fold edgewise subdivision
$\esd_r(\Delta)$ restricts to $\esd_r(2^F)$ for every
$F \in \Delta$; for a formal definition see, for
instance, \cite[Section~4]{Ath14} 
\cite[Section~4]{BW09}. 

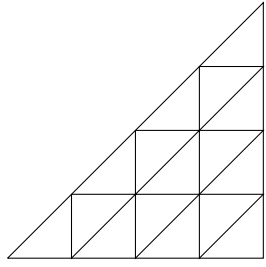
\begin{figure}
\begin{center}
\begin{tikzpicture}[scale=0.85]
\label{fg:edgewise}

  \draw(0,0) -- (4,0) -- (4,4) -- (0,0);
  \draw(1,1) -- (4,1) -- (3,0) -- (3,3) 
             -- (4,3) -- (1,0) -- (1,1); 
  \draw(2,2) -- (4,2) -- (2,0) -- (2,2); 

\end{tikzpicture}
\caption{An edgewise subdivision}
\end{center}
\end{figure}

The triangulation $\esd_r(\Delta)$ is flag for 
every flag simplicial complex $\Delta$; in 
particular, $\esd_r(2^V)$ is a flag triangulation 
of the simplex. Combinatorial interpretations 
for its local $h$-vector and local $\gamma$-vector 
can be given as follows. Let us denote by $\wW(n, 
r)$ the set of sequences $w = (w_0, w_1,\dots,w_n) 
\in \{0, 1,\dots,r-1\}^{n+1}$ satisfying $w_0 = w_n 
= 0$ and by $\sS(n,r)$ the set of those elements
of $\wW(n,r)$ (known as \emph{Smirnov words}) 
having no two consecutive entries equal. As usual,
we call $i \in \{0, 1,\dots,n-1\}$ an \emph{ascent} 
of $w \in \wW(n,r)$ if $w_i \le w_{i+1}$ and
$i \in [n-1]$ a \emph{double ascent} if 
$w_{i-1} \le w_i \le w_{i+1}$ (descents and 
double descents are defined similarly, using 
strict inequalities). We also denote 
by ${\rm E}_r$ the linear operator on the space 
$\RR[x]$ of polynomials in $x$ with real 
coefficients defined by setting ${\rm E}_r (x^m) 
= x^{m/r}$, if $m$ is divisible by $r$, and 
${\rm E}_r (x^m) = 0$ otherwise. For the second
interpretation of $\xi_{n, r, i}$ which appears 
in the following result (but not in \cite{Ath16a,
Ath16b}), see Remark~\ref{rem:GeEsdr}. 
\begin{theorem} \label{thm:AthEsdr}
{\rm (Athanasiadis~\cite[Theorem~4.6]{Ath16a} 
\cite{Ath16b})} 
The local $h$-polynomial of the $r$-fold edgewise
subdivision $\esd_r(2^V)$ of the $(n-1)$-dimensional
simplex on the vertex set $V$ can be expressed as
\begin{align*} 
\ell_V (\esd_r(2^V), x) & \ = \ {\rm E}_r \, 
               (x + x^2 + \cdots + x^{r-1})^n \ = \ 
                 \sum_{w \in \sS(n, r)} x^{\asc (w)}
								\\ & \ = \ 
\sum_{i=0}^{\lfloor n/2 \rfloor} \xi_{n, r, i} \,
               x^i (1+x)^{n-2i},
\end{align*}
where $\asc (w)$ is the number of ascents of $w \in 
\sS(n, r)$ and $\xi_{n, r, i}$ equals each of the 
following: 
\begin{itemize}
\itemsep=0pt
\item[$\bullet$] the number of sequences $w = (w_0, 
w_1,\dots,w_n) \in \sS(n, r)$ with exactly $i$ 
ascents which have the following property: for 
every double ascent $k$ of $w$ there exists a 
double descent $\ell > k$ such that $w_k = w_\ell$ 
and $w_k \le w_j$ for all $k < j < \ell$,

\item[$\bullet$] the number of $w \in \wW(n, r)$ 
for which $Asc(w) \in \Stab([n-2])$ has $i$ 
elements, where $\Asc (w)$ is the set of ascents 
of $w$.
\end{itemize}
\end{theorem}

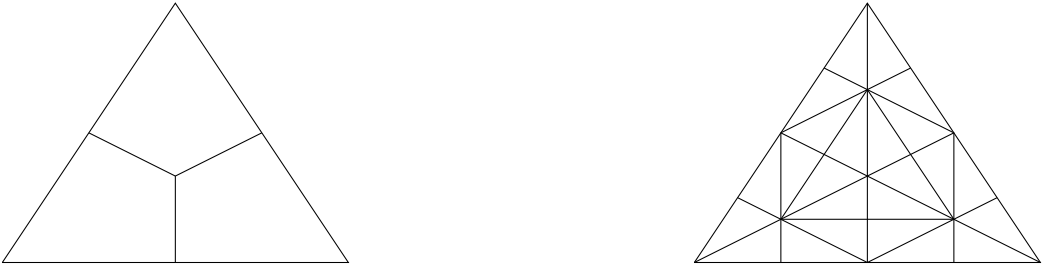
\begin{figure}
\begin{center}
\begin{tikzpicture}[scale=1.15]
\label{fg:Kn}

   \draw(-4,0) -- (0,0) -- (-2,3) -- (-4,0);
   \draw(-2,0) -- (-2,1) -- (-3,1.5);
   \draw(-2,1) -- (-1,1.5);

   \draw(4,0) -- (8,0) -- (6,3) -- (4,0);
   \draw(6,0) -- (6,3); 
   \draw(4,0) -- (7,1.5);
   \draw(5,1.5) -- (8,0);
   \draw(5,0.5) -- (7,0.5) -- (6,2) -- (5,0.5);
   \draw(5,0) -- (5,0.5) -- (4.5,0.75);
   \draw(7,0) -- (7,0.5) -- (7.5,0.75);
   \draw(5.5,2.25) -- (6,2) -- (6.5,2.25);
   \draw(6,0) -- (7,0.5) -- (7,1.5) -- (6,2) --
         (5,1.5) -- (5,0.5) -- (6,0);

\end{tikzpicture}
\caption{The triangulation $K_3$}
\end{center}
\end{figure}

The interpretations of the polynomials $A_n(x)$ 
and $d_n(x)$, described in this section, have 
interesting hyperoctahedral group analogs. 
Let $\cC_n$ denote the boundary complex of the 
standard $n$-dimensional unit cube. The facets 
of the barycentric subdivision $\sd(\cC_n)$ are 
in one-to-one correspondence with the signed 
permutations $w \in \bB_n$. As a result, the 
$h$-polynomial of $\sd(\cC_n)$ is an $x$-analog 
of the number $2^n n!$ and, in fact (see our 
discussion in Section~\ref{subsubsec:coxeter}), 
\begin{equation} \label{eq:hpolyBn}
h(\sd(\cC_n), x) \ = \ B_n(x),
\end{equation}   
where $B_n(x)$ is the $\bB_n$-Eulerian polynomial,
discussed in Section~\ref{subsubsec:CoxeterEuler}.
Consider now the cubical barycentric subdivision 
of an $(n-1)$-dimensional simplex $2^V$. This is 
a cubical complex whose face poset is isomorphic 
to the set of nonempty closed intervals in the 
poset of nonempty subsets of $V$, partially 
ordered by inclusion. Its (simplicial) barycentric
subdivision, denoted by $K_n$, is a triangulation
of $2^V$ with exactly $2^{n-1} n!$ facets. 
Figure~7 shows these subdivisions for $n=3$ (note
that $K_3$ is combinatorially isomorphic to the 
visible part of the complex $\sd(\cC_3)$ which is 
shown in Figure~5). We then have 
\cite[Remark~4.5]{Ath16a} 
\cite[Theorem~3.1.4]{Sav13}
\begin{equation} \label{eq:localdn2+}
\ell_V (K_n, x) \ = \ d^+_{n,2}(x),
\end{equation}   
where $d_{n,2}(x) = d^+_{n,2}(x) + d^{-}_{n,2}(x)$
is the decomposition of the derangement polynomial 
for $\bB_n$ discussed in 
Section~\ref{subsubsec:colored}. Thus, the 
$\gamma$-positivity of $B_n(x)$ and $d^+_{n,2}(x)$,
discussed in Section~\ref{sec:comb}, are again
instances of Conjectures~\ref{conj:Gal} 
and~\ref{conj:Ath}, respectively. 

The problem to interpret 
the polynomial $d^+_{n,r}(x)$, discussed in 
Section~\ref{subsubsec:colored}, as a local 
$h$-polynomial for all $r \ge 1$ was studied 
in~\cite{Ath14} and provided much of the 
motivation for that paper. It is not clear 
how to generalize the triangulation 
$K_n$ for this purpose. However, one can show 
directly (see \cite[Remark~4.5]{Ath16a}) that 
$\ell_V (K_n, x) =  \ell_V (\esd_2(\sd(2^V))$.
Thus, the $r$-fold edgewise subdivision 
$\esd_r(\sd(2^V))$ of the barycentric subdivision 
of the simplex $2^V$ is a reasonable candidate
to replace $K_n$ and, indeed, 
\cite[Theorem~1.2]{Ath14} shows that 
\begin{equation} \label{eq:localdnr+}
\ell_V (\esd_r(\sd(2^V)), x) \ = \ d^+_{n,r}(x).
\end{equation}   
The barycentric subdivision of the two-dimensional
simplex and its 2-fold edgewise subdivision are
shown in Figure~8. Further examples of 
barycentric and edgewise subdivisions appear 
in Section~\ref{subsec:gmethods}.

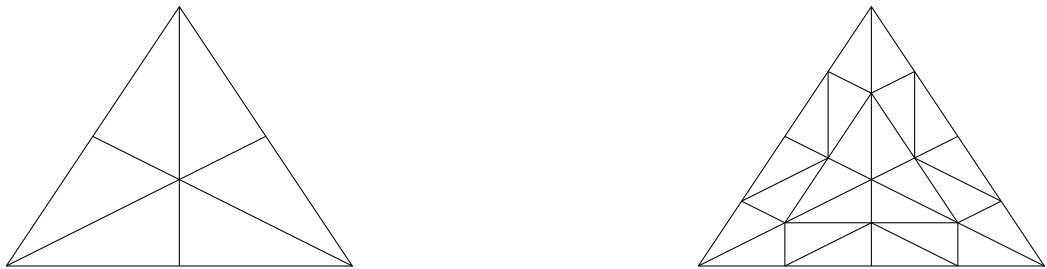
\begin{figure}
\begin{center}
\begin{tikzpicture}[scale=1.15]
\label{fg:baryedge}

   \draw(-4,0) -- (0,0) -- (-2,3) -- (-4,0);
   \draw(-4,0) -- (-1,1.5);
   \draw(-2,0) -- (-2,3);
   \draw(0,0) -- (-3,1.5);

   \draw(4,0) -- (8,0) -- (6,3) -- (4,0);
   \draw(6,0) -- (6,3); 
   \draw(4,0) -- (7,1.5);
   \draw(5,1.5) -- (8,0);
   \draw(5,0.5) -- (7,0.5) -- (6,2) -- (5,0.5);
   \draw(5,0) -- (5,0.5) -- (4.5,0.75);
   \draw(7,0) -- (7,0.5) -- (7.5,0.75);
   \draw(5.5,2.25) -- (6,2) -- (6.5,2.25);
   \draw(5,0) -- (6,0.5) -- (7,0);
   \draw(4.5,0.75) -- (5.5,1.25) -- (5.5,2.25);
   \draw(7.5,0.75) -- (6.5,1.25) -- (6.5,2.25);

\end{tikzpicture}
\caption{Barycentric and edgewise subdivisions}
\end{center}
\end{figure}

\subsubsection{Coxeter complexes}
\label{subsubsec:coxeter}

Let $W$ be a finite Coxeter group. Then $W$ can 
be realized as a reflection group in an 
$n$-dimensional Euclidean space, where $n$ is the
rank of $W$. The reflecting hyperplanes form a 
simplicial arrangement, known as the 
\emph{Coxeter arrangement}, whose face poset is 
isomorphic to that of a simplicial complex 
$\cox(W)$, called the \emph{Coxeter complex}. 
This complex is combinatorially isomorphic to
the barycentric subdivision of the boundary 
complex of the $n$-dimensional
simplex or cube (see Figure~5 for two 
three-dimensional pictures), when $W$ is the 
symmetric group $\fS_{n+1}$ or the hyperoctahedral
group $\bB_n$, respectively.

The facets of $\cox(W)$ are in a natural 
one-to-one correspondence with the elements 
of $W$. Thus, the $h$-polynomial of $\cox(W)$ 
is an $x$-analog of the order of the group
$W$ and, in fact, as was essentially shown 
by A.~Bj\"orner~\cite{Bj84} (see
\cite[Theorem~2.3]{Bre94a}), we have 
\begin{equation} \label{eq:hpolyW(x)}
h(\cox(W), x) \ = \ W(x),
\end{equation}   
where $W(x)$ is the $W$-Eulerian polynomial, 
discussed in 
Section~\ref{subsubsec:CoxeterEuler}.

Clearly, $\cox(W)$ is a triangulation of the 
$(n-1)$-dimensional sphere and, as explained 
in~\cite[Section~8.5]{Pet15}, this triangulation 
is flag. Thus, the $\gamma$-positivity of $W(x)$ 
(Theorem~\ref{thm:W(x)}) verifies Gal's 
conjecture in a special case. 

\subsubsection{Cluster complexes}
\label{subsubsec:cluster}

Let $W$ be a finite Coxeter group of rank $n$, 
viewed as a reflection group in an 
$n$-dimensional Euclidean space. Consider the 
root system $\Phi$ defined by the Coxeter 
arrangement associated to $W$ and a choice 
$\Phi^+$ of a positive system, along with 
corresponding simple system $\Pi$. The cluster
complex $\Delta_W$ is a flag simplicial complex
on the vertex set $\Phi^+ \cup (-\Pi)$, defined 
by S.~Fomin and A.~Zelevinsky \cite{FZ03}; we 
refer to~\cite{FR07} for an overview of the 
relevant theory and its connections to cluster 
algebras. This complex triangulates the 
$(n-1)$-dimensional sphere; its restriction on
the vertex set $\Phi^+$ is naturally a flag 
triangulation of the simplex $2^\Pi$, termed 
in \cite{AS12} as the \emph{cluster subdivision}. 
One then has \cite[Theorem~5.9]{FR07} 
\begin{equation} \label{eq:hpolyCat(x)}
h(\Delta_W, x) \ = \ \cat(W, x)
\end{equation}   
and \cite[Section~1.1]{AS12}
\begin{equation} \label{eq:localCat++(x)}
\ell_\Pi(\Gamma_W, x) \ = \ \cat^{++}(W, x),
\end{equation}   
where $\cat(W, x)$ and $\cat^{++}(W, x)$ were
defined in Section~\ref{subsec:nara}. Thus, 
the $\gamma$-positivity of these polynomials
(Theorems~\ref{thm:C(W,x)} 
and~\ref{thm:C++(W,x)}) verifies 
Conjectures~\ref{conj:Gal} and~\ref{conj:Ath}, 
respectively, in special cases.

\section{Methods}
\label{sec:methods}

A remarkable variety of methods has been employed to 
prove $\gamma$-positivity and often to describe the
$\gamma$-coefficients in some explicit way. This 
section reviews in some detail three such methods, 
namely valley hopping, methods of geometric 
combinatorics stemming from Stanley's seminal 
work~\cite{Sta92}, and (quasi)symmetric function 
methods. Although it seems difficult to provide an 
exhaustive list, or attempt some kind of 
classification, the author is aware of the following
methodological approaches which have been followed 
successfully to prove the $\gamma$-positivity of a 
polynomial $f(x)$:
\begin{itemize}
\itemsep=0pt
\item[$\bullet$] \emph{Combinatorial decompositions}:
     Assuming $f(x)$ enumerates a set of combinatorial  
     objects according to some statistic, one may try 
     to combinatorially decompose this set into parts,
     each contributing a binomial $x^i (1+x)^{n-2i}$ to
     this enumeration. One instance of this approach is
     valley hopping (see Section~\ref{subsec:hopping}) 
     and another is symmetric Boolean decompositions   
     of posets~\cite{Mu17, Pet13a, SU91}. Other 
     instances appear in \cite[Section~3]{AS12} 
     \cite{DPS09} \cite[Appendix~A]{Ste08} (see also 
     \cite[Section~13.2]{Pet15}).
\item[$\bullet$] \emph{Explicit combinatorial 
     formulas}: One may try to express $f(x)$ 
     explicitly as a sum of products of polynomials 
     known to be $\gamma$-positive, all products having 
     the same center of symmetry, using direct 
     combinatorial arguments, generating functions, 
     continued fraction expansions, and so on. The 
     prototypical example of this approach is 
     Br\"and\'en's~\cite{Bra04} beautiful proof of 
     Theorem~\ref{thm:BrAP(x)}; see also 
     \cite{Ch08, SZ12, SZ16, Ste08}. 
\item[$\bullet$] \emph{The theory of the cd-index of an
     Eulerian poset}: see Section~\ref{subsubsec:Karu}
     and \cite{DPS09} \cite[Appendix~A]{Ste08} for 
     related approaches.
\item[$\bullet$] \emph{Poset shellability and homology
     methods}: The prototypical example is the 
     use of shellability of Rees products of posets 
     in~\cite{LSW12}, leading to 
     Theorem~\ref{thm:LSWrees} and related results 
     discussed in Section~\ref{subsec:homo}; see also 
     Theorem~\ref{thm:Athrees} in the sequel.
\item[$\bullet$] \emph{Geometric methods}: Apart from 
     those discussed in Section~\ref{subsec:gmethods}, 
     geometric methods are employed in~\cite{Ai14, NP11},
     where the $\gamma$-positivity of $f(x)$ is proved by    
     constructing a simplicial complex whose $f$-vector 
     is shown to equal the $\gamma$-vector of $f(x)$. 
\item[$\bullet$] \emph{Recursive methods}: One can prove
     positivity of the $\gamma$-coefficients via 
     recursions, without providing any explicit
     interpretations or formulas. This is the case with 
     the proof of Theorem~\ref{thm:2sided} in~\cite{Li16} 
     and the $\gamma$-positivity of the restricted  
     Eulerian polynomials of~\cite[Section~4]{NPT11}.
\item[$\bullet$] \emph{Symmetric and quasisymmetric
     function methods}: see Section~\ref{subsec:sfmethods}.
\end{itemize}

\subsection{Valley hopping}
\label{subsec:hopping}

The idea of valley hopping is due to D.~Foata 
and V.~Strehl~\cite{FS74, FS76}, who used it 
to interpret combinatorially the 
$\gamma$-coefficients for the Eulerian 
polynomials (see Theorem~\ref{thm:FSSAn}).
It was rediscovered by L.~Shapiro, W.~Woan and 
S.~Getu~\cite{SWG83} and has found numerous 
applications to the enumeration of classes of 
permutations~\cite[Section~4]{AS12} 
\cite{Bra08, LZ15} \cite[Section~11]{PRW08} 
\cite{SuWa14} and related combinatorial 
objects~\cite[Section~3]{Ath16b}.

To explain this idea, let $w = (w_1, 
w_2,\dots,w_n) \in \fS_n$, where $w_i = w(i)$
for $1 \le i \le n$, be a permutation and set 
$w_0 = w_{n+1} = \infty$ and $\tilde{w} = (w_0, 
w_1,\dots,w_{n+1})$. Given a double ascent or 
double descent $i$ of $\tilde{w}$ (meaning, an 
index $i \in [n]$ such that $w_{i-1} < w_i < 
w_{i+1}$ or $w_{i-1} > w_i > w_{i+1}$, 
respectively), we define the permutation 
$\psi_i (w) \in \fS_n$ as follows: If $i$ is a 
double ascent of $\tilde{w}$, then $\psi_i (w)$ is 
the permutation obtained from $\tilde{w}$ by deleting 
$w_i$ and inserting it between $w_j$ and $w_{j+1}$, 
where $j$ is the largest index satisfying $0 \le j 
< i$ and $w_j > w_i$. Similarly, if $i$ is a double 
descent of $\tilde{w}$, then $\psi_i (w)$ is the 
permutation obtained from $\tilde{w}$ by deleting 
$w_i$ and inserting it between $w_j$ and $w_{j+1}$, 
where $j$ is the smallest index satisfying $i < j 
\le n$ and $w_i < w_{j+1}$. For the 
example of Figure~9 we have 
$\psi_1 (w) = (3, 2, 5, 4, 7, 8, 1, 6)$ and 
$\psi_8 (w) = (7, 3, 2, 5, 4, 8, 6, 1)$.

We call two permutations in $\fS_n$ equivalent 
if one can be obtained from the other by a 
sequence of moves of the form $w \mapsto \psi_i
(w)$. We leave to the reader to verify that 
this defines an equivalence relation on 
$\fS_n$, each equivalence class of which 
contains a unique permutation $u$ such that 
$\tilde{u}$ has no double ascent. Moreover, if 
$\tilde{u}$ has $k$ double descents for such $u 
\in \fS_n$, then the equivalence class $O(u)$ of 
$u$ has $2^k$ elements and exactly $\binom{k}{j}$ 
of them have $j$ ascents more than $u$. 
Therefore,
\begin{equation} \label{eq:orbits}
\sum_{w \in O(u)} \, x^{\asc(w)} \ = \ 
x^{\asc(u)} (1+x)^k \ = \ x^{\asc(u)}
      (1+x)^{n - 1 - 2\asc(u)}.
    \end{equation}
Summing over all equivalence classes yields 
the first combinatorial interpretation for 
$\gamma_{n,i}$, given in 
Theorem~\ref{thm:FSSAn}. 

\begin{figure}
\begin{center}
\begin{tikzpicture}[scale=0.8]
\label{fg:hop}

   \draw(1,3) node(1){$\bullet$};
   \draw(3,1) node(2){$\bullet$};
   \draw(3.5,0.5) node(3){$\bullet$};
   \draw(5,2) node(4){$\bullet$};
   \draw(5.5,1.5) node(8){$\bullet$};
   \draw(7.5,3.5) node(6){$\bullet$};
   \draw(13.5,2.5) node(5){$\bullet$};
   \draw(11,0) node(7){$\bullet$};

   \draw(0.4,3) node(9){7};
   \draw(2.4,1) node(10){3};
   \draw(3.5,-0.1) node(11){2};
   \draw(5,1.4) node(12){5};
   \draw(5.5,0.9) node(16){4};
   \draw(7.5,4.2) node(13){8};
   \draw(14.1,2.5) node(14){6};
   \draw(11,-0.6) node(15){1};

   \draw(0.5,3.5) -- (3.5,0.5) -- (5,2) -- (5.5,1.5) 
       -- (7.5,3.5) -- (11,0) -- (14.5,3.5);
   
   \draw [->] [dashed] (1,3) -- (6.9,3);
   \draw [->] [dashed] (3,1) -- (3.9,1);
   \draw [->] [dashed] (13.5,2.5) -- (8.6,2.5);

\end{tikzpicture}
\caption{Valley hopping for $w = (7,3,2,5,4,8,1,6)$}
\end{center}
\end{figure}
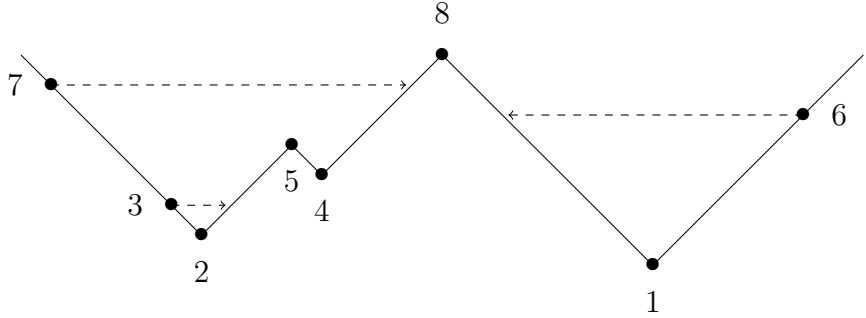

To illustrate the power of this method, let us 
modify it to prove Gessel's identity~(\ref{eq:Ge3}). 
We will use Stanley's interpretation (see 
\cite[Theorem~7.2]{SW10})
\begin{equation} \label{eq:path-chromatic}
\frac{(1-t) E(\bx; z)} {E(\bx; tz) - tE(\bx; z)} 
\ = \ 1 \, + \, \sum_{n \ge 1} z^n 
\sum_{w \in \sS(n)} t^{\asc(w)} \, \bx_w
\end{equation}
of the left-hand side of~(\ref{eq:Ge3}), where 
we have used notation of Section~\ref{subsec:sym} 
and $\sS(n)$ stands for the set of maps $w: [n] \to 
\ZZ_{>0}$ satisfying $w(i) \ne w(i+1)$ for every 
$i \in [n-1]$. 

\bigskip
\noindent
\emph{Proof of Equation~(\ref{eq:Ge3})}.
Given Equation~(\ref{eq:path-chromatic}), it 
suffices to show that 
\begin{equation} \label{eq:Ge-hop}
\sum_{w \in \sS(n)} t^{\asc(w)} \, \bx_w \ =
\sum_{\scriptsize \begin{array}{c} u: [n] \to 
     \ZZ_{>0} \\ \Asc(u) \in \Stab ([n-2]) 
     \end{array}}
     t^{\asc(u)} (1+t)^{n-1 - 2\asc(u)} \, \bx_u
\end{equation}
for all $n \ge 1$. For $w \in \sS(n)$ we write 
$w = (w_1, w_2,\dots,w_n)$ and define $\tilde{w}$,
as we did for permutations. For a double ascent $i$
of $\tilde{w}$, we denote by $\psi_i (w)$ the  
sequence obtained from $\tilde{w}$ by deleting 
$w_i$ and inserting it between positions $j$ and 
$j+1$, where $j$ is the largest index satisfying 
$0 \le j < i$ and $w_j > w_i$. We define $\psi_i 
(w)$ in a similar way for 
double descents $i$ of $\tilde{w}$ and 
say that a move $w \mapsto \psi_i(w)$ is 
\emph{valid} if $i$ is a double ascent or descent
of $\tilde{w}$ and $\psi_i(w) \in \sS(n)$. For the 
example of Figure~10 we have $\psi_4 (w) = 
(3, 2, 3, 2, 4, 3, 2, 1, 2)$, $\psi_6 (w) = 
(2, 3, 2, 3, 4, 2, 1, 2, 3)$, $\psi_7 (w) = 
(2, 3, 2, 3, 4, 3, 1, 2, 2)$ and $\psi_9 (w) 
= (2, 3, 2, 3, 4, 3, 2, 2, 1)$. The move $w \mapsto 
\psi_i(w)$ is valid for $i \in \{4, 6\}$ and 
invalid for $i \in \{7, 9\}$.

We call two elements of $\sS(n)$ equivalent 
if one can be obtained from the other by a 
sequence of valid moves of the form $w \mapsto 
\psi_i(w)$. We leave again to the reader to 
verify that this defines an equivalence 
relation on $\sS(n)$, each 
equivalence class of which contains a unique 
element $v$, call it the minimal representative, 
such that $\tilde{v}$ has a minimum number of 
double ascents. For example, the minimal 
representative of the equivalence class of 
the sequence shown in Figure~10 is 
$\psi_4(w)$. The reasoning which led to 
Equation~(\ref{eq:orbits}) also shows that the 
equivalence class $O(v)$ of a minimal 
representative $v$ satisfies
\begin{equation} \label{eq:orbits2}
\sum_{w \in O(v)} \, t^{\asc(w)} \, \bx_w \ = 
\ x^{\asc(v)} (1+x)^{n - 1 - 2\asc(v)} \, \bx_v.
\end{equation}
Summing over all equivalence classes gives 
\begin{equation} \label{eq:Ge-hop-v}
\sum_{w \in \sS(n)} t^{\asc(w)} \, \bx_w \ =
\sum \, t^{\asc(v)} (1+t)^{n-1 - 2\asc(v)} \, 
\bx_v,
\end{equation}
where the sum on the right-hand side runs over 
all minimal representatives $v$ of equivalence
classes. Applying all possible invalid moves 
to such an element $v$ which shift entries to 
the left (the order not making a difference) 
results in a map $u = f(v): [n] \to \ZZ_{>0}$ 
such that $\Asc(u) \in \Stab ([n-2])$ and 
$\bx_u = \bx_v$. Moreover, the map $f$ is a 
bijection from the set of minimal 
representatives $v$ to the set of such maps 
$u$ which preserves the weight $\bx_v$ and the 
proof follows. \qed 

\begin{remark} \label{rem:GeEsdr} \rm
The variant of valley hopping which was 
employed in the previous proof appeared in 
\cite[Section~3]{Ath16b} to derive the first
interpretation for $\xi_{n,r,i}$, given in
Theorem~\ref{thm:AthEsdr}. The last part of 
the argument in this proof yields the second 
interpretation. \qed
\end{remark}

\begin{figure}
\begin{center}
\begin{tikzpicture}[scale=0.8]
\label{fg:Gehop}

   \draw(7,1) node(4){$\bullet$};
   \draw(8,2) node(5){$\bullet$};
   \draw(9,1) node(7){$\bullet$};
   \draw(10,2) node(8){$\bullet$};
   \draw(11,3) node(9){$\bullet$};
   \draw(12,2) node(10){$\bullet$};
   \draw(13,1) node(11){$\bullet$};
   \draw(14,0) node(12){$\bullet$};
   \draw(15,1) node(13){$\bullet$};

   \draw(7,0.4) node(1){2};
   \draw(8,2.6) node(2){3};
   \draw(9,0.4) node(3){2};
   \draw(10,1.4) node(6){3};
   \draw(11,3.6) node(14){4};
   \draw(12,1.4) node(15){3};
   \draw(13,0.4) node(16){2};
   \draw(14,-0.6) node(17){1};
   \draw(15.6,1) node(17){2};

   \draw(5,3) -- (7,1) -- (8,2) -- (9,1) -- 
              (11,3) -- (14,0) -- (17,3);
  
   \draw [->] [dashed] (10,2) -- (6.1,2);
   \draw [->] [dashed] (12,2) -- (15.9,2);

\end{tikzpicture}
\caption{Valley hopping for 
              $w = (2,3,2,3,4,3,2,1,2)$}
\end{center}
\end{figure}
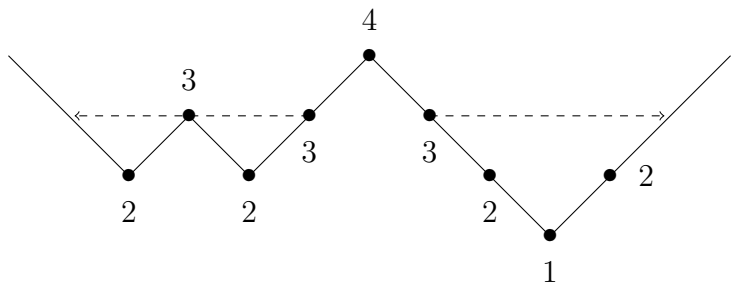

\subsection{Methods of geometric combinatorics}
\label{subsec:gmethods}

The enumerative theory of~\cite[Part~I]{Sta92},
developed by Stanley in order to study $h$-vectors 
of triangulations of simplicial complexes, together
with recent developments~\cite{Ath12}, provides 
a powerful method to prove $\gamma$-positivity of
polynomials which can often be defined purely in
combinatorial terms. This point of view, which is 
implicit in \cite[Section~4]{Ath16a} 
\cite[Chapter~3]{Sav13}, is further explained 
in this section. Familiarity with basic 
definitions from Section~\ref{sec:geom} is 
assumed.

The following statement explains the 
significance of the concept of local $h$-vector
in the theory of~\cite{Sta92}.
\begin{theorem} \label{thm:h-subdivide} 
{\rm (Stanley~\cite[Theorem~3.2]{Sta92})}
For every triangulation $\Delta'$ of a pure 
simplicial complex $\Delta$,
\begin{equation} \label{eq:h-subdivide}
h (\Delta', x) \ = \ \sum_{F \in \Delta} 
\ell_F (\Delta'_F, x) \, h (\link_\Delta (F), x).
\end{equation}
\end{theorem}

This result implies that $h (\Delta', x)$ is 
$\gamma$-positive (as a sum of $\gamma$-positive 
polynomials with the same center of symmetry), 
if so are $\ell_F (\Delta'_F, x)$ and 
$h (\link_\Delta (F), x)$ for every $F \in 
\Delta$. For instance, we have the following 
statement.
\begin{corollary} \label{cor:sd-Sigma} 
Suppose $\Delta$ is either:
\begin{itemize}
\itemsep=0pt
\item[$\bullet$] the simplicial join $\Sigma_n$ of
                 $n$ copies of the zero-dimensional 
                 sphere, or
\item[$\bullet$] the barycentric subdivision of the 
                 boundary of a simplex.
\end{itemize}
Then, $h (\Delta', x)$ is $\gamma$-positive for 
every triangulation $\Delta'$ of $\Delta$ for
which $\ell_F (\Delta'_F, x)$ is $\gamma$-positive 
for every $F \in \Delta$.
\end{corollary}
\begin{proof}
As already discussed, it suffices to confirm that 
$h (\link_\Delta (F), x)$ is $\gamma$-positive for 
all $F \in \Delta$. This is obvious if $\Delta 
= \Sigma_n$, since then $\link_\Delta (F)$ is 
isomorphic to $\Sigma_{n-|F|}$ and hence $h 
(\link_\Delta (F), x) = (1+x)^{n-|F|}$ for 
every $F \in \Delta$.

Suppose now that $\Delta = \sd(\partial(2^V))$
is the barycentric subdivision of the boundary 
complex of an $n$-dimensional simplex $2^V$.
Then, faces of $\Delta$ have the form $F = \{S_1,
S_2,\dots,S_k\}$, where $S_1 \subset S_2 \subset 
\cdots \subset S_k$ are nonempty proper subsets
of $V$. Setting $S_0 := \varnothing$ and $S_{k+1} 
:= V$, we have that $\link_\Delta(F)$ is the 
simplicial join of the complexes 
$\sd(\partial(2^{S_i \sm S_{i-1}}))$ for $i \in 
[k+1]$ and hence 
\begin{equation} \label{eq:sd-link}
h (\link_\Delta(F), x) \ = \ \prod_{i=1}^{k+1} 
A_{n_i}(x),
\end{equation}
where $n_i = |S_i \sm S_{i-1}|$ for $i \in [k+1]$.
This polynomial is indeed $\gamma$-positive, as a  
product of $\gamma$-positive polynomials.
Alternatively, $h (\link_\Delta(F), x)$ can be 
shown to be $\gamma$-positive for the barycentric
subdivision $\Delta$ of any regular CW-sphere by 
a stronger version of Theorem~\ref{thm:Karu} 
(applying to Gorenstein* order complexes).  
\end{proof}

\begin{example} \label{ex:h-esrsdsd} \rm
The following explicit formula for the 
$h$-polynomial of the $r$-fold edgewise subdivision 
of an $(n-1)$-dimensional simplicial complex 
$\Delta$ is a consequence of 
\cite[Corollary~1.2]{BW09} and 
\cite[Corollary~6.8]{BR05}:
  \begin{equation} \label{eq:hedge}
    h(\esd_r(\Delta), x) \ = \ 
    {\rm E}_r \left( (1 + x + x^2 + \cdots + 
    x^{r-1})^n \, h (\Delta, x) \right),
  \end{equation}

\noindent
where the linear operator ${\rm E}_r: \RR[x] \to 
\RR[x]$ was defined before Theorem~\ref{thm:AthEsdr}.
Thus, Corollary~\ref{cor:sd-Sigma}, combined with 
Theorem~\ref{thm:AthEsdr}, implies the 
$\gamma$-positivity of 
\begin{equation} \label{eq:h-esdrSigma}
  h(\esd_r(\Sigma_n), x) \ = \ 
  {\rm E}_r \left( (1 + x + x^2 + \cdots + x^{r-1})^n 
  \, (1+x)^n \right)
\end{equation}
and 
\begin{equation} \label{eq:h-sdsimplex}
  h(\esd_r(\sd(\partial (2^V))), x) \ = \ 
  {\rm E}_r \left( (1 + x + x^2 + \cdots + x^{r-1})^n 
  \, A_{n+1} (x) \right),
\end{equation}
where $2^V$ is an $n$-dimensional simplex. It would 
be interesting to find explicit combinatorial 
interpretations for the coefficients of the right-hand
sides of Equations~(\ref{eq:h-esdrSigma}) 
and~(\ref{eq:h-sdsimplex}) and their corresponding
$\gamma$-polynomials. 
Similar remarks apply to the second barycentric 
subdivision of the boundary complex of a simplex.
\qed
\end{example}

We now consider the following generalization of 
the concept of local $h$-polynomial. Let $V$ be 
an $n$-element set, $\Gamma$ be a triangulation 
of the simplex $2^V$ and $E \in \Gamma$.
\begin{definition} \label{def:rel-local}
{\rm (Athanasiadis~\cite[Remark~3.7]{Ath12} 
\cite[Definition~2.14]{Ath16a}, 
Nill--Schepers~\cite[Section~3]{Ni12})} 
The relative local $h$-polynomial of $\Gamma$ 
(with respect to $V$) at $E \in \Gamma$ is 
defined as
\[ \ell_V (\Gamma, E, x) \ = \sum_{\sigma(E) 
   \subseteq F \subseteq V} (-1)^{n - |F|} \, h  
   (\link_{\Gamma_F} (E), x), \]
where $\sigma(E)$ is the smallest face of $2^V$ 
containing $E$.
\end{definition}

This polynomial, which reduces to $\ell_V
(\Gamma, x)$ for $E = \varnothing$, has 
properties similar to those of $\ell_V
(\Gamma, x)$. Specifically, it has 
nonnegative and symmetric coefficients,
with center of symmetry $(n-|E|)/2$ 
(see \cite[Theorem~2.15]{Ath16a}) and, 
moreover, it is unimodal if $\Gamma$ is a
regular triangulation 
\cite[Remark~6.5]{KSt16} (see 
\cite[Section~2.3]{Ath16a} for more 
information). The significance of the 
relative local $h$-polynomial for us stems 
from the following statement. Note that if 
$\Gamma$ is a triangulation of the simplex 
$2^V$, then every triangulation $\Gamma'$ 
of $\Gamma$ induces a triangulation of $2^V$
whose local $h$-polynomial is denoted by 
$\ell_V(\Gamma', x)$.
\begin{proposition} \label{prop:hrel-local} 
{\rm (Athanasiadis~\cite[Section~3]{Ath12}
\cite[Proposition~2.15~(a)]{Ath16a})}
For every triangulation $\Gamma$ of the 
simplex $2^V$ and every triangulation 
$\Gamma'$ of $\Gamma$,
\begin{equation} \label{eq:h-local-subdivide}
\ell_V (\Gamma', x) \ = \ \sum_{E \in \Gamma} 
\ell_E (\Gamma'_E, x) \, \ell_V (\Gamma, E, x).
\end{equation}
\end{proposition}

As a result, $\ell_V (\Gamma', x)$ is 
$\gamma$-positive if so are $\ell_E (\Gamma'_E, 
x)$ and $\ell_V (\Gamma, E, x)$ for every $E \in 
\Gamma$. Using an argument similar to the one 
employed in the proof of 
Equation~(\ref{eq:sd-link}), one can show (see
\cite[Example~3.5.2]{Sav13}) that
\begin{equation} \label{eq:sdrel}
\ell_V (\sd(2^V), E, x) \ = \ d_{n_0}(x) \cdot 
\prod_{i=1}^k A_{n_i}(x),
\end{equation}
where $E = \{S_1, S_2,\dots,S_k\}$ with $\varnothing 
\subset S_1 \subset S_2 \subset \cdots \subset S_k 
\subseteq V$ is a face of $\sd(2^V)$, $d_{n_0}(x)$ 
is a derangement polynomial, $n_0 = |V \sm S_k|$ and 
$n_i = |S_i \sm S_{i-1}|$ for $1 \le i \le k$ (with 
the convention that $S_0 = \varnothing$). 
\begin{corollary} \label{cor:rel-triang-sd} 
The polynomial $\ell_V(\Gamma', x)$ is 
$\gamma$-positive for every triangulation $\Gamma'$ 
of $\sd(2^V)$ for which $\ell_E (\Gamma'_E, x)$ is 
$\gamma$-positive for every $E \in \sd(2^V)$.
\end{corollary}

The special case for which $\Gamma'$ is the 
$r$-fold edgewise subdivision proves the 
$\gamma$-positivity of the polynomial $d^+_{n,r}(x)$ 
discussed in Section~\ref{subsubsec:colored}; see
\cite[Example~4.8]{Ath16a}. We give another 
interesting application.
\begin{example} \label{ex:local-sdsd} \rm
Suppose that $\Gamma'$ is the barycentric 
subdivision of $\sd(2^V)$ and that $V$ is an  
$n$-element set. Then, the induced 
triangulation of $2^V$ is the second 
barycentric subdivision $\sd^{(2)} (2^V)$.
An explicit, but rather complicated, formula
for $\ell_V(\sd^{(2)} (2^V), x)$ can be 
derived from the definition of local 
$h$-polynomial and \cite[Theorem~1]{BW08}. 
Proposition~\ref{prop:hrel-local}, combined 
with Equations~(\ref{eq:localdn}) 
and~(\ref{eq:sdrel}), implies that
\begin{equation} \label{eq:local-sdsd}
\ell_V(\sd^{(2)} (2^V), x) \ = \ \sum 
{n \choose n_0, n_1,\dots,n_k} \, d_k(x) \, 
d_{n_0}(x) \, A_{n_1} (x) \cdots A_{n_k}(x),
\end{equation}
where the sum ranges over all $k \ge 0$ and 
over all sequences $(n_0, n_1,\dots,n_k)$ of 
integers which satisfy $n_0 \ge 0$, 
$n_1,\dots,n_k \ge 1$ and sum to $n$. The 
$\gamma$-positivity of the polynomial $\ell_V
(\sd^{(2)} (2^V), x)$ follows from this 
formula (equivalently, from 
Corollary~\ref{cor:rel-triang-sd}).

Using the principle of inclusion-exclusion, 
one can show that the sum of the coefficients 
of $\ell_V(\sd^{(2)} (2^V), x)$ equals 
the number of pairs of permutations
$u, v \in \fS_n$ which have no common fixed 
point. The problem to interpret combinatorially
the coefficients of this polynomial and its 
corresponding local $\gamma$-polynomial was 
posed in \cite[Problem~4.12]{Ath16a}.
\qed
\end{example}

\subsection{Symmetric/quasisymmetric function 
methods} \label{subsec:sfmethods}

Symmetric and quasisymmetric functions provide 
standard tools for enumeration problems; see, 
for instance, \cite{BM16, MR15} \cite[Chapter~7]
{StaEC2}. Given a polynomial
$f(t)$, to be shown to be $\gamma$-positive, one 
may attempt to find a (quasi)symmetric function 
which gives $f(t)$ via appropriate specialization.
Expanding this function in some basis of 
(quasi)symmetric functions may lead to a different
formula for $f(t)$ for which $\gamma$-positivity 
is easier to prove. The first instance of this 
approach seems to be Stembridge's proof of 
Theorem~\ref{thm:SteEnriched} (although a more 
direct proof, based on Equation~(\ref{eq:SteA'}), 
is also possible), where the role of the basis of
quasisymmetric functions is played by Gessel's 
basis of fundamental quasisymmetric 
functions~(\ref{eq:defFS(x)}).

To illustrate this approach, we sketch proofs 
of Theorems~\ref{thm:SWAn},~\ref{thm:SWdn} 
and~\ref{thm:ASWAn} which derive these statements
from the main result on Eulerian quasisymmetric 
functions~\cite{SW10} of Shareshian and Wachs, 
following the reasoning of~\cite[Section~4]{SW17}. 
We use the notation introduced in the beginning 
of the proof of Proposition~\ref{prop:athIn} and 
set 
\begin{equation} \label{eq:defF*S(x)}
  F^*_{n, S} (\bx) \ := F_{n, n-S} (\bx) \ =
  \sum_{\substack{i_1 \ge i_2 \ge \cdots 
  \ge i_n \ge 1 \\ j \in S \, \Rightarrow \, 
  i_j > i_{j+1}}}
  x_{i_1} x_{i_2} \cdots x_{i_n}
\end{equation}
for $S \subseteq [n-1]$. The principal 
specializations of this function are given 
by the formula (see \cite[Lemma~5.2]{GR93})
\begin{equation} \label{eq:speF*S(x)}
  \sum_{m \ge 1} F^*_{n, S} (1, q,\dots,q^{m-1}) 
  p^{m-1} \ = \ \frac{p^{|S|} q^{\Sum(S)}}
  {(1-p)(1-pq) \cdots (1-pq^n)}
\end{equation}
which refines 
Equation~(\ref{eq:speFS(x)}), where $\Sum(S)$ 
stands for the sum of the elements of $S$. We 
also recall the formula (see 
\cite[Proposition~7.19.12]{StaEC2}) 
\begin{equation} \label{eq:speSchur}
  \sum_{m \ge 1} s_\lambda (1, q,\dots,q^{m-1}) 
  p^{m-1} \ = \ \frac{f^\lambda(p,q)}
  {(1-p)(1-pq) \cdots (1-pq^n)}
\end{equation}
for the principal specializations of $s_\lambda
(\bx)$, where 
\begin{equation} \label{eq:def-f(p,q)}
  f^\lambda(p,q) \ := \ \sum_{P \in \SYT(\lambda)}
  p^{\des(P)} q^{\maj(P)} 
\end{equation}
and $\maj(P) := \Sum(\Des(P))$ is the major index 
of $P$.

\bigskip
\noindent
\emph{Proof of Theorems~\ref{thm:SWAn} 
and~\ref{thm:SWdn}}. Let us denote by $A_n(p,q,t)$ 
the left-hand side of Equation~(\ref{eq:SWAn}). 
Combining the special case $r=1$ of 
\cite[Theorem~1.2]{SW10} with Equation~(\ref{eq:Bre1})
and extracting the coefficient of $z^n$ shows that
\begin{equation} \label{eq:DEX1}
\sum_{w \in \fS_n} F_{n, \DEX(w)} (\bx) \, t^{\exc(w)} 
  \ = \ \sum_{\lambda \vdash n} P_\lambda(t) s_\lambda
  (\bx)
\end{equation}
for $n \ge 1$. We refer to \cite[Section~2]{SW10} for 
the definition of $\DEX(w)$; the only properties of 
this set which are essential here are the facts that 
$|\DEX(w)| = \des^*(w)$ and that $\Sum(\DEX(w)) = 
\maj(w) - \exc(w)$ for $w \in \fS_n$; see 
\cite[Lemma~2.2]{SW10}. Thus, taking the generating 
function of the principal specialization of order 
$m$ on both sides of~(\ref{eq:DEX1}) and 
using Equations~(\ref{eq:speF*S(x)}) 
and~(\ref{eq:speSchur}) results in the identity
\begin{equation} \label{eq:SW1A(p,q,t)}
  A_n(p,q,t) \ = \ \sum_{\lambda \vdash n} 
  P_\lambda(t) f^\lambda(p,q).
\end{equation}
By the $\gamma$-expansion of $P_\lambda(t)$, 
given in Corollary~\ref{cor:PRgamma}, and the 
definition of $f^\lambda(p,q)$, this formula 
implies that
\begin{equation} \label{eq:SW2A(p,q,t)}
  A_n(p,q,t) \ = \ \sum_{\lambda \vdash n} 
  \sum_{\scriptsize \begin{array}{c} P, Q \in 
  \SYT(\lambda) \\ \Des(Q) \in \Stab ([n-2]) 
  \end{array}} p^{\des(P)} q^{\maj(P)} \,
  t^{\des(Q)} (1+t)^{n-1-\des(Q)}.
\end{equation}
Using the Robinson--Schensted correspondence 
and its properties \cite[Lemma~7.23.1]{StaEC2} 
to replace the pair $(P,Q)$ of standard Young 
tableaux of the same shape with a permutation 
$w \in \fS_n$ shows that the right-hand sides of 
Equations~(\ref{eq:SW2A(p,q,t)}) and~(\ref{eq:SWAn})
are equal and the proof of Theorem~\ref{thm:SWAn} 
follows. 

To prove Theorem~\ref{thm:SWdn}, one combines 
the special case $r=0$ of \cite[Theorem~1.2]{SW10} 
with Equation~(\ref{eq:Bre2}) instead to show that
\begin{equation} \label{eq:DEX2}
\sum_{w \in \dD_n} F_{n, \DEX(w)} (\bx) \, t^{\exc(w)} 
  \ = \ \sum_{\lambda \vdash n} R_\lambda(t) s_\lambda
  (\bx)
\end{equation}
and then takes principal specialization, as before.
\qed

\bigskip
For the proof of Theorem~\ref{thm:ASWAn} we need 
the formulas (see \cite[Section~7.19]{StaEC2}) for 
the stable principal specializations 
\begin{equation} \label{eq:stablespeF*S(x)}
  F^*_{n, S} (1, q, q^2,\dots) \ = \ 
  \frac{q^{\Sum(S)}}{(1-q)(1-q^2) \cdots (1-q^n)}
\end{equation}
and
\begin{equation} \label{eq:stablespeSchur}
  s_\lambda (1, q, q^2,\dots) \ = \ 
  \frac{f^\lambda(q)}{(1-q)(1-q^2) \cdots (1-q^n)},
\end{equation}
where $f^\lambda(q) := f^\lambda(1, q) = 
\sum_{P \in \SYT(\lambda)} q^{\maj(P)}$. 

\bigskip
\noindent
\emph{Proof of Theorem~\ref{thm:ASWAn}}. Let us 
denote by $A_{n,k}(q,t)$ the left-hand side of 
Equation~(\ref{eq:ASWAn}). Comparing the 
coefficients of $r^kz^n$ on both sides of 
\cite[Equation~(1.3)]{SW10} and applying 
Equation~(\ref{eq:Bre2}) gives 
\begin{align*} 
\sum_{w \in \fS_n: \, \fix(w)=k} F_{n, \DEX(w)} 
  (\bx) \, t^{\exc(w)} & \ = \ h_k(\bx) \, [z^{n-k}] 
  \, \frac{1-t} {H(\bx; tz) - tH(\bx; z)} \\
  & \ = \ h_k(\bx) \sum_{\lambda \vdash n-k} 
  R_\lambda(t) s_\lambda(\bx).
\end{align*}
Taking stable principal specialization, we get 
\[ \frac{A_{n,k}(q,t)}{(1-q)(1-q^2) \cdots 
   (1-q^n)} \ = \ \frac{1}{(1-q) \cdots (1-q^k)} 
   \sum_{\lambda \vdash n-k} R_\lambda(t) 
   \frac{f^\lambda(q)}{(1-q) \cdots 
   (1-q^{n-k})}. \]
Finally, we solve for $A_{n,k}(q,t)$ and specialize 
to $p=1$ the expression already obtained for 
the sum $\sum_{\lambda \vdash n-k} R_\lambda(t) 
f^\lambda(p,q)$ in the proof of 
Theorem~\ref{thm:SWdn} and the proof follows.
\qed

\bigskip
For other instances of the use of symmetric 
functions in $\gamma$-positivity, see the 
proof of Theorem~\ref{thm:GGT17} in~\cite{GGT17}
and \cite[Section~4]{Gon16}.

\section{Generalizations}
\label{sec:gen}

This section discusses three possible 
generalizations of $\gamma$-positivity, namely 
nonsymmetric $\gamma$-positivity, equivariant 
$\gamma$-positivity and $q$-$\gamma$-positivity,
which may provide interesting directions for 
future research. 

\subsection{Nonsymmetric $\gamma$-positivity} 
\label{subsec:nonsym}

Nonsymmetric polynomials which can be written 
as a sum of two $\gamma$-positive polynomials,
whose centers of symmetry differ by 1/2, have
appeared in Section~\ref{sec:comb}. We formalize
this situation here as follows. 

A nonzero polynomial $f(x) \in \RR[x]$ is said to 
have center $(i+j)/2$, where $i$ (respectively, 
$j$) is the smallest (respectively, largest) 
integer $k$ for which the coefficient of $x^k$ 
in $f(x)$ is nonzero. One can verify that if 
$f(x)$ has center $n/2$, then it can be written 
uniquely as a sum
\begin{equation} \label{eq:leftgamma}
  f(x) \ = \ f^+_\alpha(x) \, + \, f^{-}_\alpha(x)
\end{equation}
of two symmetric polynomials $f^+_\alpha(x),
f^{-}_\alpha(x)$ with centers of symmetry 
$(n-1)/2$ and $n/2$, respectively. Similarly, 
$f(x)$ can be written uniquely as a sum
\begin{equation} \label{eq:rightgamma}
  f(x) \ = \ f^+_\beta(x) \, + \, f^{-}_\beta(x)
\end{equation}
of two symmetric polynomials $f^+_\beta(x),
f^{-}_\beta(x)$ with centers of symmetry 
$(n+1)/2$ and $n/2$, respectively. We call 
$f(x)$ \emph{left $\gamma$-positive} 
(respectively, \emph{right $\gamma$-positive}) 
if both $f^+_\alpha(x)$ and $f^{-}_\alpha(x)$ 
(respectively, $f^+_\beta(x)$ and 
$f^{-}_\beta(x)$) are $\gamma$-positive. We 
leave it to the reader to verify that $f(x)$ is 
$\gamma$-positive if and only if it is both
left and right $\gamma$-positive. Clearly,
every left $\gamma$-positive 
or right $\gamma$-positive polynomial is unimodal, 
with a peak at its center $n/2$, or at $(n \pm 1)/2$. 
The derangement polynomials $d_{n,r}(x)$ are left 
$\gamma$-positive by 
Theorem~\ref{thm:athdnr}, while the left-hand
sides of Equations~(\ref{eq:lsw1}) 
and~(\ref{eq:lsw2}) and $h$-polynomials of order
complexes which triangulate a ball are right 
$\gamma$-positive by Theorems~\ref{thm:LSWrees}
and~\ref{thm:EKaru}, respectively. It would be 
interesting to find other classes of nonsymmetric
$\gamma$-positive polynomials.

\subsection{Equivariant $\gamma$-positivity} 
\label{subsec:equivariant}

Let $f(t) = \sum_i a_i t^i \in \NN[t]$ be a
$\gamma$-positive polynomial. It often happens 
that $a_i = \dim_\CC (A_i)$ for some non-virtual 
representations $A_i$ of a finite group $G$ 
(all representations considered in this section 
are finite dimensional and defined over the 
field of complex numbers). Then, it is natural 
to consider the polynomial $F(t) := \sum_i A_i 
t^i$, whose coefficients belong to the representation 
ring of $G$, as an equivariant analog of $f(t)$ 
and ask whether 
  \begin{equation} \label{eq:def-equigamma}
    F(t) \ = \ \sum_{i=0}^{\lfloor n/2 \rfloor} 
    \Gamma_i \, t^i (1+t)^{n-2i}
  \end{equation}
for some non-virtual $G$-representations 
$\Gamma_i$, where $n/2$ is the center of symmetry 
of $f(t)$. We then say that $F(t)$ is 
$\gamma$-positive. For the symmetric group $\fS_n$,
this concept reduces, via the Frobenius 
characteristic map, to that of Schur 
$\gamma$-positivity \cite[Section~3]{SW17}.

Especially interesting is the case where $f(t)$ is
the $h$-polynomial of a flag triangulation of the 
sphere, or the local $h$-polynomial of a flag 
triangulation of the simplex, as pointed out by 
Shareshian and Wachs \cite[Sections~5--6]{SW17}.
The following general discussion assumes familiarity 
with face rings of simplicial complexes 
\cite[Chapter~II]{StaCCA} and local face modules 
of triangulations of simplices \cite[Section~4]{Sta92} 
\cite[Section~III.10]{StaCCA} and avoids 
technicalities:
\begin{itemize}
\itemsep=0pt
\item[$\bullet$] Let $\Delta$ be a flag 
triangulation of the $(n-1)$-dimensional 
sphere, on which a finite group $G$ acts 
simplicially. Suppose that $\Theta$ is a 
linear system of parameters for the face 
ring $\CC[\Delta]$ such that the linear 
span of the elements of $\Theta$ is 
$G$-invariant. Then $G$ acts on each 
homogeneous component of the graded vector
space 
\begin{equation} \label{eq:def-facering}
  \CC(\Delta) \ := \ \CC[\Delta] / 
    \langle \Theta \rangle \ = \ 
    \bigoplus_{i=0}^n \, \CC(\Delta)_i,  
\end{equation}
whose graded dimension is equal to the 
$h$-vector of $\Delta$. The pair $(\Delta, G)$ 
exhibits the \emph{equivariant Gal phenomenon}
if there exists $\Theta$ for which $\sum_{i=0}^n 
\CC(\Delta)_i \, t^i$ is $\gamma$-positive.

\item[$\bullet$] Let $\Gamma$ be a flag 
triangulation of the $(n-1)$-dimensional simplex
$2^V$, on which a subgroup $G$ of the automorphism 
group of $2^V$ acts simplicially. 
Suppose that $\Theta$ is a special linear system 
of parameters, in the sense of 
\cite[Definition~4.2]{Sta92}, for the face ring 
of $\Gamma$ such that the linear span of the 
elements of $\Theta$ is $G$-invariant. Then $G$ 
acts on each homogeneous component of the local
face module
\begin{equation} \label{eq:def-localfacering}
  L_V (\Gamma) \ = \ \bigoplus_{i=0}^n \, 
                     L_V (\Gamma)_i 
\end{equation}
defined by $\Theta$, whose graded dimension is 
equal to the local $h$-vector of $\Gamma$. The 
pair $(\Gamma, G)$ exhibits the \emph{local 
equivariant Gal phenomenon} if there exists 
$\Theta$ for which $\sum_{i=0}^n L_V(\Gamma)_i 
\, t^i$ is $\gamma$-positive.
\end{itemize}

When $\Delta$ is the boundary complex of an 
$n$-dimensional simplicial polytope (more 
generally, the simplicial complex associated to 
a complete simplicial fan in $\RR^n$), there
exists $\Theta$ for which (\ref{eq:def-facering}) 
is isomorphic to the cohomology ring of the 
projective toric variety associated to $\Delta$, 
as a graded $G$-representation 
\cite[Section~10]{Da78}. Shareshian and 
Wachs~\cite[Section~5]{SW17} speculate that, in 
interesting situations (but not in general) in 
which $\Delta$ is flag, the pair $(\Delta, G)$ 
exhibits the equivariant Gal phenomenon for such 
linear system of parameters. For example, as 
pointed out by these authors, and in view of 
results of C.~Procesi~\cite{Pro90} and R.~Stanley 
\cite[Proposition~12]{Sta89} 
\cite[Proposition~4.20]{Sta92}, the Schur 
$\gamma$-positivity of the coefficient of 
$z^n$ in the left-hand sides of 
Equations~(\ref{eq:Bre1}) and~(\ref{eq:Bre2}), 
already discussed in Section~\ref{subsec:sym}, 
are instances of the equivariant Gal phenomenon 
and its local analog, with corresponding flag 
triangulations the barycentric subdivisions of 
the boundary of the $(n-1)$-dimensional simplex 
and the simplex itself. Another instance of the 
equivariant Gal phenomenon is the Schur 
$\gamma$-positivity of the coefficient of $z^n$ 
in 
\begin{equation} \label{eq:genrefbinom}
  \frac{(1-t) H(\bx; z) H(\bx; tz)} 
        {H(\bx; tz) - tH(\bx; z)},
\end{equation}
established in \cite[Theorem~3.4]{SW17}, with 
corresponding flag triangulation the boundary 
complex of the $n$-dimensional simplicial 
stellohedron. This result provides an equivariant
analog for the $\gamma$-positivity 
of the binomial Eulerian polynomials, discussed 
in Section~\ref{subsec:Euler}.

A very interesting example is the Coxeter 
complex $\cox(W)$, discussed in 
Section~\ref{subsubsec:coxeter}, on which the 
finite Coxeter group $W$ acts. When $W$ is 
crystallographic, its representation on the 
cohomology of the projective toric variety 
associated to $\cox(W)$ was studied by 
C.~Procesi~\cite{Pro90}, 
R.~Stanley~\cite[p.~529]{Sta89}, I.~Dolgachev and 
V.~Lunts~\cite{DL94}, J.~Stembridge~\cite{Ste94}, 
G.~Lehrer~\cite{Le08} and 
A.~Stapledon \cite[Section~8]{Sta11}. Given the 
evidence provided in the sequel, it is expected 
that the following conjectural equivariant 
analog of Theorem~\ref{thm:W(x)} can be shown
using the classification of finite Coxeter groups
(a proof which does not use the classification is
certainly more desirable).
\begin{conjecture} \label{con:eCox}
The pair $(\cox(W), W)$ exhibits the equivariant 
Gal phenomenon for every finite crystallographic
Coxeter group $W$.
\end{conjecture}

As mentioned earlier, in the symmetric group 
case the conjecture follows from the 
$\gamma$-positivity of the polynomials $P_\lambda 
(t)$, shown in Corollary~\ref{cor:PRgamma}. In the
case of the hyperoctahedral group $\bB_n$, by 
\cite[Theorem~6.3]{DL94} or 
\cite[Theorem~7.6]{Ste94}, the Frobenius 
characteristic of the graded 
$\bB_n$-representation on the (even degree) 
cohomology of the projective toric variety 
associated to $\cox(\bB_n)$ is equal to the 
coefficient of $z^n$ in 
\begin{equation} \label{eq:genrefEulerB}
  \frac{(1-t) H(\bx; z) H(\bx; tz)} 
  {H(\bx, \by; tz) - tH(\bx, \by; z)},
\end{equation}
where $\by = (y_1, y_2,\dots)$ is another 
sequence of commuting independent 
indeterminates and $H(\bx, \by; z) = 
\sum_{n \ge 0} h_n(\bx, \by) z^n = H(\bx; z)
H(\by; z)$ is the generating function of the 
complete homogeneous symmetric functions $h_n
(\bx, \by)$ in the variable $z$. Thus, 
Conjecture~\ref{con:eCox} in this case is 
implied by the following result (the proof, 
which uses Theorem~\ref{thm:Athrees} stated in 
the sequel, will appear in the final version 
of~\cite{Ath17} or elsewhere).
\begin{proposition}
The coefficient of $z^n$ in (\ref{eq:genrefEulerB}) 
is Schur $\gamma$-positive for every $n \in \NN$.
\end{proposition}

For example, writing 
  $$ \frac{(1-t) H(\bx; z) H(\bx; tz)} 
     {H(\bx, \by; tz) - tH(\bx, \by; z)} \ = \ 
     1 \ + \ \sum_{n \ge 1} z^n
     \sum_{i=0}^{\lfloor n/2 \rfloor} 
     \gamma^B_{n,i}(\bx,\by) \, t^i (1 + t)^{n-2i}, $$
we have 
\begin{align*} 
\gamma^B_{1,0} (\bx,\by) & \ = \ s_{(1)}(\bx), \\
\gamma^B_{2,0} (\bx,\by) & \ = \ s_{(2)}(\bx), \\
\gamma^B_{2,1} (\bx,\by) & \ = \ s_{(1,1)}(\bx) + 
          s_{(1)}(\bx) s_{(1)}(\by) + s_{(2)}(\by), \\
\gamma^B_{3,0} (\bx,\by) & \ = \ s_{(3)}(\bx), \\
\gamma^B_{3,1} (\bx,\by) & \ = \ 2s_{(2,1)}(\bx) + 
          s_{(1,1)}(\bx) s_{(1)}(\by) + 2s_{(2)}(\bx) 
          s_{(1)}(\by) + 2s_{(1)}(\bx) s_{(2)}(\by) + 
          s_{(3)}(\by). 
\end{align*}
At present, a combinatorial interpretation of the 
coefficients of the functions $\gamma^B_{n,i}
(\bx,\by)$ in the Schur basis is lacking. 

Other evidence for the validity of 
Conjecture~\ref{con:eCox} is provided by the 
computation of the graded multiplicity of the 
trivial \cite[Section~3]{Ste94}
and of the sign \cite[Theorem~3.5~(iii)]{Le08} 
and reflection representation \cite[Section~4]{Le08} 
of $W$ in the cohomology of the toric variety 
associated to $\cox(W)$; see also 
\cite[Remark~8.3]{Sta11}. 

A general result on equivariant (nonsymmetric) 
$\gamma$-positivity is the following equivariant
version of Theorem~\ref{thm:LSWrees}. Let $\pP$
be a finite graded poset, as in Section~\ref{subsec:homo},
and assume that a finite group $G$ acts on $\pP$ 
by order preserving bijections. Then $G$ defines a 
permutation representation $\alpha_\pP(S)$, induced 
by its action on the set of maximal chains of 
the rank-selected subposet $\pP_S$, for every $S 
\subseteq [n]$. One can consider the virtual 
$G$-representation
\begin{equation} \label{eq:beta(S)}
\beta_\pP(S) \ = \ \sum_{T \subseteq S} (-1)^{|S-T|}
\, \alpha_\pP(T),
\end{equation}
introduced by Stanley~\cite{Sta82}. When $\pP$ is 
Cohen--Macaulay over $\CC$, $\beta_\pP(S)$ coincides
with the non-virtual $G$-representation induced on 
the top homology group of $\bar{\pP}_S$ (see 
\cite{Sta82} \cite[Section~3.4]{Wa07}); the dimensions 
of $\alpha_\pP(S)$ and $\beta_\pP(S)$ are equal to 
the numerical invariants $a_\pP(S)$ and $b_\pP(S)$, 
respectively, defined in Section~\ref{subsec:homo}.
\begin{theorem} \label{thm:Athrees}
{\rm (Athanasiadis~\cite[Theorem~1.2]{Ath17})} 
Let $G$ be a finite group acting on a finite bounded 
graded poset $\pP$ of rank $n+1$ by order preserving 
bijections. If $\pP$ is Cohen--Macaulay over $\CC$,
then 
\begin{align} \label{eq:main1}
    \widetilde{H}_{n-1} ((\pP^- \ast T_{t,n})_{-}; \CC)  
  \ \cong_G & \ 
    \sum_{S \in \Stab ([n-1])}
    \beta_\pP ([n] \sm S) \, t^{|S|} (1+t)^{n-2|S|} 
    \ \ + \\ & \nonumber \\
  & \ \sum_{S \in \Stab ([n-2])} \beta_\pP 
    ([n-1] \sm S) \, t^{|S|+1} (1+t)^{n-1-2|S|} 
    \nonumber
  \end{align}
and 
  \begin{align} \label{eq:main2}
  \widetilde{H}_{n-1} (\bar{\pP} \ast T_{t,n-1}; \CC) 
  \ \cong_G & \ 
    \sum_{S \in \Stab ([2, n-2])} \beta_\pP 
    ([n-1] \sm S) \, t^{|S|+1} (1+t)^{n-2-2|S|} 
    \ \ + \\ & \nonumber \\
  & \ \sum_{S \in \Stab ([2, n-1])}
    \beta_\pP ([n] \sm S) \, t^{|S|} (1+t)^{n-1-2|S|} 
    \nonumber
  \end{align}
for every positive integer $t$, where $\cong_G$ 
denotes isomorphism of $G$-representations.

In particular, the left-hand sides of (\ref{eq:main1}) 
and (\ref{eq:main2}) are right $\gamma$-positive.
\end{theorem}

As shown in \cite[Corollary~4.1]{Ath17}, the Schur 
$\gamma$-positivity of the coefficients of the 
left-hand sides of Equations~(\ref{eq:Ge1})
and~(\ref{eq:Ge2}) follows from the special case
of Theorem~\ref{thm:Athrees} in which $\pP^{-}$ 
is the Boolean lattice of subsets of $[n]$. Two 
more applications, establishing the Schur 
$\gamma$-positivity of the coefficients of two 
close relatives to (\ref{eq:genrefEulerB}), 
are given in \cite[Section~4]{Ath17}. We state 
these results here in a form similar to that of
Gessel's identities (\ref{eq:Ge1})--(\ref{eq:Ge4})
as follows (a more elementary proof should be 
possible). For a map $w: [n] \to \ZZ \sm \{0\}$ 
we write $\bz_w := z_{w(1)} z_{w(2)} \cdots 
z_{w(n)}$, where $z_{w(i)} = x_{w(i)}$ if $w(i) 
> 0$ and $z_{w(i)} = y_{w(i)}$ otherwise, and 
define $\Asc_B(w)$ as the set of indices $i \in 
[n]$ such that $w(i) >_r w(i+1)$ in the total 
order (\ref{eq:def<r}), where $w(n+1) := 0$. 
\begin{theorem} \label{thm:AthGeB}
{\rm (cf.~\cite[Corollaries~4.4 and~4.7]{Ath17})} 
We have 
\begin{align} \label{eq:Ath1}
\frac{(1-t) E(\bx; z) E(\bx; tz) E(\by; tz)} 
{E(\bx, \by; tz) - tE(\bx, \by; z)} 
\ = \ 1 & \ + \
\sum_{n \ge 1} z^n \, \sum_w
t^{\asc(w)} (1+t)^{n+1 - 2\asc(w)} \, \bz_w
\\ & \ + \ \sum_{n \ge 1} z^n \, \sum_w
t^{\asc(w)} (1+t)^{n - 2\asc(w)} \, \bz_w \nonumber
\end{align}
where, in the two sums, $w: [n] \to \ZZ \sm \{0\}$ 
runs through all maps for which $\Asc_B(w) \in 
\Stab([n])$ contains (respectively, does not 
contain) $n$ and
\begin{align} \label{eq:Ath2}
\frac{(1-t) E(\bx; z)} 
{E(\bx, \by; tz) - tE(\bx, \by; z)} 
\ = \ 1 & \ + \ \sum_{n \ge 1} z^n \,
\sum_w t^{\asc(w)} (1+t)^{n - 2\asc(w)} \, \bz_w
\\ & \ + \ \sum_{n \ge 1} z^n \, \sum_w
t^{\asc(w)} (1+t)^{n-1 - 2\asc(w)} \, \bz_w
\nonumber
\end{align}
where, in the two sums, $w: [n] \to \ZZ \sm \{0\}$ 
runs through all maps for which $\Asc_B(w) \in 
\Stab([2,n])$ contains (respectively, does not 
contain) $n$. 
\end{theorem}

\subsection{$q$-$\gamma$-positivity} 
\label{subsec:q-gamma}

Let $a, b \in \NN$. A polynomial $f(q, x) \in \RR[q, 
x]$ is called \emph{$q$-$\gamma$-positive} of type 
$(a, b)$ if  
\begin{equation} \label{eq:q-gamma}
  f(q, x) \ = \ \sum_{i=0}^{\lfloor n/2 \rfloor}
  \gamma_i (q) \, x^i \prod_{k=i}^{n-1-i} (1 + 
  xq^{ka+b})
\end{equation}
for some polynomials $\gamma_i (q) \in \RR[q]$ 
with nonnegative coefficients. This concept, 
which reduces to that of $\gamma$-positivity 
for $q=1$, was formally introduced by K.~Dilks 
\cite[Chapter~4]{Di15} and, in more restrictive
form, by C.~Krattenthaler and M.~Wachs 
(unpublished), after such expansions were found
for certain $q$-analogs of the Eulerian 
polynomials for symmetric and hyperoctahedral 
groups by G.~Han, F.~Jouhet and J.~Zeng 
\cite{HJZ13}. Other polynomials which admit 
$q$-analogs which are $q$-$\gamma$-positive
include binomials and Narayana polynomials for 
symmetric and hyperoctahedral groups; see 
\cite[Chapter~4]{Di15}. 

\bigskip
\noindent \textbf{Acknowledgements}. 
The author wishes to thank Francesco Brenti, 
Persi Diaconis, Satoshi Murai, Kyle Petersen, 
Richard Stanley, John Stembridge, Vasu Tewari 
and Jiang Zeng for useful comments and information 
and the SLC organizers Christian Krattenthaler, 
Volker Strehl and Jean-Yves Thibon for their 
invitation.

\medskip
\noindent \textbf{Note added in revision.} 
Conjecture~\ref{conj:dnr} has been proven by 
P.~Br\"and\'en and L.~Solus~\cite[Section~3.3]{BrS18} 
and, independently, by N.~Gustafsson and 
L.~Solus~\cite[Section~5.1]{GS18}.


\begin{thebibliography}{99}
%
\bibitem{AAER17}
R.M.~Adin, C.A.~Athanasiadis, S.~Elizalde and 
Y.~Roichman,
\textit{Character formulas and descents for the 
hyperoctahedral group},
Adv. in Appl. Math. {\bf~87} (2017), 128--169.
%
\bibitem{Ai14}
N.~Aisbett,
\textit{Frankl--F\"uredi--Kalai inequalities for
the $\gamma$-vectors of flag nestohedra},
Discrete Comput. Geom. {\bf~51} (2014), 323--336.
%
\bibitem{Arm09}
D.~Armstrong,
\textit{Generalized noncrossing partitions and 
combinatorics of Coxeter groups}, 
Mem. Amer. Math. Soc. {\bf~202} (2009), no. 949.
%
\bibitem{Ath11}
C.A.~Athanasiadis,
\textit{Some combinatorial properties of flag 
simplicial pseudomanifolds and spheres},
Ark. Mat. {\bf~49} (2011), 17--29.
%
\bibitem{Ath12}
C.A.~Athanasiadis,
\textit{Flag subdivisions and $\gamma$-vectors},
Pacific J. Math. {\bf~259} (2012), 257--278.
%
\bibitem{Ath14}
C.A.~Athanasiadis,
\textit{Edgewise subdivisions, local $h$-polynomials 
and excedances in the wreath product 
${\mathbb Z}_r \wr \mathfrak{S}_n$},
SIAM J. Discrete Math. {\bf~28} (2014), 1479--1492.
%
\bibitem{Ath16a}
C.A.~Athanasiadis,
\textit{A survey of subdivisions and local $h$-vectors},
in \textit{The Mathematical Legacy of Richard~P.~Stanley}
(P.~Hersh, T.~Lam, P.~Pylyavskyy, V.~Reiner, eds.),
Amer. Math. Society, Providence, RI, 2016, pp.~39--51.
%
\bibitem{Ath16b}
C.A.~Athanasiadis,
\textit{The local $h$-polynomial of the edgewise 
subdivision of the simplex},
Bull. Hellenic Math. Soc. (N.S.) {\bf~60} (2016), 
11--19.
%
\bibitem{Ath17}
C.A.~Athanasiadis,
\textit{Some applications of Rees products of posets 
to equivariant gamma-positivity},
{\tt arXiv:1707.08297}.
%
\bibitem{AS12}
C.A.~Athanasiadis and C.~Savvidou,
\textit{The local $h$-vector of the cluster subdivision 
of a simplex},
S\'em. Lothar. Combin. {\bf~66} (2012), Article B66c, 
21pp (electronic).
%
\bibitem{BB07}
E.~Bagno and R.~Biagioli,
\textit{Colored descent representations for complex 
reflection groups},
Israel J. Math. {\bf~160} (2007), 317--347.
%
\bibitem{BG06}
E.~Bagno and D.~Garber,
\textit{On the excedance number of colored permutation groups},
S\'em. Lothar. Combin. {\bf~53} (2006), Article B53f, 
17pp (electronic).
%
\bibitem{BBS09}
M.~Barnabei, F.~Bonetti and M.~Silimbani,
\textit{The descent statistic in involutions is not
log-concave},
European J. Combin. {\bf~30} (2009), 11--16.
%
\bibitem{BR18}
R.~Barnard and N.~Reading,
\textit{Coxeter-biCatalan combinatorics},
J. Algebr. Comb. {\bf~47} (2018), 241--300.
%
\bibitem{Ba73}
D.~Barnette,
\textit{A proof of the lower bound conjecture 
for convex polytopes},
Pacific J. Math. {\bf~46} (1973), 349--354.
%
\bibitem{Be03}
D.~Bessis, 
\textit{The dual braid monoid}, 
Ann. Sci. Ecole Norm. Sup. {\bf~36} (2003), 647--683. 
%
\bibitem{BM16}
S.C.~Billey and P.R.W.~McNamara,
\textit{The contributions of Stanley to the fabric 
of symmetric and quasisymmetric functions},
in \textit{The Mathematical Legacy of Richard~P.~Stanley}
(P.~Hersh, T.~Lam, P.~Pylyavskyy, V.~Reiner, eds.),
Amer. Math. Society, Providence, RI, 2016, pp.~83--104.
%
\bibitem{Bj84}
A.~Bj\"orner,
\textit{Some combinatorial and algebraic properties 
of Coxeter complexes and Tits buildings},
Adv. Math. {\bf~52} (1984), 173--212.
%
\bibitem{Bj95}
A.~Bj\"orner,
\textit{Topological methods}, 
in \textit{Handbook of combinatorics}
(R.L.~Graham, M.~Gr\"otschel and L.~Lov\'asz, eds.),
North Holland, Amsterdam, 1995, pp.~1819--1872.
%
\bibitem{BB05}
A.~Bj\"orner and F.~Brenti,
Combinatorics of Coxeter groups,
Graduate Texts in Mathematics {\bf~231}, Springer, 
2005.
%
\bibitem{BW05}
A.~Bj\"orner and V.~Welker,
\textit{Segre and Rees products of posets, with 
ring-theoretic applications},
J. Pure Appl. Algebra {\bf~198} (2005), 43--55.
%
\bibitem{BP14}
S.~Blanco and T.K.~Petersen,
\textit{Counting Dyck paths by area and rank},
Ann. Comb. {\bf~18} (2014), 171--197.
%
\bibitem{BW02}
T.~Brady and C.~Watt,
\textit{K($\pi$, 1)'s for Artin groups of finite type},
in \textit{Proceedings of the Conference on Geometric 
and Combinatorial group theory}, Part I (Haifa 2000),
Geom. Dedicata {\bf~94} (2002), 225--250.
%
\bibitem{Bra04}
P.~Br\"and\'en,
\textit{Sign-graded posets, unimodality of $W$-polynomials
and the Charney-Davis conjecture},
in \textit{Stanley Festschrift} (B.E.~Sagan, ed.), 
Electron. J. Combin. {\bf~11} (2004), no. 2, Research 
Paper 9, 15pp (electronic).
%
\bibitem{Bra04b}
P.~Br\"and\'en,
\textit{Counterexamples to the Neggers-Stanley conjecture}, 
Electron. Res. Announc. Amer. Math. Soc. {\bf~10} (2004), 
155--158 (electronic).
%
\bibitem{Bra08}
P.~Br\"and\'en,
\textit{Actions on permutations and unimodality of 
descent polynomials},
European J. Combin. {\bf~29} (2008), 514--531.
%
\bibitem{Bra15}
P.~Br\"and\'en,
\textit{Unimodality, log-concavity, real-rootedness 
and beyond},
in \textit{Handbook of Combinatorics} (M.~Bona, ed.), 
CRC Press, 2015, pp.~437--483.
%
\bibitem{BrS18}
P.~Br\"and\'en and L.~Solus,
\textit{Symmetric decompositions and real-rootedness},
{\tt arXiv:1808.04141}.
%
\bibitem{Bre89}
F.~Brenti,
\textit{Unimodal, log-concave and P\'olya frequency 
sequences in combinatorics},
Mem. Amer. Math. Soc., no. 413, 1989.
%
\bibitem{Bre90}
F.~Brenti,
\textit{Unimodal polynomials arising from symmetric functions},
Proc. Amer. Math. Soc. {\bf~108} (1990), 1133--1141.
%
\bibitem{Bre94a}
F.~Brenti,
\textit{$q$-Eulerian polynomials arising from 
Coxeter groups},
European J. Combin. {\bf~15} (1994), 417--441.
%
\bibitem{Bre94b}
F.~Brenti,
\textit{Log-concave and unimodal sequences in algebra,
combinatorics and geometry: an update},
in \textit{Jerusalem Combinatorics '93}
(H.~Barcelo and G.~Kalai, eds.),
Contemp. Math. {\bf~178}, Amer. Math. Society, 
Providence, RI, 1994, pp.~71--89.
%
\bibitem{BW08}
F.~Brenti and V.~Welker,
\textit{$f$-Vectors of barycentric subdivisions},
Math. Z. {\bf~259} (2008), 849--865.
%
\bibitem{BW09}
F.~Brenti and V.~Welker,
\textit{The Veronese construction for formal power 
series and graded algebras},
Adv. in Appl. Math. {\bf~42} (2009), 545--556.
%
\bibitem{BR05}
M.~Brun and T.~R\"omer,
\textit{Subdivisions of toric complexes},
J. Algebr. Comb. {\bf~21} (2005), 423--448.
%
\bibitem{BP15}
V.M.~Buchstaber and T.E.~Panov,
Toric Topology,
Mathematical Surveys and Monographs {\bf~204},
Amer. Math. Society, 2015.
%
\bibitem{CD95}
R.~Charney and M.~Davis,
\textit{The Euler characteristic of a nonpositively 
curved, piecewise Euclidean manifold},
Pacific J. Math. {\bf~171} (1995), 117--137.
%
\bibitem{CTZ09}
W.Y.C.~Chen, R.L.~Tang and A.F.Y.~Zhao,
\textit{Derangement polynomials and excedances of 
type $B$},
Electron. J. Combin. {\bf~16} (2) (2009), 
Research Paper 15, 16pp (electronic).
%
\bibitem{Ch08}
C.-O.~Chow,
\textit{On certain combinatorial expansions of the
Eulerian polynomials},
Adv. in Appl. Math. {\bf~41} (2008), 133--157.
%
\bibitem{Ch09}
C.-O.~Chow,
\textit{On derangement polynomials of type $B$. II},
J. Combin. Theory Series A {\bf~116} (2009), 816--830.
%
\bibitem{CM10}
C.-O.~Chow and T.~Mansour,
\textit{Counting derangements, involutions and 
unimodal elements in the wreath product 
$C_r \wr \mathfrak{S}_n$},
Israel J. Math. {\bf~179} (2010), 425--448.
%
\bibitem{CGK10}
F.~Chung, R.~Graham and D.~Knuth,
\textit{A symmetrical Eulerian identity},
J. Combinatorics {\bf~1} (2010), 29--38.
%
\bibitem{Da78}
V.I.~Danilov,
\textit{The geometry of toric vatieties},
Russian Math. Surveys {\bf~33} (1978), 97--154;
translated from Uspekhi 
Mat. Nauk. {\bf~33} (1978), 85--134.
%
\bibitem{DO01}
M.~Davis and B.~Okun,
\textit{Vanishing theorems and conjectures for 
the $L^2$-homology of right-angled Coxeter groups},
Geom. Topol. {\bf~5} (2001), 7--74.
%
\bibitem{DRS10}
J.A.~De Loera, J.~Rambau and F.~Santos,
Triangulations: Structures for Algorithms and Applications,
Algorithms and Computation in Mathematics {\bf~25},
Springer, 2010.
%
\bibitem{DF85}
J.~D\'esarm\'enien and D.~Foata,
\textit{Fonctions sym\'etriques et s\'eries
hyperg\'eom\'etriques basiques multivari\'ees},
Bull. Soc. Math. France {\bf~113} (1985), 3--22.
%
\bibitem{Di15}
K.~Dilks,
\textit{Involutions on Baxter objects, and 
$q$-gamma nonnegativity},
Ph.D thesis, University of Minnesota, 2015.
%
\bibitem{DPS09}
K.~Dilks, T.K.~Petersen and J.R.~Stembridge,
\textit{Affine descents and the Steinberg torus},
Adv. in Appl. Math. {\bf~42} (2009), 423--444.
%
\bibitem{DL94}
I.~Dolgachev and V.~Lunts,
\textit{A character formula for the representation of 
a Weyl in the cohomology of the associated toric 
variety},
J. Algebra {\bf~168} (1994), 741--772.
%
\bibitem{ER86}
\"O.~E\u{g}ecio\u{g}lu and J.B.~Remmel,
\textit{Bijections for Cayley trees, spanning 
trees, and their $q$-analogues},
J. Combin. Theory Series A {\bf~42} (1986), 15--30.
%
\bibitem{EK07}
R.~Ehrenborg and K.~Karu,
\textit{Decomposition theorem for the cd-index of 
Gorenstein* posets},
J. Algebr. Comb. {\bf~26} (2007), 225--251.
%
\bibitem{FS05}
E.-M.~Feichtner and B.~Sturmfels,
\textit{Matroid polytopes, nested sets and Bergman
fans},
Port. Math. (N.S.) {\bf~62} (2005), 437--468.
%
\bibitem{FY04}
E.-M.~Feichtner and S.~Yuzvinsky,
\textit{Chow rings of toric varieties defined by 
atomic lattices},
Invent. Math. {\bf~155} (2004), 515--536.
%
\bibitem{FH11}
D.~Foata and G.-N.~Han,
\textit{The decrease value theorem with an application 
to permutation statistics},
Adv. in Appl. Math. {\bf~46} (2011), 296--311.
%
\bibitem{Fo72}
D.~Foata,
\textit{Groupes de r\'earrangements et nombres d'Euler},
C.R. Acad. Sci. Paris Sr. A--B {\bf~275} (1972), 
1147--1150.
%
\bibitem{FSc70}
D.~Foata and M.-P.~Sch\"utzenberger,
Th\'eorie G\'eometrique des Polyn\^omes Eul\'eriens,
Lecture Notes in Mathematics {\bf~138}, Springer-Verlag, 
1970.
%
\bibitem{FS74}
D.~Foata and V.~Strehl,
\textit{Rearrangements of the symmetric group and 
enumerative properties of the tangent and secant 
numbers},
Math. Z. {\bf~137} (1974), 257--264.
%
\bibitem{FS76}
D.~Foata and V.~Strehl,
\textit{Euler numbers and variations of permutations},
in \textit{Colloquio Internazionale sulle Teorie
Combinatorie (Roma, 1973), Tomo I},
Atti dei Convegni Lincei, No. 17, 
Accad. Naz. Lincei, Rome, 1976, pp.~119--131.
%
\bibitem{FR07}
S.~Fomin and N.~Reading,
\textit{Root systems and generalized associahedra},
in \textit{Geometric Combinatorics}
(E.~Miller, V.~Reiner and B.~Sturmfels, eds.),
IAS/Park City Mathematics Series {\bf~13}, pp.~389--496,
Amer. Math. Society, Providence, RI, 2007.
%
\bibitem{FZ03}
S.~Fomin and A.V.~Zelevinsky,
\textit{$Y$-systems and generalized associahedra},
Ann. of Math. {\bf~158} (2003), 977--1018.
%
\bibitem{Ga05}
\'S.R.~Gal,
\textit{Real root conjecture fails for five- and 
higher-dimensional spheres},
Discrete Comput. Geom. {\bf~34} (2005), 269--284.
%
\bibitem{Ga98}
V.~Gasharov,
\textit{On the Neggers--Stanley conjecture and the 
Eulerian polynomials},
J. Combin. Theory Series A {\bf~82} (1998), 
134--146.
%
\bibitem{Ge96}
I.M.~Gessel, 
\textit{Counting forests by descents and leaves},
Electron. J. Combin. {\bf~3}(2) (1996), 
Research Paper 8, 5pp (electronic).
%
\bibitem{GGT17}
I.M.~Gessel, S.~Griffin and V.~Tewari,
\textit{Labeled plane binary trees and Schur-positivity},
{\tt arXiv:1706. 03055}.
%
\bibitem{GR93}
I.M.~Gessel and C.~Reutenauer, 
\textit{Counting permutations with given cycle 
structure and descent set}, 
J. Combin. Theory Series A~{\bf~64} (1993), 189--215.
%
\bibitem{Gon16}
R.S.~Gonz\'alez D'Le\'on, 
\textit{A note on the $\gamma$-coefficients of the
tree Eulerian polynomial},
Electron. J. Combin. {\bf~23} (2016), Research Paper 
1.20, 13pp (electronic).
%
\bibitem{Go10}
M.A.~Gorsky,
\textit{Proof of Gal's conjecture for the $D$ series
of generalized associahedra},
Russian Math. Surveys {\bf~65} (2010), 1178--1180.
%
\bibitem{GZ06}
V.W.J.~Guo and J.~Zeng, 
\textit{The Eulerian distribution on involutions is 
indeed unimodal}, 
J. Combin. Theory Series A {\bf~113} (2006), 
1061--1071.
%
\bibitem{GS18}
N.~Gustafsson and L.~Solus,
\textit{Derangements, Ehrhart theory and local 
$h$-polynomials},
{\tt arXiv:1807. 05246}.
%
\bibitem{HJZ13}
G.~Han, F.~Jouhet and J.~Zeng, 
\textit{Two new triangles of $q$-integers via 
$q$-Eulerian polynomials of type $A$ and $B$}, 
Ramanujan J. {\bf~31} (2013), 115--127.
%
\bibitem{Hu07}
A.~Hultman,
\textit{The combinatorics of twisted involutions in
Coxeter groups},
Trans. Amer. Math. Soc. {\bf~359} (2007), 2787--2798.
%
\bibitem{Isa76}
I.M.~Isaacs,
Character Theory of Finite Groups,
Pure and Applied Mathematics,
Academic Press, 1976.
%
\bibitem{KMS17}
M.~Juhnke-Kubitzke, S.~Murai and R.~Sieg,
\textit{Local $h$-vectors of quasi-geometric and 
barycentric subdivisions},
Discrete Comput. Geom. (to appear).
%
\bibitem{Ka06}
K.~Karu,
\textit{The cd-index of fans and posets},
Compos. Math. {\bf~142} (2006), 701--718.
%
\bibitem{KSt16}
E.~Katz and A.~Slapledon,
\textit{Local $h$-polynomials, invariants of 
subdivisions and mixed Ehrhart theory},
Adv. Math. {\bf~286} (2016), 181--239.
%
\bibitem{Kl64}
V.~Klee,
\textit{A combinatorial analogue of Poincar\'e's
duality theorem},
Canad. J. Math. {\bf~16} (1964), 517--531.
%
\bibitem{Koz08}
D.~Kozlov,
Combinatorial Algebraic Topology,
Algorithms and Computation in Mathematics {\bf~21}, 
Springer, 2008.
%
\bibitem{Le08}
G.I.~Lehrer,
\textit{Rational points and Coxeter group actions 
on the cohomology of toric varieties},
Ann. Inst. Fourier (Grenoble) {\bf~58} (2008), 671--688.
%
\bibitem{Lin14}
Z.~Lin,
\textit{Eulerian calculus arising from permutation 
statistics},
Doctoral Dissertation, Universit\'e Claude 
Bernard-Lyon 1, 2014.
%
\bibitem{Li16}
Z.~Lin, 
\textit{Proof of Gessel's $\gamma$-positivity conjecture},
Electron. J. Combin. {\bf~23} (2016), Research Paper 3.15, 
9pp (electronic).
%
\bibitem{Li17}
Z.~Lin, 
\textit{On $\gamma$-positive polynomials arising
in pattern avoidance},
Adv. in Appl. Math. {\bf~82} (2017), 1--22.
%
\bibitem{LF17}
Z.~Lin and S.~Fu, 
\textit{On 1212-avoiding restricted growth functions},
Electron. J. Combin. {\bf~24} (2017), Research Paper 1.53, 
20pp (electronic).
%
\bibitem{LZ15}
Z.~Lin and J.~Zeng, 
\textit{The $\gamma$-positivity of basic Eulerian 
polynomials via group actions}, 
J. Combin. Theory Series A {\bf~135} (2015), 112--129.
%
\bibitem{LSW12}
S.~Linusson, J.~Shareshian and M.L.~Wachs,
\textit{Rees products and lexicographic shellability},
J. Combinatorics {\bf~3} (2012), 243--276.
%
\bibitem{MR15}
A.~Mendes and J.~Remmel,
Counting with Symmetric Functions,
Developments in Mathematics {\bf~43}, Springer, 
2015.
%
\bibitem{Mo13}
P.~Mongelli,
\textit{Excedances in classical and affine Weyl groups},
J. Combin. Theory Series A {\bf~120} (2013), 1216--1234.
%
\bibitem{MouMS}
V-D.P.~Moustakas,
\textit{The Eulerian distribution on the involutions
of the hyperoctahedral group} (in Greek),
Master Thesis, University of Athens, 2017.
%
\bibitem{Mu17}
H.~M\"uhle,
\emph{Symmetric decompositions and the strong Sperner
property for noncrossing partition lattices},
J. Algebr. Comb. {\bf~45} (2017), 745--775.
%
\bibitem{Ne78}
J.~Neggers,
\emph{Representations of finite partially ordered sets},
J. Comb. Inform. Syst. Sci. {\bf~3} (1978), 113--133.
%
\bibitem{Ne09}
E.~Nevo,
\textit{Remarks on missing faces and generalized lower 
bounds on face numbers},
Electron. J. Combin. {\bf~16} (2009), no. 2, Research 
Paper 8, 11pp (electronic).
%
\bibitem{NP11}
E.~Nevo and T.K.~Petersen,
\textit{On gamma vectors satisfying the 
Kruskal--Katona inequalities},
Discrete Comput. Geom. {\bf~45} (2011), 503--521.
%
\bibitem{NPT11}
E.~Nevo, T.K.~Petersen and B.E.~Tenner,
\textit{The gamma-vector of a barycentric subdivision},
J. Combin. Theory Series A {\bf~118} (2011), 
1364--1380.
%
\bibitem{Ni12}
B.~Nill and J.~Schepers,
\textit{Combinatorial questions related to stringy 
$E$-polynomials of Gorenstein polytopes},
in \textit{Toric Geometry} (K.~Altmann et. al., eds.),
Oberwolfach Reports (2012), no. 21, pp.~62--64.
%
\bibitem{Pet07}
T.K.~Petersen,
\textit{Enriched $P$-partitions and peak algebras},
Adv. Math. {\bf~209} (2007), 561--610.
%
\bibitem{Pet13a}
T.K.~Petersen,
\textit{On the shard intersection order of a Coxeter 
group},
SIAM J. Discrete Math. {\bf~27} (2013), 1880--1912.
%
\bibitem{Pet13b}
T.K.~Petersen,
\textit{Two-sided Eulerian numbers via balls in boxes},
Math. Mag. {\bf~86} (2013), 159--176.
%
\bibitem{Pet15}
T.K.~Petersen,
Eulerian Numbers,
Birkh\"auser Advanced Texts, Birkh\"auser, 2015.
%
\bibitem{Pet16}
T.K.~Petersen,
\textit{A two-sided analogue of the Coxeter complex},
{\tt arXiv:1607.00086}.
%
\bibitem{Poi98}
S.~Poirier, \textit{Cycle type and descent set 
in wreath products}, 
Discrete Math.~{\bf~180} (1998), 315--343.
%
\bibitem{Pos97}
A.~Postnikov,
\textit{Intransitive trees},
J. Combin. Theory Series A {\bf~79} (1997), 
360--366.
%
\bibitem{Pos09}
A.~Postnikov,
\textit{Permutohedra, associahedra, and beyond},
Int. Math. Res. Not. {\bf~2009} (2009), 
1026--1106.
%
\bibitem{PRW08}
A.~Postnikov, V.~Reiner and L.~Williams,
\textit{Faces of generalized permutohedra},
Doc. Math. {\bf~13} (2008), 207--273.
%
\bibitem{Pro90}
C.~Procesi,
\textit{The toric variety associated to Weyl chambers},
in \textit{Mots} (M.~Lothaire, ed.),
Herm\'es, Paris, 1990, pp.~153--161.
%
\bibitem{RW05}
V.~Reiner and V.~Welker, 
\textit{On the Charney--Davis and Neggers--Stanley 
conjectures}, 
J. Combin. Theory Series A {\bf~109} (2005), 
247--280.
%
\bibitem{Rei76}
G.~Reisner,
\textit{Cohen--Macaulay quotients of polynomial rings},
Adv. Math. {\bf~21} (1976), 30--49.
%
\bibitem{Roi14}
Y.~Roichman,
\textit{A note on the number of $k$-roots in $S_n$},
S\'em. Lothar. Combin. {\bf~70} (2014), Article B70i, 
5pp (electronic).
%
\bibitem{SV15}
C.D.~Savage and M.~Visontai,
\textit{The $s$-Eulerian polynomials have only real 
roots},
Trans. Amer. Math. Soc. {\bf~367} (2015), 763--788.
%
\bibitem{Sav13}
C.~Savvidou,
\textit{Barycentric subdivisions, clusters and 
permutation enumeration} (in Greek),
Doctoral Dissertation, University of Athens, 2013.
%
\bibitem{SWG83}
L.W.~Shapiro, W.J.~Woan and S.~Getu,
\textit{Runs, slides and moments},
SIAM J. Algebraic Discrete Methods {\bf~4} (1983), 
459--466.
%
\bibitem{SW09}
J.~Shareshian and M.L.~Wachs,
\textit{Poset homology of Rees products and 
$q$-Eulerian polynomials},
Electron. J. Combin. {\bf~16} (2) (2009), 
Research Paper 20, 29pp (electronic).
%
\bibitem{SW10}
J.~Shareshian and M.L.~Wachs,
\textit{Eulerian quasisymmetric functions},
Adv. Math. {\bf~225} (2010), 2921--2966.
%
\bibitem{SW16a}
J.~Shareshian and M.L.~Wachs,
\textit{Chromatic quasisymmetric functions},
Adv. Math. {\bf~295} (2016), 497--551.
%
\bibitem{SW16b}
J.~Shareshian and M.L.~Wachs,
\textit{From poset topology to $q$-Eulerian polynomials
to Stanley's chromatic symmetric functions},
in \textit{The Mathematical Legacy of Richard P.~Stanley}
(P.~Hersh, T.~Lam, P.~Pylyavskyy and V.~Reiner, eds.),
pp.~301--321, Amer. Math. Society, Providence, RI, 2016.
%
\bibitem{SW17}
J.~Shareshian and M.L.~Wachs,
\textit{Gamma-positivity of variations of Eulerian
polynomials},
{\tt arXiv: 1702.06666v3}.
%
\bibitem{SZ12}
H.~Shin and J.~Zeng,
\textit{The symmetric and unimodal expansion of 
Eulerian polynomials via continued fractions},
European J. Combin. {\bf~33} (2012), 111--127.
%
\bibitem{SZ16}
H.~Shin and J.~Zeng,
\textit{Symmetric unimodal expansions of excedances
in colored permutations},
European J. Combin. {\bf~52} (2016), 174--196.
%
\bibitem{SU91}
R.~Simion and D.~Ullman,
\textit{On the structure of the lattice of noncrossing 
partitions},
Discrete Math. {\bf~98} (1991), 193--206.
%
\bibitem{Sta72}
R.P.~Stanley,
\textit{Ordered structures and partitions},
Mem. Amer. Math. Soc., no. 119, 1972.
%
\bibitem{Sta75}
R.P.~Stanley,
\textit{The Upper Bound Conjecture and 
Cohen--Macaulay rings},
Studies in Appl. Math. {\bf~54} (1975), 135--142.
%
\bibitem{Sta80}
R.P.~Stanley,
\textit{The number of faces of a simplicial convex 
polytope},
Adv. Math. {\bf~35} (1980), 236--238.
%
\bibitem{Sta82}
R.P.~Stanley, 
\textit{Some aspects of groups acting on finite 
posets}, 
J. Combin. Theory Series A {\bf~32} (1982), 132--161.
%
\bibitem{Sta89}
R.P.~Stanley,
\textit{Log-concave and unimodal sequences in algebra, 
combinatorics, and geometry},
in \textit{Graph Theory and its Applications: East 
and West}, Annals of the New York Academy of Sciences 
{\bf~576}, New York Acad. Sci., New York, 1989, 
pp.~500--535.
%
\bibitem{Sta92}
R.P.~Stanley,
\textit{Subdivisions and local $h$-vectors},
J. Amer. Math. Soc. {\bf~5} (1992), 805--851.
%
\bibitem{Sta94}
R.P.~Stanley,
\textit{Flag $f$-vectors and the $cd$--index},
Math. Z. {\bf~216} (1994), 483--499.
%
\bibitem{StaEC2}
R.P.~Stanley,
Enumerative Combinatorics, vol.~2,
Cambridge Studies in Advanced Mathematics {\bf~62},
Cambridge University Press, Cambridge, 1999.
%
\bibitem{StaCCA}
R.P.~Stanley,
Combinatorics and Commutative Algebra,
second edition, Birkh\"auser, Basel, 1996.
%
\bibitem{StaEC1}
R.P.~Stanley,
Enumerative Combinatorics, vol.~1,
Cambridge Studies in Advanced Mathematics {\bf~49},
Cambridge University Press, second edition, 
Cambridge, 2011.
%
\bibitem{Sta11}
A.~Stapledon,
\textit{Equivariant Ehrhart theory},
Adv. Math. {\bf~226} (2011), 3622--3654.
%
\bibitem{Stei92}
E.~Steingr\'imsson,
\textit{Permutation statistics of indexed and 
poset permutations},
Ph.D. thesis, MIT, 1992.
%
\bibitem{Stei94}
E.~Steingr\'imsson,
\textit{Permutation statistics of indexed permutations},
European J. Combin. {\bf~15} (1994), 187--205.
%
\bibitem{Ste92}
J.R.~Stembridge,
\textit{Eulerian numbers, tableaux, and the Betti 
numbers of a toric variety},
Discrete Math. {\bf~99} (1992), 307--320.
%
\bibitem{Ste94}
J.R.~Stembridge,
\textit{Some permutation representations of Weyl 
groups associated with the cohomology of toric 
varieties},
Adv. Math. {\bf~106} (1994), 244--301.
%
\bibitem{Ste97}
J.R.~Stembridge,
\textit{Enriched $P$-partitions},
Trans. Amer. Math. Soc. {\bf~349} (1997), 763--788.
%
\bibitem{Ste07}
J.R.~Stembridge,
\textit{Counterexamples to the poset conjectures of 
Neggers, Stanley, and Stembridge},
Trans. Amer. Math. Soc. {\bf~359} (2007), 1115--1128.
%
\bibitem{Ste08}
J.R.~Stembridge,
\textit{Coxeter cones and their $h$-vectors},
Adv. Math. {\bf~217} (2008), 1935--1961.
%
\bibitem{Str80}
V.~Strehl,
\textit{Symmetric Eulerian distributions for involutions},
S\'em. Lothar. Combin. {\bf~1} (1980), 1pp.
%
\bibitem{SuWa14}
H.~Sun and Y.~Wang,
\textit{A group action on derangements},
Electron. J. Combin. {\bf~21} (2014), 
Research Paper 1.67, 5pp (electronic).
%
\bibitem{Te17}
V.~Tewari,
personal communication, October 2017.
%
\bibitem{Vo10}
V.D.~Volodin,
\textit{Cubical realizations of flag nestohedra and a 
proof of Gal's conjecture for them},
Uspekhi Mat. Nauk. {\bf~65} (2010), 188--190.
%
\bibitem{Wa07}
M.~Wachs,
\textit{Poset Topology: Tools and Applications}, in
\emph{Geometric Combinatorics} (E.~Miller, V.~Reiner 
and B.~Sturmfels, eds.),
IAS/Park City Mathematics Series {\bf~13},
pp.~497--615, Amer. Math. Society, Providence, RI, 
2007.
%
\bibitem{Ze06}
A.~Zelevinsky,
\textit{Nested complexes and their polyhedral 
realizations},
Pure Appl. Math. Quart. {\bf~2} (2006), 655--671.
%
\bibitem{Zha95}
X.~Zhang,
\textit{On $q$-derangement polynomials},
in \textit{Combinatorics and Graph Theory '95}, 
Vol. 1 (Hefei),
World Sci. Publishing, River Edge, NJ, 1995, 
pp.~462--465.
%
\end{thebibliography}
\end{document}